\documentclass[12pt,leqno]{article}
\usepackage{makeidx}

\usepackage{amssymb,amsfonts,stmaryrd,mathrsfs}
\usepackage{amsmath,latexsym,amsthm}

\usepackage{manfnt}
\usepackage{verbatim}
\usepackage[applemac]{inputenc} 
\usepackage[cyr]{aeguill}
\usepackage[russian,francais,english]{babel}
\input cyracc.def

\usepackage{frcursive}

\usepackage{url}
\let\urlorig\url
\renewcommand{\url}[1]{%
  \begin{otherlanguage}{english}\urlorig{#1}\end{otherlanguage}%
}
\usepackage{graphicx}

\makeindex

\newtheorem{proposition}{Proposition}
\newtheorem{theorem}{Theorem}
\newtheorem{remark}{Remark}

\newcommand{\sJ}{{\mathscr J}}
\newcommand{\bJ}{{\bf J}}

\newcommand{\dif}{{\rm d}}
\newcommand{\Dif}{{\rm D}}

\newcommand{\vits}{{u}}
\newcommand{\vitsL}{{w}}
\newcommand{\Vits}{{U}}
\newcommand{\VitsL}{{W}}
\newcommand{\Rho}{{R}}
\newcommand{\vol}{{v}}
\newcommand{\Vol}{{V}}
\newcommand{\Y}{{Z}}

\newcommand{\bu}{{\bf u}}
\newcommand{\bv}{{\bf v}}

\newcommand{\bW}{{\bf W}}
\newcommand{\bP}{{\bf P}}
\newcommand{\bF}{{\bf F}}

\newcommand{\bPhi}{\boldsymbol{\Phi}}
\newcommand{\bPsi}{\boldsymbol{\Psi}}

\newcommand{\blambda}{\boldsymbol{\lambda}}

\newcommand{\bU}{{\bf U}}
\newcommand{\ubU}{\underline{\bf U}}

\newcommand{\ubu}{\underline{\bf u}}
\newcommand{\ubv}{\underline{\bf v}}
\newcommand{\ublambda}{\underline{\blambda}}

\newcommand{\uXi}{\underline{\Xi}}
\newcommand{\bUuz}{\bv_{0}}
\newcommand{\bUuzx}{\bv_{0,x}}

\newcommand{\uc}{{\underline{c}}}

\newcommand{\uk}{{\underline{k}}}

\newcommand{\bG}{{\bf G}}

\newcommand{\bV}{{\bf V}}
\newcommand{\bM}{{\bf M}}
\newcommand{\ubM}{\underline{\bf M}}

\newcommand{\uP}{\underline{P}}

\newcommand{\Real}{{\rm Re}}
\newcommand{\Imag}{{\rm Im}}

\newcommand{\R}{{\mathbb R}}
\newcommand{\Z}{{\mathbb Z}}

\newcommand{\ee}{{\rm e}}
\newcommand{\id}{{\rm Id}}

\newcommand{\Euler}{{\sf E}}
\newcommand{\Legendre}{{\sf L}}

\newcommand{\Hess}{{\sf Hess}}

\newcommand{\range}{{\sf R}}

\newcommand{\Lag}{\mathscr{L}}
\newcommand{\lag}{{\ell}}
\newcommand{\meane}{{\rm e}}
\newcommand{\meanE}{{\rm E}}
\newcommand{\meanvol}{{\rm v}}
\newcommand{\meanrho}{\varrho}

\newcommand{\meanvitsL}{{\rm w}}
\newcommand{\meanpress}{{\rm p}}
\newcommand{\press}{{p}}
\newcommand{\chem}{{g}}
\newcommand{\meanchem}{{\rm g}}

\renewcommand{\Cap}{\mathscr{K}}
\renewcommand{\cap}{{\kappa}}

\newcommand{\En}{\mathscr{E}}
\newcommand{\en}{\mbox{{\begin{cursive}{\emph{e}}\end{cursive}}}}
\newcommand{\Ec}{\mathscr{I}}
\newcommand{\Ham}{\mathscr{H}}
\newcommand{\Hamk}{\mathscr{H}^{(k)}}

\newcommand{\Impulse}{\mathscr{Q}}
\newcommand{\Impulseflux}{\mathscr{S}}

\newcommand{\Linvar}{{\bf A}}
\newcommand{\Linvarz}{{\bf A}^{\!(0)}}
\newcommand{\Linvaru}{{\bf A}^{\!(1)}}
\newcommand{\Linvard}{{\bf A}^{\!(2)}}

\newcommand{\Lin}{{\mathscr{A}}}
\newcommand{\Linz}{{\mathscr{A}}^{(0)}}
\newcommand{\Linu}{{\mathscr{A}}^{(1)}}
\newcommand{\Lind}{{\mathscr{A}}^{(2)}}
\newcommand{\Lint}{{\mathscr{A}}^{(3)}}

\newcommand{\uLinvar}{\underline{\bf A}}

\newcommand{\Op}{{\bf C}}

\renewcommand{\Dz}{{D}^{(0)}}
\newcommand{\Du}{{D}^{(1)}}
\newcommand{\Dd}{{D}^{(2)}}

\newcommand{\Evans}{{D}}

\renewcommand{\d}{\partial}
\newcommand{\cO}{\mathcal{O}}
\newcommand{\Su}{{\bf \Sigma}_1}
\newcommand{\Sd}{{\bf \Sigma}_2}
\newcommand{\St}{{\bf \Sigma}_3}
\newcommand{\su}{{\bf \sigma}_1}
\newcommand{\st}{{\bf \sigma}_3}

\begin{document}
\title{Slow modulations of periodic waves in Hamiltonian PDEs, with application to capillary fluids}
\author{S. Benzoni-Gavage\thanks{benzoni@math.univ-lyon1.fr}, P. Noble\thanks{noble@math.univ-lyon1.fr}, and L.M. Rodrigues\thanks{rodrigues@math.univ-lyon1.fr}\\
Universit\'e de Lyon,
CNRS UMR 5208,
Université Lyon 1,\\
Institut Camille Jordan,
43 bd 11 novembre 1918\\
F-69622 Villeurbanne cedex}
\maketitle

\begin{abstract}
Since its elaboration by Whitham, almost fifty years ago, modulation theory has been known to be closely related to the stability of periodic traveling waves. However, it is only recently that this relationship has been elucidated, and that fully nonlinear results have been obtained. These only concern dissipative systems though:  reaction-diffusion systems were first considered  by Doelman, Sandstede, Scheel, and Schneider [Mem. Amer. Math. Soc. 2009], and viscous systems of conservation laws have been addressed by Johnson, Noble, Rodrigues, and Zumbrun [preprint 2012]. Here, only nondissipative models are considered, and a most basic question is investigated, namely the  
expected link between the hyperbolicity of modulated equations and the spectral stability of periodic traveling waves to sideband perturbations. This is done first in an abstract Hamiltonian framework, which encompasses a 
number of dispersive models, in particular the well-known (generalized) Korteweg--de Vries equation, and the less known Euler--Korteweg system, in both Eulerian coordinates and Lagrangian coordinates.
The latter is itself an abstract framework for several models arising in water waves theory, superfluidity, and quantum hydrodynamics. As regards its application to compressible capillary fluids, attention is paid here to untangle the interplay between traveling waves/modulation equations in Eulerian coordinates and those in Lagrangian coordinates. In the most general setting, it is proved that the hyperbolicity of modulated equations is indeed necessary for the spectral stability of periodic traveling waves. This extends earlier results by
Serre [Comm. Partial Differential Equations 2005], Oh and Zumbrun [Arch. Ration. Mech. Anal. 2003], and Johnson, Zumbrun and Bronski [Phys. D 2010].
In addition, reduced necessary conditions are obtained in the small amplitude limit. Then numerical investigations are carried out for the modulated equations of the Euler--Korteweg system with two types of `pressure' laws, namely the quadratic law of shallow water equations, and the nonmonotone van der Waals pressure law. Both the evolutionarity and the hyperbolicity of the modulated equations are tested, and regions of modulational instability are thus exhibited. 
\end{abstract}




\section{Introduction}
This work is  motivated by the nonlinear waves analysis of the so-called Euler--Korteweg system, which arises in the modelling of capillary fluids~--~these comprise liquid-vapor mixtures
(for instance highly pressurized and hot water in nuclear reactors cooling system, in which the presence of vapor is actually dramatic), superfluids (Helium 
near absolute zero), or even regular fluids at sufficiently small scales (think of ripples on shallow water or other thin films).
In one space dimension, the most general form of the Euler--Korteweg system we consider is
\begin{equation}\label{eq:EKabs1d}
\left\{\begin{array}{l}\partial_t\rho +\partial_x (\rho \vits)\,=\,0\,,\\ [5pt]
\partial_t \vits + \vits\partial_x\vits \,+\,\partial_x(\Euler_\rho \En )\,=\,0\,,
\end{array}\right.
\end{equation}
 in Eulerian coordinates,
its counterpart in mass Lagrangian coordinates being
\begin{equation}\label{eq:EKabsLagb}
\left\{\begin{array}{l}\partial_t{\vol} \,=\,\partial_y{\vits}\,,\\ [5pt]
\partial_t {\vits} \,=\, \partial_y( \Euler_\vol \en)\,,
\end{array}\right.
\end{equation}
where $\rho$ is the fluid density, $\vol=1/\rho$ its specific volume, $\vits$ its velocity,  viewed either as a function of $(t,x)$ or as a function of $(t,y)$, where $y$ is the mass Lagrangian coordinate (by definition, $\dif y = \rho \dif x\,-\,\rho\,\vits\,\dif t$).
The energy density $\En$ in \eqref{eq:EKabs1d} and the specific energy $\en$ in \eqref{eq:EKabsLagb}
are related through $\En =\rho \en$, or equivalently $\en=\vol\En$, and are regarded as functions of $(\rho,\rho_x)$ and $(\vol,\vol_y)$ respectively. In those systems, the notation $\Euler$ stands for the Euler operator, that is
$$\Euler_\rho \En\,=\,\frac{\partial \En }{\partial \rho}\,-\,\Dif_{x}\left(\frac{\partial \En }{\partial \rho_{x}}\right)\,,\;\Euler_\vol \en\,=\,\frac{\partial \en }{\partial \vol}\,-\,\Dif_{y}\left(\frac{\partial \en }{\partial \vol_{y}}\right)\,,$$
where $\Dif_{x}$ and $\Dif_{y}$ mean total derivatives.
A widely used class of energies, dating back to Korteweg's theory of capillarity, read
\begin{equation}
\label{eq:EnK}
\En (\rho,\rho_x)=F(\rho)\,+\,\frac{1}{2}\Cap(\rho)\,(\rho_x)^2\,,
\end{equation}
or equivalently
\begin{equation}
\label{eq:enK}\en(\vol,\vol_y)\,=\,f(\vol)\,+\,\frac{1}{2}\,\cap(\vol)\,(\vol_y)^2\,,
\end{equation}
with the relationships $F=\rho f$, $\cap=\rho^5\Cap$, and with various choices of $\Cap$
(as far as \eqref{eq:EKabsLagb} is concerned, $\cap$ is often chosen to be constant, 
whereas quantum hydrodynamics equations correspond to $\rho\Cap=$~constant).

\paragraph{Abstract framework} Equations \eqref{eq:EKabs1d} and \eqref{eq:EKabsLagb} fall into the class of
abstract Hamiltonian systems of evolution PDEs of the form
\begin{equation}
\label{eq:absHam}
\partial_t\bU = \sJ (\Euler \Ham[\bU])\,,
\end{equation}
where $\sJ=\partial_x \bJ$ is a skew-symmetric differential operator, $\bJ$ being a symmetric, nonsingular matrix with constant coefficients, 
$\Ham$ is a functional involving first order derivatives only, and $\Euler$ denotes again the Euler operator:
$$ \Euler \Ham[\bU]_{\alpha}= \frac{\partial \Ham}{\partial U_\alpha}(\bU,\bU_{x}) \,-\,\Dif_x\left(\frac{\partial \Ham}{\partial U_{\alpha,x}}(\bU,\bU_{x})\right)\,,\quad \alpha\in \{1,\ldots,N\}$$
if $\bU$ has $N$ components.
Equation \eqref{eq:absHam} being space-invariant, it admits a conserved quantity called an \emph{impulse}, say
$\Impulse$  such that
$$\sJ \Euler \Impulse[\bU] = \partial_x \bU\,,$$
and since $\bJ$ is nonsingular we can explicitly take  for $\Impulse$ the quadratic quantity
$$\Impulse(\bU):= \tfrac{1}{2}\, \bU\cdot \bJ^{-1} \bU\,.$$
For further use, let us mention that associated with $\Impulse$ is the local conservation law
\begin{equation}
\label{eq:impulsecl}
\partial_t\Impulse(\bU) = \partial_x( \Impulseflux[\bU])\,
\end{equation}
satisfied along any smooth solution of \eqref{eq:absHam}, 
where
\begin{equation}
\label{eq:impulseflux}
 \Impulseflux[\bU]\,:=\,\bU \cdot \Euler \Ham[\bU]\,+\,\Legendre \Ham[\bU]\,,\quad\Legendre \Ham[\bU]\,:=\,
 U_{\alpha,x}\,\frac{\partial \Ham}{\partial  U_{\alpha,x}}(\bU,\bU_{x}) - \Ham(\bU,\bU_{x})\,.
 \end{equation}
 The dot $\cdot$ here above is just for the `canonical' inner product 
 $\bU\cdot\bV=U_\alpha V_\alpha$ in $\R^N$, and the letter $\Legendre$ stands for the Legendre transform (even though it is considered in the original variables
$(\bU,\bU_{x})$). In its definition we have used Einstein's convention of summation over repeated indices, and we shall do so repeatedly in the sequel.

\paragraph{Examples} The Euler--Korteweg system \eqref{eq:EKabs1d} fits into this framework
with $$ \bU= \left(\begin{array}{c} \rho \\ \vits\end{array}\right)\,,\quad
\Ham=\frac{1}{2}\rho \vits^2+\En(\rho,\rho_x)\,,\quad
\Impulse=\rho \vits\,,\quad \Impulseflux=\rho\, \Euler_\rho \En +\Legendre\En\,,$$
as well as the system in Lagrangian coordinates \eqref{eq:EKabsLagb},
with
$$\bU= \left(\begin{array}{c} \vol \\ \vits\end{array}\right)\,,\quad
\Ham=\frac{1}{2}\vits^2+\en(\vol,\vol_y)\,,\quad
 \Impulse 
= \vol \vits\,,\quad \Impulseflux=\vol\, \Euler_\vol \en +\Legendre\en\,.$$
An even simpler example is the scalar equation
\begin{equation}
 \label{eq:KdVabs}
\partial_tv=\partial_x(\Euler_v \Ham)\,,
\end{equation}
with $\Impulse = \tfrac{1}{2}\, v^2$, comprising the
generalized Korteweg--de Vries equation (gKdV),
\begin{equation}
 \label{eq:KdVgen}
\partial_t \vol+ \partial_x \press(\vol)= -\partial_x^3\vol\,,
\end{equation}
if we take
$$\Ham 
=f(v)\,+\,\frac{1}{2}v_x^2\,,\quad f '=-\press\,.$$
In this case, $\Impulseflux= -f-pv -vv_{xx}+\frac{1}{2}v_x^2$ -~and \eqref{eq:impulsecl} is the well-known conservation law for $v^2$ when $v$ is solution to 
(gKdV).
 
We will focus on the 'generic' situation  in which
the abstract system \eqref{eq:absHam} admits a  family of periodic traveling wave solutions parametrized, up to translations, by their speed and $N+1$ constants of integration. This rather large number of degrees of freedom makes the stability of periodic waves a difficult problem.
Furthermore, in the context of Hamiltonian PDEs we can only hope for neutral stability, namely that the
full spectrum of linearized equations lies on the imaginary axis. On the one hand, this kind of spectral stability makes the nonlinear stability analysis more delicate to tackle than in situations involving dissipation processes that are likely to push the spectrum~--~except for the null eigenvalue linked to translation invariance~--~into the left half plane. (Nonlinear results have recently been obtained concerning reaction-diffusion systems \cite{DSSS}, and viscous systems of conservation laws \cite{JNRZ}.)
On the other hand, the underlying variational framework can be of great help to prove nonlinear stability results. This has been done for a wide variety of solitary waves, see \cite{AnguloPava} and references therein. The literature on the stability of periodic waves is much more limited,
and nonlinear results are limited to stability under perturbations of the same period \cite{AnguloPava}.
Up to our knowledge, spectral stability has been proved for periodic wave solutions to the (standard) Korteweg--de Vries equation by using its integrability  to compute {\it explicitly} the spectrum 
\cite{BoDe}, and for small amplitude periodic wave solutions to a limited number of dispersive equations, comprising the nonlinear Schr\"odinger equation and the generalized Korteweg--de Vries equation 
\cite{GallayHaragus}. 
It is still a wide open problem for large amplitude periodic wave solutions to more general, dispersive PDEs.

As regards spectral stability, it is more tractable if sideband perturbations only are considered. There are indeed several analytical tools~--~like for instance Evans functions and associated winding numbers~--~ to look for possible spectrum in the vicinity of zero. A related topic is modulational stability, in connection with
Whitham's approach of slow modulations to periodic waves (the most famous modulational instability being the one shown by Benjamin and Feir for Stokes water waves). 

By ~--~at least formal~--~ asymptotic analysis, we can see that slow modulations to periodic waves are governed by a system of averaged equations. Its size is the dimension of the periodic orbits' manifold, here $N+2$. This will be made more precise in Section \ref{s:permod}, 
of which the main purpose is to extend to our abstract framework a result 
previously known for (gKdV) \cite{JohnsonZumbrunBronski} and for viscous systems of conservation laws
\cite{OhZumbrun03a,OhZumbrun03b,Serre}. This result gives a quantitative relationship between the 
sideband stability  of periodic traveling waves and the spectral properties of modulated equations. In particular, it shows that a necessary condition for the stability of a given periodic traveling wave is the (weak) hyperbolicity of modulated equations at the corresponding point in parameters' space. In other words, we give a rigorous proof that modulational stability is necessary for spectral stability, 
a result that is often taken for granted in the physics literature~--~ back to 1970, Whitham himself was indeed saying that
`the relation of the stability of the periodic wave with the type of the [modulated equations is] given in the previous papers' \cite{Whitham70}.
In addition, we investigate in some detail how the modulated equations degenerate in the small amplitude limit, and receive a reduced system coupled with a $2\times 2$ system for the wavenumber and the amplitude of the wave. This extends to our abstract framework observations that were made by Whitham \cite{Whitham}.
Unsurprisingly, when applied for instance to the Euler-Korteweg system, that limit gives as a reduced system the lower-order, Euler equations. Hence a stability condition for small amplitude periodic waves: the Euler system must be hyperbolic at the mean value of the wave. This rather natural condition does not seem to have been pointed out earlier. A reason is certainly that modulated equations have mostly been considered for \emph{scalar} models, like KdV or the Klein-Gordon equation in \cite{Whitham}, for which the hyperbolicity of the reduced model -~the inviscid Burgers equation for KdV, the wave equation for Klein-Gordon~-~ is trivial.

In Section \ref{s:EK}, we concentrate on the Euler--Korteweg system. We derive modulated equations in both kinds of coordinate systems,  namely the Eulerian one \eqref{eq:EKabs1d}, and the  
mass Lagrangian one \eqref{eq:EKabsLagb}. In addition, we point out a nice~--~if not surprising~---~relationship between them. To be precise we show that, away from vacuum, the modulated system for \eqref{eq:EKabsLagb} is equivalent to the modulated system for  \eqref{eq:EKabs1d} through a mass Lagrangian change of coordinates, hence the following commutative diagram.

\vspace{2mm}
\begin{center}
\fbox{
\begin{tabular}{ccccc}
& & \begin{minipage}{4cm} \centering mass Lagrangian \\ change of coordinates \end{minipage}  & & \\
& (1) & $\longrightarrow$ &  (2) & \\
\begin{minipage}{2cm}  \centering Whitham's \\ averaging \end{minipage}
& $\downarrow$ & & $\downarrow$ & \\
& $\langle \mbox{\rm 1} \rangle$ & $\longrightarrow$ &  $\langle\mbox{\rm 2} \rangle$ & \\
\end{tabular}
}
\end{center}
\vspace{2mm}
\noindent
Then in Section \ref{s:num} we go further into specialized cases and investigate in more details the periodic orbits' manifold when the energy is of Korteweg type \eqref{eq:EnK}. In this case, the nature of the phase portrait associated with profile equations for traveling waves highly depends on the monotonicity and convexity properties of pressure in terms of volume \cite{SBG-DIE}. We have considered two types of pressure laws, each one being motivated by a specific physical application. At first, we have taken a quadratic pressure with respect to density, 
which corresponds to shallow water equations~--~ the Korteweg part of the energy then taking into account
surface tension on the water surface.
A more involved case arises with van der Waals type pressure laws, typically corresponding to liquid-vapor mixtures (with capillarity effects).
In both cases~--~shallow water and van der Waals~--~, we have investigated numerically the hyperbolicity of Whitham's equations associated with various families of periodic waves.  We have found rather large regions of hyperbolicity. Failure of hyperbolicity occurs for waves `close to' unstable constant states/solitary waves or/and with sufficiently large periods.

\section{Periodic waves and modulated equations}\label{s:permod}

\subsection{General material}\label{ss:mat}
\paragraph{Traveling wave solutions}
A traveling wave solution to \eqref{eq:absHam} of speed $c$ is characterized by the profile equations
\begin{equation}
\label{eq:absHamprof}\Euler ( \Ham +c \Impulse)[\ubU] = \blambda\,,
\end{equation}
where the components $\lambda_\alpha$ of $\blambda\in\mathbb{R}^N$ are merely constants of integration. As was observed by Benjamin \cite{Benjamin}, the impulse flux $\Impulseflux$ is involved in the Hamiltonian associated with the Euler--Lagrange equations \eqref{eq:absHamprof}, which turns out to be 
$\Impulseflux +c \Impulse$. In other words, solutions of \eqref{eq:absHamprof} must satisfy
\begin{equation}
\label{eq:absHamprofint}\Impulseflux[\ubU] + c \Impulse(\ubU) = \mu
\end{equation}
for some new  (scalar) parameter $\mu$.
The profile $\ubU$ may be viewed as a \emph{stationary} solution of the abstract Hamiltonian system \eqref{eq:absHam} rewritten in a frame moving with speed $c$, that is
\begin{equation}
\label{eq:absHammov}
\partial_t\bU = \sJ (\Euler (\Ham+ c \Impulse)[\bU])\,.
\end{equation}
We may also note that in this moving frame the additional conservation law \eqref{eq:impulsecl} reads
$$\partial_t\Impulse(\bU) = \partial_x( (\Impulseflux +c \Impulse)[\bU])\,,$$
which obviously admits the profile $\ubU$ as a stationary solution, according to \eqref{eq:absHamprofint}.

\paragraph{Linearized problem}
In order to investigate the stability of $\ubU$ as a solution of \eqref{eq:absHammov}, we start by linearizing this system about 
$\ubU$, which yields
\begin{equation}
\label{eq:absHamlin}
\partial_t\bU = \Lin \bU\,,\; \Lin:=\sJ \Linvar\,,\;\Linvar:= \Hess (\Ham+ c \Impulse)[\ubU]\,.
\end{equation}
The Hessians here above are given by
$$\Hess \Impulse[\ubU]=\bJ^{-1}$$
whatever $\ubU$, and
$$(\Hess \Ham[\ubU] \bU)_\alpha=\, \frac{\partial^2 \Ham}{\partial U_\alpha \partial U_\beta} U_\beta \,+\,
\frac{\partial^2 \Ham}{\partial U_\alpha \partial U_{\beta,x}} U_{\beta,x}
\,-\,\Dif_x\left(\frac{\partial^2 \Ham}{\partial U_{\alpha,x}\partial U_\beta} U_\beta\,+\,\frac{\partial^2 \Ham}{\partial U_{\alpha,x}\partial U_{\beta,x}} U_{\beta,x}\right)\,,$$
where all second order derivatives of $\Ham$ are evaluated at $(\ubU,\ubU_{x})$.
By differentiating the profile equations  \eqref{eq:absHamprof} with respect to $x$ we observe as usual with translation-invariant problems that 
$\Linvar \ubU_x=0$, hence also $\Lin \ubU_x=0$.
Furthermore, we have several parameters, namely the speed $c$, and the constants of integration $\lambda_\alpha$, $\mu$. 
If they are all independent, as is typically the case with periodic traveling waves, we receive remarkable identities 
by differentiating the profile equations  \eqref{eq:absHamprof} with respect to those parameters. This yields
$$\begin{array}{ll}
\Linvar \ubU_c = - \Euler \Impulse [\ubU]\,,& \mbox{hence} \quad\Lin \ubU_c = -\ubU_x\,,\\ [5pt]
\Linvar \ubU_\mu = 0\,,& \mbox{hence} \quad\Lin \ubU_\mu = 0\,,\\ [5pt]
\Linvar \ubU_{\lambda_\alpha} = \ee_{\alpha}\,,& \mbox{hence} \quad\Lin \ubU_{\lambda_\alpha} = 0\,.
\end{array}$$
(Here above we have denoted by $\ee_{\alpha}$ the $\alpha$-th vector of the `canonical' basis in the $\bU$-space, and the subscripts $c$, $\mu$, $\lambda_\alpha$ stand for partial derivatives with respect to those parameters.)
When $\ubU$ is periodic, say of period $\Xi$, this period of course depends on the parameters $c$, $\lambda_\alpha$, $\mu$, and thus the derivatives $\ubU_c$, $\ubU_\mu$, $\ubU_{\lambda_\alpha}$ have no reason to be periodic. However, as pointed out in \cite{BronskiJohnsonKapitula}, we can set up the generalized kernel of  $\Lin$ in the space of $\Xi$-periodic functions with linear combinations of those derivatives.
As a matter of fact, $\ubU_x$ and $\bPhi_\alpha:= \Xi_{\lambda_\alpha} \ubU_\mu -\Xi_{\mu}\ubU_{\lambda_\alpha}$ are all in the $\Xi$-periodic kernel of $\Lin$, while
$$\bPsi_\alpha:= \{\Xi, \langle \ubU\rangle\}_{{\lambda_\alpha},c} \,\ubU_\mu \,+\,\{\Xi, \langle \ubU\rangle\}_{c,\mu} \,\ubU_{\lambda_\alpha}\,+\,\{\Xi, \langle \ubU\rangle\}_{\mu,{\lambda_\alpha}} \,\ubU_c$$
is also $\Xi$-periodic, and such that 
$$\Lin \bPsi_\alpha = - \{\Xi, \langle \ubU\rangle\}_{\mu,{\lambda_\alpha}}\, \ubU_x \in \ker \Lin\,.$$
Here above we have used the same convenient notation as in \cite{BronskiJohnsonKapitula} 
$$\{f,g\}_{a,b}= f_a g_g-f_b g_a\,,$$
and the brackets $\langle \cdot \rangle$ stand 
for mean values on a period. The periodicity of the linear combinations mentioned above follows from the identity
$$\ubU_a(\Xi)-\ubU_a(0)=\Xi_a\,\ubU_x(0)\,,$$
which holds true whatever the parameter $a$.
Note that we find in this way at most $N+2$ independent elements of the generalized kernel of $\Lin$ in the space of $\Xi$-periodic functions if $N$ is the dimension of the $\bU$-space,
even though we have used $N+3$ candidates.

\paragraph{Modulated equations}
Following Whitham's `two-timing method' \cite{Whitham70}, we search for solutions of \eqref{eq:absHam} having an asymptotic expansion of the form
$$\bU(t,x)=\bU^0(\varepsilon t, \varepsilon x, \phi(\varepsilon t, \varepsilon x)/\varepsilon)\,+\,
\varepsilon\,\bU^1(\varepsilon t, \varepsilon x, \phi(\varepsilon t, \varepsilon x)/\varepsilon,\varepsilon)\,+\,o(\varepsilon)\,,$$
where $\bU^0$ and $\bU^1$ are $1$-periodic in their third variable (subsequently denoted by $\theta$).
Denoting by $X=\varepsilon x$, $T=\varepsilon t$ the rescaled space variable and time respectively, 
$$k:=\phi_X\,,\;\omega:=\phi_T\,,\;c:=-\omega/k\,,$$
the very existence of a twice differentiable phase $\phi$ requires that 
\begin{equation}
\label{eq:Whithamcomp}
\partial_T k + \partial_X(ck)=0\,.
\end{equation}
Plugging the asymptotic expansion in \eqref{eq:absHam} and using that 
$\partial_t = \varepsilon \partial_T + \omega \partial_\theta$, 
$\partial_x= \varepsilon \partial_X+k \partial_\theta$, 
 we formally receive the following equations
\begin{equation}
\label{eq:Whitham0} \partial_\theta(\bG^0 \,+ \,c \,\bV^0)=0\,,
\end{equation}
\begin{equation}
\label{eq:Whitham1}
\partial_T \bU^0= \bJ \partial_X \bG^0\,+\,\bJ \,k \partial_\theta(\bG^1\,+\,c\,\bV^1)
\end{equation}
where the components $G^{0,1}_\alpha$, $V^{0,1}_\alpha$ of $\bG^{0,1}$ and $\bV^{0,1}$ are given by
$$ G^0_\alpha: = \frac{\partial \Ham}{\partial U_\alpha} \,-\,k\Dif_\theta\left(\frac{\partial \Ham}{\partial U_{\alpha,x}}\right)\,,$$
$$ V^0_\alpha: = \frac{\partial \Impulse}{\partial U_\alpha}\,,\quad \mbox{or equivalently } \bV^0= \bJ^{-1} \bU^0\,,$$
$$ G^1_\alpha: =\begin{array}[t]{l} \displaystyle\frac{\partial^2 \Ham}{\partial U_\alpha \partial U_\beta}\, U^1_\beta \,+\,
\frac{\partial^2 \Ham}{\partial U_\alpha \partial U_{\beta,x}} \, k\partial_\theta U^1_{\beta}
\,-\,k\Dif_\theta\left(\frac{\partial^2 \Ham}{\partial U_{\alpha,x}\partial U_\beta}\, U^1_\beta\,+\,\frac{\partial^2 \Ham}{\partial U_{\alpha,x}\partial U_{\beta,x}}\, k\partial_\theta U^1_{\beta}\right)\\ [15pt]
\displaystyle+\,\frac{\partial^2 \Ham}{\partial U_\alpha \partial U_{\beta,x}} \, \partial_X U^0_{\beta}\,-\,
\Dif_X\left(\frac{\partial \Ham}{\partial U_{\alpha,x}}\right)\,-\,k\Dif_\theta \left(\frac{\partial^2 \Ham}{\partial U_{\alpha,x} \partial U_{\beta,x}} \, \partial_X U^0_{\beta}\right)\,,
\end{array}$$
$$ V^1_\alpha: = \frac{\partial^2 \Impulse}{\partial U_\alpha U_\beta}\,U^1_\beta\,,\quad \mbox{or equivalently } \bV^1= \bJ^{-1} \bU^1\,,$$
all derivatives of $\Ham$ being evaluated at $(\bU^0,k\partial_\theta \bU^0)$, and those of $\Impulse$ at $\bU^0$.
The zeroth order equation in \eqref{eq:Whitham0} yields $(\bG^0 \,+ \,c \,\bV^0)=$~constant, which just amounts to the traveling wave profile equation \eqref{eq:absHamprof}. More precisely, $(\bG^0 \,+ \,c \,\bV^0)=\blambda$ requires that, at fixed $(T,X)$, $\ubU(x):=\bU^0(T,X,kx)$ solves \eqref{eq:absHamprof}.
The first order equation in \eqref{eq:Whitham1} yields the averaged equation over $\theta\in [0,1]$
\begin{equation}
\label{eq:Whitham1av}
\partial_T \langle\bU^0\rangle= \bJ \partial_X \langle\bG^0\rangle\,.
\end{equation}
Together with the compatibility equation \eqref{eq:Whithamcomp}, this is the main set of modulated equations. There is an additional one associated with translation invariance of the original system.
The fastest way to obtain it is to use 
the conservation law in \eqref{eq:impulsecl}, which yields
\begin{equation}
\label{eq:Whitham1bav}
\partial_T \langle Q^0\rangle=  \partial_X \langle S^0\rangle\,,
\end{equation}
where 
$$Q^0:=\tfrac{1}{2}\,\bV^0\cdot \bU^0\,,\;
S^0\,:=\,\bU^0 \cdot \bG^0\,- \Ham(\bU^0,k\partial_\theta \bU^0)\,+\,\frac{\partial \Ham}{\partial  U_{\alpha,x}}\,k\partial_\theta 
 U^0_{\alpha}\,,$$
the derivative of $\Ham$ being again evaluated at $(\bU^0,k\partial_\theta \bU^0)$. 
The other way is the main reason why we have written down $\bG^1$ in details. As a matter of fact,
Equation \eqref{eq:Whitham1bav} can be obtained by averaging the inner product of
\eqref{eq:Whitham1} with $\bV^0$. 
The first term in this operation is simply
$$\langle \bV^0\cdot \partial_T\bU^0\rangle\,=\,\partial_T \langle Q^0\rangle\,.$$
On the right-hand side of $\langle\bV^0\cdot\eqref{eq:Whitham1bav}\rangle$, we first have
$$\langle \bV^0\cdot \bJ \partial_X\bG^0\rangle\,=\,\langle \bU^0\cdot  \partial_X\bG^0\rangle\,.
$$
In order to deal with the other term, we can split $\bG^1+c\bV^1$ into
$$\bG^1+c\bV^1= \Linvar(k\partial_\theta)\,\bU^1\,+\,\Op(k\partial_\theta)\,\partial_X \bU^0$$
where $\Linvar(k\partial_\theta)$ and $\Op(k\partial_\theta)$ are both linear differential operators with coefficients that are functions of 
$(\bU^0,k\partial_\theta \bU^0)$, and moreover Equation \eqref{eq:Whitham0} means that
$$\Linvar(k\partial_\theta)\,\partial_\theta \bU^0\,=\,0\,.$$
This implies that 
$$\langle \bV^0\cdot \bJ k \partial_\theta(\Linvar(k\partial_\theta)\,\bU^1)\rangle\,=\,-\,\langle \bU^1\cdot \Linvar(k\partial_\theta)\,k\partial_\theta \bU^0 \rangle\,=\,0\,,$$
hence
$$\langle \bV^0\cdot \bJ k \partial_\theta(\bG^1+c\bV^1)\rangle\,=\,
-\,\langle k\partial_\theta \bU^0\cdot \Op(k\partial_\theta)\,\partial_X \bU^0\rangle\,.$$
Therefore, it remains to check that 
$$\langle \bU^0\cdot  \partial_X\bG^0\rangle\,-\,\langle k\partial_\theta \bU^0\cdot \Op(k\partial_\theta)\,\partial_X \bU^0\rangle\,=\,
\partial_X \langle S^0\rangle\,,$$
or equivalently,
$$\langle  \bG^0\cdot \partial_X\bU^0 \rangle\,+\,\partial_X\langle k\partial_\theta \bU^0\cdot \nabla_{\bU_{x}} \Ham \,-\,\Ham\rangle
\,+\,
\langle k\partial_\theta \bU^0\cdot \Op(k\partial_\theta)\,\partial_X \bU^0\rangle\,=\,0\,.$$
Now, recalling the definition of $\bG^0$ and making an integration by parts we get
$$\langle  \bG^0\cdot \partial_X\bU^0 \rangle\,=\,
\langle(\partial_X \bU^0)\cdot \nabla_{\bU} \Ham\,+\,(\partial_\theta\partial_X \bU^0)\cdot \nabla_{\bU_{x}} \Ham\rangle\,=\,
\partial_X\langle \Ham \rangle\,-\, \langle(\partial_Xk) (\partial_\theta\bU^0)\cdot  \nabla_{\bU_{x}} \Ham\rangle$$
$$=\partial_X\langle \Ham\,-\,k\partial_\theta \bU^0\cdot \nabla_{\bU_{x}} \Ham\rangle\,+\,
k\,\partial_X\langle (\partial_\theta \bU^0)\cdot \nabla_{\bU_{x}} \Ham\rangle\,.$$
Finally, recalling the definition of $\Op$ and making an integration by parts, we have
$$\langle \partial_\theta \bU^0\cdot \Op(k\partial_\theta)\,\partial_X \bU^0\rangle\,=\,$$
$$\left\langle (\partial_\theta U^0_\alpha)\,\frac{\partial^2 \Ham}{\partial U_\alpha \partial U_{\beta,x}} \, \partial_X U^0_{\beta}\,-\,
(\partial_\theta U^0_\alpha)\,\Dif_X\left(\frac{\partial \Ham}{\partial U_{\alpha,x}}\right)\,+\,k(\partial_\theta^2 U^0_\alpha)\,\frac{\partial^2 \Ham}{\partial U_{\alpha,x} \partial U_{\beta,x}} \, \partial_X U^0_{\beta}\right\rangle$$
$$=\,\left\langle \Dif_\theta\left(\frac{\partial \Ham}{\partial U_{\beta,x}} \, \right)\,\partial_X U^0_{\beta}\,-\,
(\partial_\theta U^0_\alpha)\,\Dif_X\left(\frac{\partial \Ham}{\partial U_{\alpha,x}}\right)\right\rangle
$$
$$
=-\,\partial_X\langle (\partial_\theta \bU^0)\cdot \nabla_{\bU_{x}} \Ham\rangle\,
$$
thanks to another integration by parts.

\paragraph{Low frequency analysis}
We assume that $\ubU$ is a given periodic traveling wave profile of period $\Xi=1/\uk$, and that the set of nearby periodic traveling wave profiles is a $N+2$ dimensional manifold if $N$ the dimension of the $\bU$-space. As explained above, natural parameters for this manifold are the speed $c$ of waves, the constants of integration $\lambda_\alpha$ for $\alpha\in \{1,\ldots,N\}$ as well as $\mu$ showing up in the profile equations \eqref{eq:absHamprof}-\eqref{eq:absHamprofint}.
In fact, we shall prefer a parametrization that is more natural in connection with Whitham's modulated equations,
and assume that the manifold of periodic traveling wave profiles is parametrized by their wave number $k$ (inverse of period),
and the mean values 
\begin{equation}
\label{eq:defmean}
\bM=\langle \bU\rangle\,,\quad P=\langle \Impulse (\bU)\rangle\,.
\end{equation}
For a discussion of this assumption, see 
Appendix \ref{Whitham_param}.
In addition, we rescale all periodic profiles into $1$-periodic ones, in such a way that $k$ appears explicitly in the profile equations. For simplicity we still denote by $\ubU$ the profile now viewed as a function of $\theta=\uk x$, and similarly 
any nearby profile $\bU$ is viewed as a function of  $\theta=k x$. The latter must therefore satisfy
\begin{equation}
\label{eq:absHamprofk}
\partial_\theta\left(\frac{\partial \Ham}{\partial U_\alpha}(\bU,k\partial_\theta\bU)-k\partial_\theta\left(\frac{\partial \Ham}{\partial U_{\alpha,x}}(\bU,k\partial_\theta\bU)\right)+c\,\frac{\partial \Impulse}{\partial U_\alpha}(\bU)\right)=0\,,\: \alpha\in\{1,\ldots,N\}\,.
\end{equation}
Our previous assumption means that for all $(k,\bM,P)$ close to $(\uk,\ubM,\uP)$ there is a unique speed 
$c=c(k,\bM,P)$ and a unique profile up to translations
$\bU=\bU(\theta;k,\bM,P)$ that is $1$-periodic, close\footnote{We say that two $1$-periodic functions are close to each other if their distance with respect to the sup norm up to translations is small.} to $\ubU$, and solution to
\eqref{eq:defmean}-\eqref{eq:absHamprofk}. 
For simplicity again, we just denote by  $\Linvar$ the differential operator $\Linvar(k\partial_\theta)$ considered in the previous paragraph with $\bU^0$ replaced by $\bU$ (which amounts to $\Hess (\Ham + c \Impulse)[\bU]$ where $\partial_x$ is replaced by $k\partial_\theta$), and by $\Lin$ the operator
$$\Lin(k\partial_\theta)= k\partial_\theta \bJ \Linvar(k\partial_\theta)\,.$$
More explicitly, we have
$$(\Linvar \bV)_\alpha\,=\,\frac{\partial^2 \Ham}{\partial U_\alpha \partial U_{\beta}}\,V_\beta \,+\,
\frac{\partial^2 \Ham}{\partial U_\alpha \partial U_{\beta,x}} \, k\partial_\theta V_{\beta}
\,-\,k\Dif_\theta\left(\frac{\partial^2 \Ham}{\partial U_{\alpha,x}\partial U_\beta}\,V_\beta\,+\,\frac{\partial^2 \Ham}{\partial U_{\alpha,x}\partial U_{\beta,x}}\, k\partial_\theta V_{\beta}\right).$$
Equation \eqref{eq:absHamprofk} equivalently reads 
$\Linvar \partial_\theta\bU=0$. 
Besides, this equation implies that $$\partial_\theta(\Linvar \bU_a\,+\,c_a\,\bJ^{-1}\bU) = 0\,,\;\mbox{hence} \quad
\Lin \bU_a\,=\,-\,c_a\,k\partial_\theta \bU\,,$$
where the subscript $a$ stands for a partial derivative  with respect to any parameter among 
$M_\alpha$ and $P$. According to our assumption on the parametrization of periodic profiles, this makes at least $N+1$ independent elements of the generalized kernel of $\Lin$ in the space of $1$-periodic functions, and in fact $N+2$ counting $\partial_\theta \bU$.  
Let us mention straightforward identities for the formal adjoint 
$$\Lin^*=-\Linvar \bJ k \partial_\theta\,.$$
We clearly have indeed
$$\Lin^*\ee_{\alpha}=0\,,\quad \Lin^* \bJ^{-1} \bU=0\,.$$
(Recall that $\ee_{\alpha}$ is just a constant vector in the $\bU$-space.)
Now, we are not only interested in the spectrum of $\Lin$ in the space of $1$-periodic functions but in the whole space of bounded functions. This is why
we introduce the Bloch operators
$$\Lin^\nu:=\Lin(k(\partial_\theta+i\nu))\,,\;\nu\in\R/2\pi\Z\,,$$
of which the spectra in the space of $1$-periodic functions give the one of $\Lin$ on $L^\infty$:
$$\sigma(\Lin)\,=\,\bigcup_{\nu\in\R/2\pi\Z}\sigma (\Lin^\nu)\,.$$

\subsection{Modulational stability \emph{vs} spectral stability}\label{ss:modspec}

Given the material introduced in \S \ref{ss:mat}, we can show the following.
\begin{theorem}\label{thm:absstab} Let us assume that $\ubU$ is the profile of a periodic traveling wave solution to
\eqref{eq:absHam} of period $1/\uk$ and speed $\uc$, and that the set of nearby periodic traveling wave profiles $\bU$ of speed $c$ close to $\uc$ is a $N+2$ dimensional manifold parametrized by $(k,\bM=\langle \bU \rangle,P=\langle Q \rangle)$, where $1/k$ is their period and
$$Q:=\tfrac{1}{2}\,\bU \cdot \bJ^{-1}\bU\,.$$
Let us consider the modulated system
\begin{equation}
\label{eq:Whithamsyst}
\left\{\begin{array}{l}\partial_T k + \partial_X(ck)=0\,,\\ [5pt]
\displaystyle
\partial_T \langle \bU \rangle\,=\,\bJ\,\partial_X\langle \bG\rangle\,,\\ [5pt]
\partial_T \langle Q \rangle\,=\,\partial_X\langle S \rangle\,,
\end{array}\right.
\end{equation}
where
$$G_\alpha:=\frac{\partial \Ham}{\partial U_\alpha}(\bU,k\partial_\theta\bU) \,-\,k\partial_\theta\left(\frac{\partial \Ham}{\partial U_{\alpha,x}}(\bU,k\partial_\theta\bU)\right)\,,$$
$$S\,:=\,\bU \cdot \bG\,- \Ham(\bU,k\partial_\theta \bU)\,+\,(k\partial_\theta 
 U_{\alpha})\,\frac{\partial \Ham}{\partial  U_{\alpha,x}}(\bU,k\partial_\theta \bU)\,.$$
 We also assume that the generalized kernel of $\Lin$ in the space of $1$-periodic functions is of dimension $N+2$.
Then a necessary condition for $\ubU$ to be stable is that the system \eqref{eq:Whithamsyst} be `weakly hyperbolic' at 
$(\uk,\ubM,\uP)$, in the sense that all its characteristic speeds must be real.
\end{theorem}

\begin{proof}
It is based on a perturbation calculation, which relates the matrix of \eqref{eq:Whithamsyst} at 
$(k,\bM,P)$ to the one of $\Lin^\nu$ restricted by spectral projection to a $(N+2)$ dimensional invariant subspace.
We first introduce the expansion
$$\Lin^\nu\,=\,\Linz\,+\,i\,k\nu\,\Linu\,-\,k^2\nu^2\,\Lind\,-\,i\,k^3\nu^3\,\Lint\,,$$
where
$\Linz=\Lin^0=\bJ k\partial_\theta \Linvarz$ 
is just $\Lin$ viewed as an operator acting on $1$-periodic functions,
as well as $\Linvarz$ is just $\Linvar$ acting on $1$-periodic functions,
$$\Linu:=\bJ (\Linvarz + k \partial_\theta \Linvaru)\,,\;\Lind:=\bJ (\Linvaru + k \partial_\theta \Linvard)\,,\;
\Lint:=\bJ \Linvard\,,$$
$$\begin{array}{l} \displaystyle( \Linvaru \bV)_\alpha:= \frac{\partial^2 \Ham}{\partial U_\alpha \partial U_{\beta,x}}\,V_\beta \,-\,
\frac{\partial^2 \Ham}{\partial U_{\alpha,x} \partial U_{\beta}} \,V_{\beta}
\,-\,k\Dif_\theta\left(\frac{\partial^2 \Ham}{\partial U_{\alpha,x}\partial U_{\beta,x}}\,V_\beta\right)\,-\,\frac{\partial^2 \Ham}{\partial U_{\alpha,x}\partial U_{\beta,x}}\, k\partial_\theta V_{\beta}\,,\\ [15pt]
\displaystyle ( \Linvard \bV)_\alpha:=-\,\frac{\partial^2 \Ham}{\partial U_{\alpha,x}\partial U_{\beta,x}}\, V_{\beta}\,.\end{array}$$
Differentiating \eqref{eq:absHamprofk} with respect to $k$ we find that
\begin{equation}\label{eq:Uk}
\Linz\bU_{k} \,+\,\Linu \partial_\theta \bU \,=\,-\,c_{k}\,k\partial_\theta \bU\,.
\end{equation}
Indeed, we find at once that for all $\alpha\in\{1,\ldots,N\}$,
$$\partial_\theta\left((\Linvarz \bU_{k}\,+\,c_{k}\,\bJ^{-1}\bU)_\alpha +\,
\frac{\partial^2 \Ham}{\partial U_\alpha \partial U_{\beta,x}}\,\partial_\theta U_\beta \,-\,\Dif_\theta\left(\frac{\partial \Ham}{\partial U_{\alpha,x}}\right) \,-\,
k\Dif_\theta\left(\frac{\partial^2 \Ham}{\partial U_{\alpha,x}\partial U_{\beta,x}}\, \partial_\alpha U_{\beta}\right)\right) = 0\,,
$$
hence
$$\Linz\bU_{k}\,+\,c_{k}\,k\partial_\theta \bU\,+\, \bJ k\partial_\theta \Linvaru\partial_\theta\bU\,=\,0\,,$$
which is equivalent to \eqref{eq:Uk} since $\Linvar\partial_\theta\bU=0$.

As already mentioned,
$$\bPhi_0^0:=\partial_\theta \bU\,,\;\bPhi_\alpha^0:=\bU_{M_\alpha}\,,\;\bPhi_{N+1}^0:=\bU_P$$
all belong to~--~and span~--~ the generalized kernel of $\Lin$, while 
$$\bPsi_\alpha^0:=\ee_\alpha\,,
\;\alpha\in \{1,\ldots,N\}\,,
\;\bPsi_{N+1}^0:=\bJ^{-1}\bU$$
belong to~--~and span~--~ the kernel of $\Lin^*$. By definition of the mean values $M_\alpha$ and $P$
in \eqref{eq:defmean},
those functions are such that
$$\langle \bPsi_\alpha^0\cdot \bPhi_\beta^0\rangle=\delta_{\alpha,\beta}\,.$$
Therefore, we can add in a function $\bPsi_0^0$ such that
$(\bPsi_0^0,\bPsi_1^0,\ldots,\bPsi_N^0,\bPsi_{N+1}^0)$ be dual to the basis
$(\bPhi_0^0,\bPhi_1^0,\ldots,\bPhi_N^0,\bPhi_{N+1}^0)$ of the generalized kernel of $\Lin$,  
and span the generalized kernel of $\Lin^*$.
Recall that by our main assumption, $0$ is an isolated eigenvalue of $\Linz$ of algebraic multiplicity equal to $N+2$.
Therefore, 
since our structural assumptions ensure that $\Lin^\nu$ is 
a relatively compact perturbation of $\Linz$ depending analytically on $\nu$
(see Appendix \ref{compact_resolvent}), 
there exist an analytic mapping $\nu \mapsto \Pi(\nu)$ where $\Pi(\nu)$ is a spectral projector for $\Lin^\nu$ 
of finite rank $N+2$, and coincides with  the orthogonal projector onto
$\mbox{\rm span}(\bPhi_0^0,\bPhi_1^0,\ldots,\bPhi_N^0,\bPhi_{N+1}^0)$ at $\nu=0$.  By Kato's perturbation method \cite[pp.~99-100]{Kato}, we thus construct dual bases 
$(\bPhi_0^\nu,\bPhi_1^\nu,\ldots,\bPhi_N^\nu,\bPhi_{N+1}^\nu)$ and 
$(\bPsi_0^\nu,\bPsi_1^\nu,\ldots,\bPsi_N^\nu,\bPsi_{N+1}^\nu)$ of, respectively,
$\range( \Pi(\nu))$ and $\range( \Pi(\nu)^*)$,
which depend analytically on $\nu$ in a real neighborhood of zero. The part of $\Lin^\nu$ on the finite dimensional subspace $\range( \Pi(\nu))$ is determined by the matrix
\begin{equation*}
D^\nu:=\left(\langle \bPsi_\alpha^\nu \cdot \Lin^\nu \bPhi_\beta^\nu \rangle\right)_{0\leq\alpha,\beta\leq N+1}\,.
\end{equation*}
Similarly as $\Lin^\nu$, this matrix has an expansion
$$D^\nu= \Dz\,+\,i\,k\nu\,\Du\,-\,k^2\nu^2\,\Dd
+o(\nu^2)\,.$$
By using that
\begin{equation*}
\Linz \bPhi_0^0\,=\,0\,,\quad\Linz \bPhi_\beta^0\,=\,-\,c_{M_\beta}\,k\,\bPhi_0^0\,,\;1\leq \beta\leq N \,,\quad
\Linz \bPhi_{N+1}^0\,=\,-\,c_{P}\,k\,\bPhi_0^0\,,
\end{equation*}
\begin{equation*}
(\Linz)^* \bPsi_\alpha^0 = 0\,,1\leq \alpha\leq N \,,\quad (\Linz)^* \bPsi_{N+1}^0=0\,,
\end{equation*}
and
$\langle \bPsi_0^0, \bPhi_0^0\rangle =1$, we get that
$$\Dz=\left(\begin{array}{c|cccc}
0 & -kc_{M_1} & \ldots & -kc_{M_N} & -kc_P\\ \hline
\vdots & & 0 & & \\
0 & && &\end{array}\right)\,.$$
Using in addition Eq.~\eqref{eq:Uk}, which equivalently reads
\begin{equation}
\label{eq:Ukb}
\Linz\bU_{k} \,+\,\Linu \bPhi_0^0 \,=\,-\,c_{k}\,k\,\bPhi_0^0\,,
\end{equation}
we see that for all $\alpha\in\{1,\ldots,N\}$,
$$\langle\bPsi_{\alpha}^0\cdot \Linu \bPhi_0^0\rangle=0\,,$$
hence 
$$\Du=\left(\begin{array}{c|ccc}
* & * & \ldots & *\\ \hline
0 & & & \\
\vdots & & * & \\
0 & & &\end{array}\right)\,.$$
(Here above, we have also used that
$\partial_\nu \bPsi_\alpha^\nu \cdot \Linz \bPhi_0^0=0$ for all $\alpha$, and 
$\bPsi_\alpha^0 \cdot \Linz \partial_\nu \bPhi^\nu=0$ for all $\alpha\geq 1$.)
Moreover, we claim that the upper-left entry of $\Du$ is 
$$\langle \bPsi_0^0 \cdot \Linu \bPhi_0^0\rangle\,=\,
-\,k\,c_{k}\,.$$
Indeed, this equality comes from \eqref{eq:Ukb}, and the only other nontrivial term in the upper-left entry of $\Du$ is
$$\frac{1}{ik}\,\langle \bPsi_0^0 \cdot \Linz \partial_\nu(\bPhi_0^\nu)_{|\nu=0}\rangle\,.$$
Since $(\Linz)^*(\bPsi_0^0)$ belongs to $\mbox{\rm span}(\bPsi_1^0,\ldots,\bPsi_N^0,\bPsi_{N+1}^0)$,
that term will cancel out provided that for all $\alpha\in\{1,\ldots,N+1\}$,
\begin{equation}\label{eq:cancel}
\langle \bPsi_\alpha^0 \cdot \partial_\nu(\bPhi_0^\nu)_{|\nu=0}\rangle\,=\,0\,.
\end{equation}
This we can arrange, up to a harmless modification of $\Phi^\nu_0$. Let us explain how.
Using that $\Lin^\nu\Phi_0^0=0$, we see by expanding
$$\Pi(\nu) \,\Lin^\nu \Phi_0^\nu = \Lin^\nu \Phi_0^\nu$$
that
$$\Pi(0)\,(\Linu\,\Phi_0^0\,+\,\frac{1}{ik}\;\Linz\, \partial_\nu(\bPhi_0^\nu)_{|\nu=0})\,=\,
\Linu\,\Phi_0^0\,+\,\frac{1}{ik}\;\Linz\, \partial_\nu(\bPhi_0^\nu)_{|\nu=0}\,,$$
or, using again \eqref{eq:Ukb} and that $\Phi_0^0\in \range(\Pi(0))$,
$$\Pi(0)\,(-\Linz\bU_{k}\,+\,\frac{1}{ik}\;\Linz\, \partial_\nu(\bPhi_0^\nu)_{|\nu=0})\,=\,
-\Linz\bU_{k}\,+\,\frac{1}{ik}\;\Linz\, \partial_\nu(\bPhi_0^\nu)_{|\nu=0}\,.$$
This shows that $$\Linz\,( \partial_\nu(\bPhi_0^\nu)_{|\nu=0}\,-\,ik\,\bU_{k})\in \range(\Pi(0))\,=\,\ker ((\Linz)^2)\,=\,\ker ((\Linz)^3)\,,$$
hence
$$ \partial_\nu(\bPhi_0^\nu)_{|\nu=0}\,-\,ik\,\bU_{k}\,\in \range(\Pi(0))\,.$$
This means that there exist numbers $z_\alpha$ such that
$$\partial_\nu(\bPhi_0^\nu)_{|\nu=0}\,-\,ik\,\bU_{k}\,=\,z_\alpha\,\Phi_\alpha^0\,.$$
If we substitute 
$$\widetilde{\bPhi}_0^\nu:={\bPhi}_0^\nu\,-\,\nu\,z_\alpha\,\Phi_\alpha^\nu$$
for ${\bPhi}_0^\nu$,
we thus have that 
$$ \partial_\nu(\widetilde{\bPhi}_0^\nu)_{|\nu=0}\,=\,ik\,\bU_{k}\,.$$
(In order to keep duality we also substitute
$$\widetilde{\Psi}_\alpha^\nu= \Psi_\alpha^\nu +\nu z_\alpha \Psi_0^\nu$$
for  $\Psi_\alpha^\nu$ with $\alpha\in\{1,\ldots,N+1\}$.)
Forgetting the tilda, this implies \eqref{eq:cancel} for all $\alpha\in\{1,\ldots,N+1\}$ because
$$\langle \bPsi_\alpha^0 \cdot \bU_{k}\rangle\,=\,0\,.$$
As to the other diagonal block in $\Du$, it reduces to 
$$(\langle \bPsi_\alpha^0 \cdot \Linu \bPhi_\beta^0\rangle)_{1\leq \alpha,\beta\leq N+1}$$
because $(\Linz)^*\Psi_\alpha^0=0$ for $\alpha\geq 1$.
It remains to compute the first column, starting from second row, in $\Dd$.
At $\alpha$-th row we find
$$\frac{1}{ik}\;\langle \Psi_\alpha^0 \cdot \Linu\,\partial_\nu(\bPhi_0^\nu)_{|\nu=0}\rangle\,+\,
\frac{1}{ik}\;\langle  \partial_\nu(\bPsi_\alpha^\nu)_{|\nu=0}\cdot (\Linu\,\Phi_0^0\,+\,\frac{1}{ik}\;\Linz\, \partial_\nu(\bPhi_0^\nu)_{|\nu=0})\rangle\,+\,
\langle \Psi_\alpha^0\cdot \Lind\,\bPhi_0^0\rangle
$$
$$\,=\,\langle \Psi_\alpha^0 \cdot \Linu\,\bU_{k}\rangle\,-\,
\frac{c_{k}}{i}\;\langle  \partial_\nu(\bPsi_\alpha^\nu)_{|\nu=0}\cdot \bPhi_0^0\rangle\,+\,
\langle \Psi_\alpha^0\cdot \Lind\,\bPhi_0^0\rangle\,,
$$
where in fact the middle term is zero because of \eqref{eq:cancel} and the fact that
$$\langle  \partial_\nu(\bPsi_\alpha^\nu)_{|\nu=0}\cdot \bPhi_0^0\rangle\,=\,-\,
\langle \bPsi_\alpha^0 \cdot \partial_\nu(\bPhi_0^\nu)_{|\nu=0}\rangle$$
by duality. Collecting together the results of the above computations, we get that
$$\widetilde{D}^\nu:=\left(\begin{array}{c|ccc}
\frac{1}{ik\nu}\langle \bPsi_0^\nu \cdot \Lin^\nu \bPhi_0^\nu \rangle & & \left(\langle \bPsi_0^\nu \cdot \Lin^\nu \bPhi_\beta^\nu \rangle\right)_{1\leq \beta\leq N+1} & \\ \hline
 & & & \\
 \left(\frac{1}{(ik\nu)^2}\langle \bPsi_\alpha^\nu \cdot \Lin^\nu \bPhi_0^\nu \rangle\right)_{1\leq \alpha\leq N+1} & & \left(\frac{1}{ik\nu}\langle \bPsi_\alpha^\nu \cdot \Lin^\nu \bPhi_\beta^\nu \rangle\right)_{1\leq\alpha,\beta\leq N+1} & \\
 & & &\end{array}\right)$$
tends to $\widetilde{D}^{(0)}:=$
$$
\left(\begin{array}{c|c|c}
-\,k\,c_{k} &  -kc_{M_1} \; \ldots \; -kc_{M_N} & -kc_P   \\ \hline
\langle (\Linu\,\bU_{k} + \Lind \partial_\theta \bU)_1\rangle & & \langle ( \Linu \bU_P)_1 \rangle\\
\vdots & (\langle ( \Linu \bU_{M_\beta})_\alpha \rangle)_{1\leq \alpha,\beta\leq N} & \vdots \\
\langle (\Linu\,\bU_{k} + \Lind \partial_\theta \bU)_N\rangle & &  \langle ( \Linu \bU_P)_N \rangle\\ \hline
\langle (\bJ^{-1}\bU) \cdot (\Linu\,\bU_{k} + \Lind \partial_\theta \bU)\rangle & \langle (\bJ^{-1}\bU) \cdot  \Linu \bU_{M_\beta}\rangle)_{1\leq \beta\leq N}& \langle (\bJ^{-1}\bU) \cdot \Linu\,\bU_P\rangle\end{array}\right)$$
when $\nu$ goes to zero. We are now going to check that the matrix of the modulated system \eqref{eq:Whithamsyst} linearized about $(k,\bM=\langle \bU\rangle, P=\langle\tfrac{1}{2}\,\bU \cdot \bJ^{-1}\bU\rangle)$ is nothing but 
$\widetilde{D}^{(0)}\,-\,c\,{\bf I}_{N+2}$.
Since 
$$D^\nu= \frac{1}{ik\nu}\,\Sigma(\nu)^{-1} \widetilde{D}^\nu \Sigma(\nu) $$
with
$$\Sigma(\nu)=\left(\begin{array}{c|ccc}
1 & 0 & \ldots & 0\\ \hline
0 & & & \\
\vdots & & \frac{1}{ik\nu} {\bf I}_{N+1} & \\
0 & & &\end{array}\right)\,,$$
the existence of a non-real eigenvalue of $\widetilde{D}^{(0)}\,-\,c\,{\bf I}_{N+2}$ would imply the existence of 
an eigenvalue of $D^\nu$ bifurcating from zero into the right-half plane.
Before linearizing it, let write \eqref{eq:Whithamsyst} in the simplest abstract form
\begin{equation}
\label{eq:Whithamsystabs}\left\{\begin{array}{l}\partial_T k \,=\,-\,\partial_X(c k)\,,\\ [5pt]
\displaystyle
\partial_T \bM\,=\,\partial_X\langle \bJ \Euler \Hamk [\bU]\rangle \,,\\ [5pt]
\partial_T P\,=\,\partial_X\langle \bU\cdot \Euler \Hamk [\bU] \,+\,\Legendre \Hamk [\bU] \rangle\,,
\end{array}\right.
\end{equation}
where 
$$\Hamk (\bU,\partial_\theta\bU):=\Ham(\bU,k\partial_\theta\bU)\,.$$
In quasilinear form, \eqref{eq:Whithamsystabs} reads
\begin{equation}
\left\{\begin{array}{ll}\partial_T k \,=\,&-\,c\,\partial_Xk \,-\, k\,c_{k}\,\partial_X k\,-\,k\,c_{M_\beta} \,\partial_X M_\beta\,-\,k\,c_P\,\partial_X P\,,\\ [5pt]
\displaystyle
\partial_T \bM\,=\,& \langle \bJ \Euler \Hamk [\bU]\rangle_{k}\,\partial_Xk \,+\, \langle \bJ \Euler \Hamk [\bU]\rangle_{M_\beta}\,\partial_X M_\beta\,+\,\langle \bJ \Euler \Hamk [\bU]\rangle_P\,\partial_X P\,,\\ [5pt]
\partial_T P\,=\,&\langle \bU\cdot \Euler \Hamk [\bU] \,+\,\Legendre \Hamk [\bU]\rangle_{k}\,\partial_Xk \,+\, \langle \bU\cdot \Euler \Hamk [\bU] \,+\,\Legendre \Hamk [\bU]\rangle_{M_\beta}\,\partial_X M_\beta \\ [5pt]
&\,+\,\langle \bU\cdot \Euler \Hamk [\bU]\,+\,\Legendre \Hamk [\bU]\rangle_P\,\partial_X P\,.
\end{array}\right.
\end{equation}
The right-hand side in the first row here above is indeed the first component of 
$$(\widetilde{D}^{(0)}\,-\,c\,{\bf I}_{N+2})\,\partial_X(k,M_1,\ldots,M_N,P)^T\,.$$
It remains to identify the other rows. If $a$ denotes any one of the variables $M_\alpha$ or $P$, 
we have
$$\langle \bJ \Euler \Hamk [\bU]\rangle_a\,=\,
\langle \bJ \Hess \Hamk [\bU] \bU_a\rangle\,.$$
Recalling that 
$$\Linu=\bJ (\Linvarz + k \partial_\theta \Linvaru)$$
and observing that 
$$\Linvarz\,=\,\Hess \Hamk \,+\,c\,\bJ^{-1}\,,$$
we thus find that
$$\langle \bJ \Euler \Hamk [\bU]\rangle_a\,=\,
\langle \Linu \bU_a\rangle\,-\,c\,\langle \bU\rangle_a\,.$$
Besides, we have
$$\langle \bJ \Euler \Hamk [\bU]\rangle_{k}\,=\,\langle \bJ \Hess \Hamk [\bU] \bU_{k}\rangle\,+\,\langle 
\bJ \Linvaru \partial_\theta \bU\rangle\,=\,
\langle \Linu \bU_{k}\rangle \,+\,\langle \Lind \partial_\theta \bU\rangle
$$
where we have also used that 
$$\Lind=\bJ (\Linvaru + k \partial_\theta \Linvard)\,.$$
This shows that the second row in \eqref{eq:Whithamsystabs} equivalently reads in quasilinear form
$$\partial_T \bM\,=\, -\,c\,\partial_X \bM\,+\,\langle \Linu \bU_{k} \,+\, \Lind \partial_\theta \bU\rangle\,\partial_Xk \,+\, \langle \Linu \bU_{M_\beta}\rangle\,\partial_X M_\beta\,+\,\langle \Linu \bU_{P}\rangle\,\partial_X P\,,$$
in which the right-hand side clearly coincides with the $\alpha$-th components, $\alpha\in \{1,\ldots,N\}$, of
$$(\widetilde{D}^{(0)}\,-\,c\,{\bf I}_{N+2})\,\partial_X(k,M_1,\ldots,M_N,P)^T\,.$$
In order to check the last component, we first compute that
$$ \langle \bU\cdot \Euler \Hamk [\bU] \,+\,\Legendre \Hamk [\bU]\rangle_a\begin{array}[t]{l}\,=\,\begin{array}[t]{l}
\langle \bU_a\cdot \Euler \Hamk [\bU] \,+\,\bU \cdot \Hess \Hamk [\bU] \bU_a\rangle \\ [5pt] \,+\,
\displaystyle\left\langle \,\frac{\partial^2 \Ham}{\partial U_{\alpha,x}\partial U_{\beta,x}}\;(k\partial_\theta U_\alpha) \,(k\partial_\theta U_\beta)_a\right\rangle
 \\ [15pt] \,+\,
\displaystyle\left\langle \frac{\partial^2 \Ham}{\partial U_{\alpha,x}\partial U_{\beta}}\;(k\partial_\theta U_\alpha) \,(U_\beta)_a\right\rangle\,-\,\left\langle \frac{\partial \Ham}{\partial U_{\alpha}}\,(U_\alpha)_a\right\rangle
\end{array}\\ [80pt]
= \begin{array}[t]{l}
\langle \bU \cdot (\bJ^{-1}\Linu \bU_a)\rangle \,-\,c\,\langle \bU\cdot \bJ^{-1} \bU_a\rangle\,-\,\langle \bU \cdot k\partial_\theta(\Linvaru \bU_a)\rangle \\ [5pt] \,+\,
\displaystyle\left\langle \,\frac{\partial^2 \Ham}{\partial U_{\alpha,x}\partial U_{\beta,x}}\;(k\partial_\theta U_\alpha) \,(k\partial_\theta U_\beta)_a\right\rangle
 \\ [15pt] \,+\,
\displaystyle\left\langle \frac{\partial^2 \Ham}{\partial U_{\alpha,x}\partial U_{\beta}}\;(k\partial_\theta U_\alpha) \,(U_\beta)_a\right\rangle\,-\,\left\langle 
(U_\alpha)_a\cdot k\Dif_\theta\left(\frac{\partial \Ham}{\partial U_{\alpha,x}}\right)\right\rangle\,,
\end{array}
\end{array}
$$
where again all derivatives of $\Ham$ are evaluated at $(\bU,k\partial_\theta\bU)$, and in fact the last four terms cancel out by definition of $\Linvaru$. Therefore, using the symmetry of $\bJ$, we obtain that
$$ \langle \bU\cdot \Euler \Hamk [\bU] \,+\,\Legendre \Hamk [\bU]\rangle_{M_\beta}\,=\,
\langle (\bJ^{-1}\bU) \cdot \Linu \bU_{M_\beta}\rangle$$
and 
$$ \langle \bU\cdot \Euler \Hamk [\bU] \,+\,\Legendre \Hamk [\bU]\rangle_P\,=\,-\,c\,+\,
\langle (\bJ^{-1}\bU) \cdot \Linu \bU_{P}\rangle\,.$$
Concerning the derivative $ \langle \bU\cdot \Euler \Hamk [\bU] \,+\,\Legendre \Hamk [\bU]\rangle_{k}$, the computation is similar, with a few more terms:
$$ \langle \bU\cdot \Euler \Hamk [\bU] \,+\,\Legendre \Hamk [\bU]\rangle_{k}\begin{array}[t]{l}\,=\,
\langle (\bJ^{-1}\bU) \cdot \Linu \bU_{k}\rangle\,+\,\langle \bU\cdot \Linvaru \partial_\theta \bU\rangle\\ [10pt] \quad+\,
\left\langle \dfrac{\partial^2 \Ham}{\partial U_{\alpha,x}\partial U_{\beta,x}}\;(k\partial_\theta U_\alpha)\,(\partial_\theta U_\beta)
 \right\rangle\\ [15pt]
=\, \langle (\bJ^{-1}\bU) \cdot ( \Linu \bU_{k}\,+\,\Lind \partial_\theta \bU)\rangle \end{array}
$$
after another integration by parts and by definition of $\Lind$ (plus the symmetry of $\bJ$).
\end{proof}

\subsection{Small amplitude wave trains}\label{ss:small}
A natural question is whether at least small amplitude waves are stable. Our main purpose here is to extend to our general, abstract framework, the necessary stability conditions that were exhibited by Whitham \cite[pp.~489-491,512-513]{Whitham}, basically for scalar problems.
Small amplitude wave trains correspond to a near-linear situation. More precisely, solutions to \eqref{eq:absHam} involving only small oscillations (with bounded wavenumber) around a mean value $\ubM$ are expected to be well approximated by solutions to the linearized system
\begin{equation}
\label{eq:absHamzero}
\partial_t\bU = \sJ (\Hess \Ham[\ubM] \bU)\,,
\end{equation}
with
$$(\Hess \Ham[\ubM] \bU)_\alpha=\, \frac{\partial^2 \Ham}{\partial U_\alpha \partial U_\beta} U_\beta \,+\,
\frac{\partial^2 \Ham}{\partial U_\alpha \partial U_{\beta,x}} U_{\beta,x}
\,-\,\Dif_x\left(\frac{\partial^2 \Ham}{\partial U_{\alpha,x}\partial U_\beta} U_\beta\,+\,\frac{\partial^2 \Ham}{\partial U_{\alpha,x}\partial U_{\beta,x}}  U_{\beta,x}\right)\,,$$
where all second order derivatives of $\Ham$ are evaluated at $(\ubM,0)$.

Let us concentrate for a while on \emph{periodic} wave solutions to \eqref{eq:absHam} of small amplitude, that is of the form
$\bU(x,t)=\ubU(kx+\omega t)$ with $\ubU$ being a $1$-periodic function,  $\langle \ubU\rangle = \ubM$, and  $\|\ubU-\ubM\|_{L^\infty}$ small.
If in addition we assume that $\Ham$ is an even function of $\bU_x$ (which is always the case when the contribution of $\bU_x$ to $\Ham$ comes from a kinetic energy), the equations governing $\ubU$  
\begin{equation}
\label{eq:absHamuprofk}
\partial_\theta\left(\frac{\partial \Ham}{\partial U_\alpha}(\ubU,k\partial_\theta\ubU)-k\partial_\theta\left(\frac{\partial \Ham}{\partial U_{\alpha,x}}(\ubU,k\partial_\theta\ubU)\right)+c\,\frac{\partial \Impulse}{\partial U_\alpha}(\ubU)\right)=0\,,\: \alpha\in\{1,\ldots,N\}\,,
\end{equation}
with $c=-\omega/k$, are obviously symmetric under $\theta\mapsto -\theta$. In this case, we may assume without loss of generality that $\ubU$ is an even function of $\theta$.
Then, denoting by $a= 2 \|\langle \ubU \cos (2\pi \cdot)\rangle\|$ an approximate amplitude of the wave,
successive Lyapunov-Schmidt reduction arguments show that $\ubU$ can be expanded  for small values of $a$ as 
\begin{equation}\label{eq:expansionU}
\ubU(k,\ubM,a;\theta)=\ubM\,+\,a \ubU_1(k,\ubM;\theta)\,+\,a^2 \ubU_2(k,\ubM;\theta)\,+\,a^3 \ubU_3(k,\ubM;\theta)\,+\,o(a^3)
\end{equation}
with $\langle  \ubU_m\rangle =0$ for all $m=1,2,3$, $\langle  \ubU_1\cdot {\bf K} \ubU_m  \rangle =0$ for $m=2,3$, 
$\langle  \ubU_2\cdot {\bf K} \ubU_3  \rangle =0$
whatever the $N\times N$ matrix $ {\bf K}$ (in fact, since $\ubU$ is even, $\ubU_1(k,\ubM;\theta)= 2  \cos (2\pi \theta) \langle \ubU_1(k,\ubM;\cdot) \cos(2\pi \cdot) \rangle$, and 
$\ubU_2(k,\ubM;\theta)= 2  \cos (4\pi \theta) \langle \ubU_2(k,\ubM;\cdot) \cos (4\pi \cdot) \rangle$), 
and the frequency $\omega$ can also be expanded as
\begin{equation}\label{eq:expansionom}
\omega(k,\ubM,a)=\omega_0(k,\ubM)+a^2 \omega_2(k,\ubM)+o(a^2),
\end{equation}
where $\omega_0(k,\ubM)$ is the frequency of periodic wave solutions to \eqref{eq:absHamzero}, determined by the (linear) dispersion relation
$$\det (k\uLinvar(2i\pi k)-\omega_0(k,\ubM) \bJ^{-1})\,=\,0\,,$$
$$ (\uLinvar(2i\pi k))_{\alpha\beta}\,:=\,\frac{\partial^2 \Ham}{\partial U_\alpha \partial U_{\beta}}\,+\,i\,(2\pi k)\,
\left(\frac{\partial^2 \Ham}{\partial U_\alpha \partial U_{\beta,x}} \, 
\,-\,\frac{\partial^2 \Ham}{\partial U_{\alpha,x}\partial U_\beta}\right)\,+\,(2\pi k)^2\,\frac{\partial^2 \Ham}{\partial U_{\alpha,x}\partial U_{\beta,x}}\,,$$
all derivatives of $\Ham$ being evaluated at $(\ubM,0)$. Furthermore, $\omega_2(k,\ubM)$ can be expressed in terms of 
mean values involving $\ubU_1$ and $\ubU_2$ as
\begin{equation}\label{eq:omega2}
\omega_2(k,\ubM):=\,- \,k\,\frac{\langle \partial_\theta\ubU_1 \cdot \partial_\theta (\delta^3\!\Hamk [\ubM](\ubU_1,\ubU_2) + \frac{1}{6}\,\delta^4\!\Hamk [\ubM](\ubU_1,\ubU_1,\ubU_1))\rangle }{\langle \partial_\theta\ubU_1 \cdot \bJ^{-1} \partial_\theta \ubU_1\rangle}\,,
\end{equation}
where $\delta^3\!\Hamk$ and $\delta^4\!\Hamk$ denote respectively the third and fourth order variational derivatives of $\Hamk: \bU\mapsto \Ham(\bU,k\partial_\theta\bU)$ (up to this point, we have preferred the notation $\Euler \Ham$ for the first derivative $\delta \Ham$, and $\Hess \Ham$ for the second order one $\delta^2 \Ham$).
The actual derivation of \eqref{eq:omega2} follows the lines of computations made by Whitham in \cite[pp.~472--475]{Whitham}\footnote{We warn the reader that we have taken the opposite sign for $\omega$ compared to the one chosen by Whitham and normalized periods of profiles to one instead of $2\pi$.}, first for KdV and then for water waves in arbitrary depth, except that here we have no explicit formula for $\omega_0$. The reader may check that \eqref{eq:omega2} is consistent with values found by Whitham, namely
$2\pi\,\omega_2(k,\ubM)=-3c_0^2/[32\,(2\pi k)]$ for (gKdV) with $f(v)=c_0(\frac{1}{2}v^2+\frac{1}{4} v^3)$ (see \cite[pp.~463,473]{Whitham}),
and also $2\pi\,\omega_2(k,\ubM)=-3\sigma\,(2\pi k)+24\sigma^2\ubM^2/(2\pi k)$ for (gKdV) with $f(v)=\frac{1}{2}v^2+\sigma v^4$, the zero-mean periodic solutions of which are basically governed by the same equation as periodic solutions to the Klein-Gordon equation considered in \cite[pp.~486--487]{Whitham}. 
\paragraph{Derivation of \eqref{eq:omega2}.}
By plugging \eqref{eq:expansionU} and \eqref{eq:expansionom} in \eqref{eq:absHamuprofk}, we get successively, by increasing order in powers of $a$,
$$\begin{array}{ll} (k\uLinvar(k\partial_\theta)-\omega_0(k,\ubM) \bJ^{-1} )\,\partial_\theta \ubU_1  \,=\,& 0\,,\\ [8pt]
(k\uLinvar(k\partial_\theta)-\omega_0(k,\ubM) \bJ^{-1} )\,\partial_\theta \ubU_2\,=\,&  \partial_\theta\big(\tfrac12\,k\,\delta^3\!\Hamk [\ubM](\ubU_1,\ubU_1)\big)\,,\\ [8pt]
(k\uLinvar(k\partial_\theta)-\omega_0(k,\ubM) \bJ^{-1} )\,\partial_\theta \ubU_3\,=\,& \partial_\theta\big(k\,\delta^3\!\Hamk [\ubM](\ubU_1,\ubU_2)+ \tfrac{1}{6}\,k\,\delta^4\!\Hamk [\ubM](\ubU_1,\ubU_1,\ubU_1))\big) \\ [5pt]
& \,+\, \omega_2(k,\ubM)\,\bJ^{-1}\,\partial_\theta\ubU_1\,.\end{array}$$
By taking the inner product of the last equation with $\partial_\theta \ubU_1$, we see that the left hand side vanishes because of the first equation and of the self-adjointness of the operator $(k\uLinvar(k\partial_\theta)-\omega_0(k,\ubM) \bJ^{-1} )$, hence \eqref{eq:omega2}.
Note that for the very same reason, one can draw information 
on $\omega_{0,k}(k,\ubM)$
from the first equation. Indeed, by differentiating it with respect to $k$ and taking the inner product with $\partial_\theta \ubU_1$, the term involving $\partial_\theta \ubU_{1,k}$ vanishes, and we find that
\begin{equation}\label{eq:c0p}
\omega_{0,k}(k,\ubM)\,=\,\frac{\langle \partial_\theta\ubU_1 \cdot (k\uLinvar(k\partial_\theta))_{k} \,\partial_\theta \ubU_1\rangle }{\langle \partial_\theta\ubU_1 \cdot \bJ^{-1} \partial_\theta \ubU_1\rangle}\,.
\end{equation}
More explicitly, the numerator here above reduces to
$$\langle\nabla_{\bU}^2\Ham(\ubM,0) (\partial_\theta\bU_1,\partial_\theta\bU_1)\rangle
\,-\,3k^2\,\langle\nabla_{\bU_x}^2\Ham(\ubM,0) (\partial_\theta\bU_1,\partial_\theta^3\bU_1)\rangle\,.$$

Returning to a more general small amplitude modulated wave train, the previous remarks show that we can expand its lower order term $\bU^0$, which we merely denote by $\bU$ in what follows, as
$$
\begin{array}{rl}
\bU(T,X,\theta)=&\langle \bU\rangle(T,X)
\,+\,a(T,X)\ \bU_1(k(T,X),\langle \bU\rangle(T,X);\theta)\\
&+\,a(T,X)^2\ \bU_2(k(T,X),\langle \bU\rangle(T,X);\theta)\\
&+\,a(T,X)^3\ \bU_3(k(T,X),\langle \bU\rangle(T,X);\theta)\,+\,o(a(T,X)^3)\,.
\end{array}
$$
Our aim is to show that, for small $a$, the modulated system \eqref{eq:Whithamsyst} associated with $\bU$ decouples into the lower order system
\begin{equation}\label{eq:lower}
\partial_T\langle \bU\rangle=\bJ\,\partial_X(\nabla_{\bU} \Ham(\langle \bU\rangle,0))\,,
\end{equation}
and the $2\times 2$ system (pointed out in \cite[p.~490]{Whitham} when $\langle \bU\rangle\equiv \ubM$)
\begin{equation}\label{eq:ka2}
\left\{\begin{array}{rlll}\partial_Tk &\,-\,\omega_{0,k}(k,\langle \bU\rangle)\ \partial_Xk & \,-\,\omega_2(k,\langle \bU\rangle)\ \partial_X(a^2)&\,=\,\omega_{0,\bM}(k,\langle \bU\rangle)\ \partial_X\langle \bU\rangle\,,\\[0.5em]
\partial_T(a^2)& \,-\,\omega_{0,kk}(k,\langle \bU\rangle) a^2\,\partial_Xk&\,-\,\omega_{0,k}(k,\langle \bU\rangle) \partial_X(a^2)&\,=\,\cO(a^2)\,\partial_X\langle \bU\rangle\,.
\end{array}\right.
\end{equation}
As a consequence, a necessary condition for \eqref{eq:Whithamsyst} to be hyperbolic in the small amplitude limit is that both \eqref{eq:lower} and the lef-hand side of \eqref{eq:ka2} with frozen mean be hyperbolic. Regarding the latter, this requires  that 
$\omega_2(k,\langle \bU\rangle)\omega_{0,kk}(k,\langle \bU\rangle)$ be positive, in which case the characteristic velocities of \eqref{eq:ka2} are $\cO(a)$ perturbations ($-\,\omega_{0,k}(k,\langle \bU\rangle)\pm a \sqrt{\omega_2(k,\langle \bU\rangle)\omega_{0,kk}(k,\langle \bU\rangle)}$) of 
$-\,\omega_{0,k}(k,\langle \bU\rangle)$. As to the former system \eqref{eq:lower}, 
it is for instance nothing but the first order conservation law
\begin{equation}\label{eq:burgers}
\partial_t \langle v\rangle+\partial_Xp(\langle v\rangle)\,=\,0
\end{equation}
if the original equation is (gKdV), that is Eq.~\eqref{eq:KdVgen}. Of course \eqref{eq:burgers} is always hyperbolic, as any (real) first order conservation law. 
However, the hyperbolicity of the reduced system \eqref{eq:lower} is not automatic in general. In particular, 
when we start from the Euler--Korteweg system, we receive as \eqref{eq:lower} the Euler system. Therefore, a necessary condition for the modulated system to be hyperbolic in the small amplitude limit is that the mean value of the wave train be a stable state of the Euler system. According to Theorem~\ref{thm:absstab}, this implies that small amplitude periodic solutions whose mean value is an unstable state of the Euler system are themselves (spectrally) unstable. 

\paragraph{Derivation of \eqref{eq:lower} and \eqref{eq:ka2}.} There is not much to do about the first equation in \eqref{eq:ka2}. Indeed, the first equation $\partial_T k + \partial_X(ck)=0$ in \eqref{eq:Whithamsyst} together with the fact that 
$ck=-\omega_0(k,\langle \bU\rangle)-a^2\omega_2(k,\langle \bU\rangle)+o(a^2)$ yield 
$$\partial_Tk \,-\,(\omega_{0,k}\,+\,\,a^2\,\omega_{2,k}) \partial_Xk  \,-\,\omega_2\,\partial_X(a^2)\,=\,(\omega_{0,\bM}\,+\,\,a^2\,\omega_{2,\bM}) \partial_X\,\langle \bU\rangle+\,o(a^2)\,,$$
and the terms $a^2\,\omega_{2,k}$ and $a^2\,\omega_{2,\bM}$ are negligible compared to the $\cO(a)$ perturbation of $-\,\omega_{0,k}$ we are expecting (this was already pointed out in \cite[p.~490]{Whitham}). So the main points consist in showing that the middle equations 
$\partial_T \langle \bU \rangle\,=\,\bJ\,\partial_X\langle \bG\rangle$ in \eqref{eq:Whithamsyst} do reduce to \eqref{eq:lower} when $a$ goes to zero,
and that together with the last equation $\partial_T \langle Q \rangle\,=\,\partial_X\langle S \rangle$ in \eqref{eq:Whithamsyst} they simplify into an equation for $a^2$ whose principal part amounts to
$$\partial_T(a^2) \,-\,\partial_X(\omega_{0,k}(k,\langle \bU\rangle) a^2)\,=\,\cO(a^2)\,\partial_X\langle \bU\rangle\,.$$
Let us start by expanding $\bG$ in powers of $a$. Note that we shall need an expansion not only for $\langle \bG\rangle$ but also for $\langle \bU\cdot \bG\rangle$, which arises in $S$. Recall indeed that 
$$\bG=\delta \Hamk[\bU]\,,\;S\,=\,\bU \cdot \bG\,- \Ham(\bU,k\partial_\theta \bU)\,+\,(k\partial_\theta 
 U_{\alpha})\,\frac{\partial \Ham}{\partial  U_{\alpha,x}}(\bU,k\partial_\theta \bU)\,.$$
We thus see that 
$$\bG=\delta \Hamk[\langle\bU\rangle]\,+\,a \,\delta^2\!\Hamk[\langle\bU\rangle](\bU_1)\,+\,a^2\,\big( \delta^2\!\Hamk[\langle\bU\rangle](\bU_2)\,+\,
\tfrac{1}{2}\delta^3\!\Hamk[\langle\bU\rangle](\bU_1,\bU_1)\big)\,+\,o(a^2)\,,$$
with 
$$\delta \Hamk[\langle\bU\rangle]\,=\,\nabla_{\bU} \Ham(\langle \bU\rangle,0)\,,$$
and $$\langle \delta^2\!\Hamk[\langle\bU\rangle](\bU_1)\rangle \,=\,0\,,\;\langle \delta^2\!\Hamk[\langle\bU\rangle](\bU_1)\rangle \,=\,0\,,$$
since $\delta^2\!\Hamk[\langle\bU\rangle]$ is a constant-coefficient differential operator, and $\bU_1$, $\bU_2$, as well of course as their derivatives, have zero mean values. Therefore, it just remains
$$\langle \bG\rangle\,=\,\nabla_{\bU} \Ham(\langle \bU\rangle,0)\,+\,\tfrac{1}{2}\,a^2\,\langle \delta^3\!\Hamk[\langle\bU\rangle](\bU_1,\bU_1)\rangle \,+\,o(a^2)\,.$$
By neglecting all the $O(a^2)$ in $\partial_T \langle \bU \rangle\,=\,\bJ\,\partial_X\langle \bG\rangle$, we thus receive \eqref{eq:lower}.
Let us turn to the expansions of $\langle Q\rangle $ and $\langle S\rangle$.
The expansion of $\langle Q\rangle$ is readily seen to be given by
$$\langle Q\rangle\,=\, \tfrac{1}{2} \langle\bU\rangle \cdot \bJ^{-1}\langle\bU\rangle \,+\,
\tfrac12\,
a^2\,\langle \bU_1\cdot \bJ^{-1}\bU_1\rangle \,+\,o(a^2)\,,$$
and we also find that
$$\langle\bU\cdot  \bG\rangle\,=\,\langle \bU\rangle \cdot \langle \bG\rangle\,+\,a^2\,
\langle \bU_1\cdot\delta^2\!\Hamk[\langle\bU\rangle](\bU_1)\rangle
\,+\,o(a^2)\,.$$
(We have used again that $\bU_1$ and $\bU_2$ have zero mean values.) 
Furthermore, we have
$$\langle\Hamk[\bU]\rangle\,=\,\Hamk[\langle\bU\rangle]\,+\,\tfrac{1}{2}\,a^2\,\langle\delta^2\!\Hamk[\langle\bU\rangle](\bU_1,\bU_1)\rangle\,+\,o(a^2)\,,$$
and finally
$$\begin{array}[t]{l} (k\partial_\theta 
 U_{\alpha})\,\dfrac{\partial \Ham}{\partial  U_{\alpha,x}}(\bU,k\partial_\theta \bU)\,=\,a\,(k\partial_\theta 
 U_{1\alpha})\,\dfrac{\partial \Ham}{\partial  U_{\alpha,x}}(\langle\bU\rangle,0)\,+\,a^2\,(k\partial_\theta 
 U_{2\alpha})\,\dfrac{\partial \Ham}{\partial  U_{\alpha,x}}(\langle\bU\rangle,0)\\ [10pt]
 +\, a^2\,(k\partial_\theta 
 U_{1\alpha})\,\left(U_{1\beta}\,\dfrac{\partial^2 \Ham}{\partial  U_{\alpha,x}\partial U_\beta}(\langle\bU\rangle,0) +\,(k\partial_\theta U_{1\beta})\dfrac{\partial^2 \Ham}{\partial  U_{\alpha,x}\partial U_{\beta,x}}(\langle\bU\rangle,0)\right)\,+\,o(a^2)\,,\end{array}
$$
hence 
$$\left\langle  (k\partial_\theta U_{\alpha})\,\dfrac{\partial \Ham}{\partial  U_{\alpha,x}}(\bU,k\partial_\theta \bU)\right\rangle\,=\,a^2\,k^2\,
\langle\nabla_{\bU_x}^2\Ham(\langle\bU\rangle,0) (\partial_\theta\bU_1,\partial_\theta\bU_1)\rangle\,+\,o(a^2)\,.
$$
Now, taking the inner product with $\langle \bU \rangle$ of $\bJ^{-1}\,\partial_T \langle \bU \rangle\,=\,\partial_X\langle \bG\rangle$ and subtracting to  $\partial_T \langle Q \rangle\,=\,\partial_X\langle S \rangle$ we see that
$$\partial_T\left(\tfrac12\,
a^2\,\langle \bU_1\cdot \bJ^{-1}\bU_1\rangle\right)\begin{array}[t]{l}\,=\,\displaystyle \partial_X\left(\langle\bU\cdot  \bG\rangle \,-\,\langle\Hamk[\bU]\rangle\,+\,
\left\langle  (k\partial_\theta U_{\alpha})\,\dfrac{\partial \Ham}{\partial  U_{\alpha,x}}(\bU,k\partial_\theta \bU)\right\rangle\right)\\ [10pt]
\quad \,-\,\langle \bU \rangle\cdot \partial_X\langle \bG\rangle\,+\,o(a^2) \,.\end{array}$$
As expected, the zeroth order terms in the right-hand side cancel out, and we receive after several simplifying operations
$$\begin{array}[t]{rcl} 
\displaystyle
\partial_T\left(
a^2\,\langle \bU_1\cdot \bJ^{-1}\bU_1\rangle\right)&=&
\displaystyle \partial_X\Big( a^2 \langle\nabla_{\bU}^2\Ham(\langle\bU\rangle,0) (\bU_1,\bU_1) \\[10pt]
&&\displaystyle\qquad\qquad \,+\,3\,k^2\,\nabla_{\bU_x}^2\Ham(\langle\bU\rangle,0) (\partial_\theta\bU_1,\partial_\theta\bU_1)\rangle
\Big)\\[10pt]
&&\displaystyle \,+\,a^2\,\langle \delta^3\!\Hamk[\langle\bU\rangle](\bU_1,\bU_1)\rangle\ \d_X\langle\bU\rangle\,+\,o(a^2) \,.\end{array}$$
In order to conclude, let us recall \eqref{eq:c0p} hence, since the dependency of $\ubU_1$ is a cosine function,
$$
\omega_{0,k}(k,\ubM)\,=\,\frac{\langle\nabla_{\bU}^2\Ham(\langle\bU\rangle,0) (\bU_1,\bU_1)\,+\,3\,k^2\,\nabla_{\bU_x}^2\Ham(\langle\bU\rangle,0) (\partial_\theta\bU_1,\partial_\theta\bU_1)\rangle}{\langle \partial_\theta\ubU_1 \cdot \bJ^{-1} \partial_\theta \ubU_1\rangle}\,.
$$
Therefore, the equation above reads
$$\partial_T\left(
a^2\,\langle \bU_1\cdot \bJ^{-1}\bU_1\rangle\right)\,+\,\displaystyle \partial_X \Big(\omega_{0,k}(k,\ubM)\, a^2 \langle \bU_1\cdot \bJ^{-1}\bU_1\rangle\Big)\,=\,\cO(a^2)\,\partial_X\langle \bU\rangle\,+\,o(a^2) \,. $$
The factor $\langle \bU_1\cdot \bJ^{-1}\bU_1\rangle$ can now be eliminated by using the equation on $k$.

\section{Application to the Euler--Korteweg system}\label{s:EK}

In this section, we concentrate on the Euler--Korteweg equations,  \eqref{eq:EKabs1d} in Eulerian coordinates, and  \eqref{eq:EKabsLagb} in mass Lagrangian coordinates. We derive Whitham's modulation equations for both systems, and show that away from vacuum, the modulated systems are equivalent through a mass Lagrangian change of coordinates. Thus, it is sufficient to check the hyperbolicity of modulation equations for either one of these systems in order to determine 
whether our necessary condition for the spectral stability of periodic waves under small wave number perturbations is satisfied.

\subsection{Periodic traveling waves}
 Periodic traveling wave solutions to \eqref{eq:EKabs1d} and \eqref{eq:EKabsLagb} are respectively of the form
$(\rho,\vits)=(\Rho,\Vits)(x-\sigma t)$ and $(\vol,\vits)=(\Vol,\VitsL)(y+j t)$, with a one-to-one correspondance between the two frameworks encoded by
$$\Rho(\xi) \Vol(\Y(\xi))\,=\,1\,,\;\Vits(\xi)=\VitsL(\Y(\xi))\,,
\dfrac{\dif \Y}{\dif \xi}\,=\,\Rho\,=\,\dfrac{1}{\Vol(\Y)}\,.$$
Up to translations, these periodic traveling waves generically arise as four-parameter families. Natural parameters are 
\begin{itemize}
\item their speed, that is $\sigma$ in Eulerian coordinates, and $-j$ in Lagrangian coordinates,
\item a first constant of integration, which turns out to be $j$ in Eulerian coordinates, and $\sigma$ in Lagrangian coordinates.
\item two other constants of integration/Lagrange multipliers, which we denote by $\lambda$ and $\mu$, in the profile equations.
\end{itemize}
To be more precise, the profile equations read
\begin{equation}\label{eq:prof}\left\{\begin{array}{l}\Rho\,(\,\Vits\,-\,\sigma)\,\equiv\,j\,,\\
(\Legendre_\rho \Lag)(\Rho,\Rho_\xi)\,\equiv\,-\,{\lambda}\end{array}\right.\;\qquad\;\left\{\begin{array}{l} \VitsL\,-\,j\,\Vol\,\equiv\,\sigma\,,\\
(\Legendre_\vol{\lag })(\Vol,\Vol_\zeta)\,=\,-\,{\mu}\,,\end{array}\right.
\end{equation}
where the Lagrangians $\Lag$ and $\lag$ are defined by
\begin{equation}\label{eq:lag}
\Lag:\,=\,\En\,-\,\frac{j^2}{2\rho} \,-\,\mu\rho\,,\quad \lag \,:=\, \en\,-\,\frac{j^2\,\vol^2}{2}\,-\,\lambda\,\vol\,,
\end{equation}
and
$$\Legendre_\rho \Lag \,:=\,\rho_x\,\frac{\partial \Lag }{\partial \rho_x}\,-\,\Lag \,,\quad \Legendre_\vol \lag \,:=\,\vol_y\,\frac{\partial \lag }{\partial \vol_y}\,-\,\lag\,,$$
are (obvious) first integrals of the Euler--Lagrange equations $\Euler_\rho\Lag\,=\,0$ and $\Euler_\vol \lag\,=\,0$ respectively.
In addition, there is a simple relationship between the mean values of Eulerian profiles and of Lagrangian profiles. Indeed,
if  $\Xi$ is the period of a traveling wave in Eulerian coordinates, the period of its counterpart in Lagrangian coordinates
is $\Y(\Xi)$ (if $\Y$ is chosen so that $\Y(0)=0$),
 and we have
$$\langle \Rho \rangle := \frac{1}{\Xi}\int_{0}^{\Xi} \Rho(\xi)\,\dif \xi\,=\,\frac{\Y(\Xi)}{\Xi}\,,\quad 
\langle \Vol \rangle := \frac{1}{\Y(\Xi)}\int_{0}^{\Y(\Xi)} \Vol(\zeta)\,\dif \zeta\,=\,\frac{\Xi}{\Y(\Xi)}\,,$$
$$\langle \Vits \rangle := \frac{1}{\Xi}\int_{0}^{\Xi} \Vits(\xi)\,\dif \xi\,=\,\frac{\Y(\Xi)}{\Xi}\,\langle \Vol \,\VitsL\rangle,\quad 
\langle \VitsL \rangle := \frac{1}{\Y(\Xi)}\int_{0}^{\Y(\Xi)} \VitsL(\zeta)\,\dif \zeta\,=\,\frac{\Xi}{\Y(\Xi)}\,\langle \Rho \,\Vits\rangle\,,$$
hence the remarkable identities
$$\langle \Rho \rangle\,=\,\frac{1}{\langle \Vol \rangle}\,,\quad 
\langle \Vits \rangle\,=\,\frac{\langle \Vol \,\VitsL\rangle}{\langle \Vol \rangle}\,,\quad 
\langle \VitsL \rangle\,=\,\frac{\langle \Rho \,\Vits\rangle}{\langle \Rho\rangle}\,.$$
Note that these relations are of course compatible with the profile equations
$\Rho\,(\,\Vits\,-\,\sigma)\,\equiv\,j$, $\VitsL\,-\,j\,\Vol\,\equiv\,\sigma$.
We warn the reader that for convenience we denote by the same brackets $\langle \cdot \rangle$ the mean values with respect to $\xi$ and those with respect to $\zeta$. This should not be too much confusing  since we use different notations for the Eulerian and the Lagrangian dependent variables. 

\subsection{Whitham's modulated equations}
As done previously in the abstract case, we look for solutions of \eqref{eq:EKabs1d} and \eqref{eq:EKabsLagb} having asymptotic expansions of the form
$$(\rho,\vits)(t,x)=(\rho_0,\vits_0)(\varepsilon t, \varepsilon x, \Phi(\varepsilon t, \varepsilon x)/\varepsilon)\,+\,
\varepsilon\,(\rho_1,\vits_1)(\varepsilon t, \varepsilon x, \Phi(\varepsilon t, \varepsilon x)/\varepsilon,\varepsilon)\,+\,o(\varepsilon)\,,$$
$$(\vol,\vits)(s,y)=(\vol_0,\vitsL_0)(\varepsilon s, \varepsilon y, \phi(\varepsilon s, \varepsilon y)/\varepsilon)\,+\,
\varepsilon\,(\vol_1,\vitsL_1)(\varepsilon s, \varepsilon y, \phi(\varepsilon s, \varepsilon y)/\varepsilon,\varepsilon)\,+\,o(\varepsilon)\,,$$
where the profiles $(\rho_0,\vits_0)$, $(\rho_1,\vits_1)$, $(\vol_0,\vitsL_0)$, and $(\vol_1,\vitsL_1)$ are $1$-periodic in their third variable $\theta$.
Denoting by $T$, $X$, $S$, and $Y$ the rescaled variables ($\varepsilon t$, $\varepsilon x$, $\varepsilon s$, and $\varepsilon y$),
we introduce the further notations
$$K:=\Phi_X\,,\;\Omega:=\Phi_T\,,\;\sigma\,:=-\,\frac{\Omega}{K}\,,$$
$$k:=\phi_Y\,,\;\omega:=\phi_S\,,\;j\,:=\frac{\omega}{k}\,.$$
The existence of smooth phases $\Phi$ and $\phi$ requires the standard compatibility conditions
\begin{equation}\label{eq:comp}
\partial_T K \,+\,\partial_X (\sigma K)\,=\,0\,,\quad
\partial_S k \,-\,\partial_Y (jk)\,=\,0\,.
\end{equation}
We obtain equations for the leading profiles in both Eulerian and Lagrangian coordinates by plugging the asymptotic expansions in \eqref{eq:EKabs1d} and 
\eqref{eq:EKabsLagb}, using that
 $$\partial_t= \varepsilon \partial_T\,+\,\Omega\,\partial_\theta\,,\quad \partial_x= \varepsilon \partial_X\,+\,K\,\partial_\theta\,,\qquad\partial_s= \varepsilon \partial_S\,+\,\omega\,\partial_\theta\,,\quad \partial_y= \varepsilon \partial_Y\,+\,k\,\partial_\theta\,,
 $$ and retaining only the leading order terms. As expected, we recover the traveling profile equations \eqref{eq:prof}-\eqref{eq:lag}, up to the rescaling that makes their period equal to one. In Eulerian coordinates, this gives
\begin{equation}
\label{eq:EKW0}
\partial_\theta(\rho_0\,\vits_0)\,-\sigma\,\partial_\theta\rho_0\,=\,0\,,\quad (\vits_0-\sigma)\,\partial_\theta\vits_0\,+\,\partial_\theta\chem_0\,=\,0\,,
\end{equation}
where
$$\chem_0:=\,\frac{\partial \En}{\partial \rho}(\rho_0,K\partial_\theta\rho_0 )\,-\,K\,\Dif_\theta\left(\frac{\partial \En}{\partial \rho_x}(\rho_0,K\partial_\theta\rho_0 )\right)\,,$$
while in Lagrangian coordinates,
\begin{equation}
\label{eq:EKLagW0}
\partial_\theta\vitsL_0\,-\,j\,\partial_\theta\vol_0\,=\,0\,,\quad j\,\partial_\theta\vitsL_0\,+\,\partial_\theta\press_0\,=\,0\,,
\end{equation}
where
$$\press_0:=-\,\frac{\partial \en}{\partial \vol}(\vol_0,k\partial_\theta\vol_0 )\,+\,k\,\Dif_\theta\left(\frac{\partial \en}{\partial \vol_y}(\vol_0,k\partial_\theta\vol_0 )\right)\,.$$

The next order leads to Whitham's modulated equations. Indeed, collecting all the terms of order one in \eqref{eq:EKabs1d} and \eqref{eq:EKabsLagb} we get
 \begin{equation}\label{eq:EKabs1d1}
\left\{\begin{array}{l}-\sigma\,K\,\partial_\theta\rho_1\,+\,K\,\partial_\theta(\rho_0\,\vits_1\,+\,\rho_1\,\vits_0)\,+\,\partial_T\rho_0 +\partial_X (\rho_0 \vits_0)\,=\,0\,,\\ [5pt]
-\sigma\,K\,\partial_\theta\vits_1\,+\,K\,\partial_\theta(\vits_0\,\vits_1)\,+\,K\,\partial_\theta(A_0 \rho_1\,+\,B_0)\,+\,\partial_T \vits_0 + \vits_0\partial_X\vits_0 \,+\,\partial_X\chem_0\,=\,0\,,
\end{array}\right.
\end{equation}
\begin{equation}\label{eq:EKabsLagb1}
\left\{\begin{array}{l}j\,k\,\partial_\theta \vol_1 \,-\,k\partial_\theta\vitsL_1\,+\,\partial_S{\vol}_0 \,-\,\partial_Y{\vitsL}_0\,=\,0\,,\\ [5pt]
j\,k\,\partial_\theta \vitsL_1\,-\,k\partial_\theta (a_0\,\vol_1\,+\,b_0)\,+\,\partial_S {\vits}_0 \,+\, \partial_Yp_0\,=\,0\,,
\end{array}\right.
\end{equation}
where
$$A_0:= - K\Dif_\theta \Cap_0 K \Dif_\theta\,+\,\Gamma_0\,,\quad 
\Cap_0:=\frac{\partial^2 \En}{\partial \rho_x^2}(\rho_0,K \partial_\theta\rho_0)\,,$$
$$\Gamma_0:=\frac{\partial^2 \En}{\partial \rho^2}(\rho_0,K \partial_\theta\rho_0)\,-\,K\,\Dif_{\theta}\left(\frac{\partial^2 \En}{\partial\rho \partial\rho_x}(\rho_0,K \partial_\theta\rho_0)\right)\,,$$
$$B_0:=\frac{\partial^2\En}{\partial\rho\partial\rho_x}(\rho_0,K \partial_\theta\rho_0)\,\partial_X\rho_0\,-\,\partial_X\left(\frac{\partial\En}{\partial\rho_x}(\rho_0,K \partial_\theta\rho_0)\right)\,-\,K\,\Dif_\theta\left(
\Cap_0\,\partial_X\rho_0\right)\,,$$
$$a_0:= - k\Dif_\theta \cap_0 k \Dif_\theta\,+\,\gamma_0\,,\quad 
\cap_0:=\frac{\partial^2 \en}{\partial \vol_y^2}(\vol_0,k \partial_\theta\vol_0)\,,$$
$$\gamma_0:=\frac{\partial^2 \en}{\partial \vol^2}(\vol_0,k \partial_\theta\vol_0)\,-\,k\,\Dif_{\theta}\left(\frac{\partial^2 \en}{\partial\vol \partial\vol_y}(\vol_0,k \partial_\theta\vol_0)\right)\,,$$
$$b_0\,:=\,\frac{\partial^2\en}{\partial\vol\partial\vol_y}(\vol_0,k \partial_\theta\vol_0)\,\partial_Y\vol_0\,-\,\partial_Y\left(\frac{\partial\en}{\partial\vol_y}(\vol_0,k \partial_\theta\vol_0)\right)\,-\,k\,\Dif_\theta\left(
\cap_0\,\partial_Y\vol_0\right)\,.$$
Integrating \eqref{eq:EKabs1d1} and \eqref{eq:EKabsLagb1}
with respect to $\theta$ over $[0,1]$, we get the modulated equations
 \begin{equation}\label{eq:EKW12}
\left\{\begin{array}{l}\partial_T\langle\rho_0\rangle +\partial_X \langle\rho_0 \vits_0\rangle\,=\,0\,,\\ [5pt]
\partial_T \langle\vits_0\rangle + \partial_X\langle\frac{1}{2}\vits_0^2\rangle \,+\,\partial_X \langle\chem_0\rangle\,=\,0\,,
\end{array}\right.
\end{equation}
\begin{equation}\label{eq:EKLagW12}
\left\{\begin{array}{l}\partial_S\langle{\vol}_0\rangle \,-\,\partial_Y\langle{\vitsL}_0\rangle\,=\,0\,,\\ [5pt]
\partial_S \langle{\vitsL}_0\rangle \,+\, \partial_Y\langle p_0\rangle\,=\,0\,.
\end{array}\right.
\end{equation}
(Observe that the terms $B_0$ and $b_0$, the only ones involving cross derivatives of the leading profiles, play no role at this level.)
Now, taking into account the compatibility conditions in \eqref{eq:comp}, we only need to find a fourth averaged equation to have a complete set of modulated equations. We can proceed in two ways. 

As mentioned in the abstract case, the fastest way consists in averaging additional conservation laws satisfied (at least formally) by solutions of \eqref{eq:EKabs1d} and \eqref{eq:EKabsLagb}. Two possible choices are the conservation law of total energy (associated with invariance with respect to time translations, \emph{via} Noether's theorem), and the (local) conservation law of \emph{Benjamin}'s impulse
 (associated with  invariance with respect to spatial translations).
 These conservation laws read
\begin{equation}\label{eq:consEn}
\partial_t(\tfrac{1}{2}\,\rho\,\vits^2+\En)\,+\,\partial_x\left(\tfrac{1}{2}\,\rho\,\vits^3\,+\,\rho\,\vits\,\Euler_\rho \En\,+\,\partial_x(\rho \,\vits)\,\frac{\partial \En}{\partial \rho_x}\right)\,=\,0\,, 
\end{equation}
\begin{equation}\label{eq:consImpulse}
\partial_t(\rho\,\vits)\,+\,\partial_x\left(\rho\,\vits^2\,+\,\rho\, \Euler_\rho \En \,+\,\Legendre_\rho \En\right)\,=\,0\,.
\end{equation}
in the Eulerian framework, and 
\begin{equation}\label{eq:consen}
\partial_s(\tfrac{1}{2}\,\vits^2+\en)\,-\,\partial_y\left(\vits\,\Euler_\vol \en\,+\,(\partial_y\vits)\,\frac{\partial \en}{\partial \vol_y}\right)\,=\,0\,,
\end{equation}
\begin{equation}\label{eq:consimpulse}
\partial_s(\vol\,\vits)\,-\,\partial_y\left(\tfrac{1}{2}\,\vits^2\,+\,\vol \,\Euler_\vol \en \,+\,\Legendre_\vol \en\right)\,=\,0\,.
\end{equation}
in the Lagrangian framework. Upon plugging the asymptotic expansions and averaging we get
\begin{equation}\label{eq:consEnav}
\partial_T\langle \tfrac{1}{2}\,\rho_0\,\vits_0^2+\En_0\rangle \,+\,\partial_X\left\langle \tfrac{1}{2}\,\rho_0\,\vits_0^3\,+\,\rho_0\,\vits_0\,\chem_0\,+\,K\,\partial_\theta (\rho_0 \,\vits_0)\,\frac{\partial \En}{\partial \rho_x}(\rho_0,K\partial_\theta \rho_0)\right\rangle \,=\,0\,, 
\end{equation}
\begin{equation}\label{eq:consImpulseav}
\partial_T\langle \rho_0\,\vits_0\rangle \,+\,\partial_X\left\langle \rho_0\,\vits_0^2\,+\,\rho_0\,\chem_0 \,+\,K\,(\partial_\theta \rho_0)\,\frac{\partial \En}{\partial \rho_x}(\rho_0,K\partial_\theta \rho_0)\,-\,\En(\rho_0,K\partial_\theta \rho_0)
\right\rangle \,=\,0\,
\end{equation}
in the Eulerian framework, and 
\begin{equation}\label{eq:consenav}
\partial_S\langle \tfrac{1}{2}\,\vitsL_0^2+\en_0\rangle \,+\,\partial_Y\left\langle \vitsL_0\,p_0\,-\,k\,(\partial_\theta\vitsL_0)\,\frac{\partial \en}{\partial \vol_y}(\vol_0,k\partial_\theta\vol_0)\right\rangle \,=\,0\,,
\end{equation}
\begin{equation}\label{eq:consimpulseav}
\partial_S\langle \vol_0\,\vitsL_0\rangle \,+\,\partial_Y\left\langle -\,\tfrac{1}{2}\,\vitsL_0^2\,+\,\vol_0 \,p_0\,+\,\en(\vol_0,k\partial_\theta\vol_0) \,-\,k\,(\partial_\theta \vol_0)\,\frac{\partial \en}{\partial \vol_y}(\vol_0,k\partial_\theta\vol_0)\right\rangle \,=\,0\,
\end{equation}
in the Lagrangian one.
For simplicity, we have denoted by $\En_0$ and $\en_0$ the energies evaluated at $(\rho_0,K\partial_\theta \rho_0)$ and 
$(\vol_0,k\partial_\theta\vol_0)$ respectively.
At first glance it may look like we have five modulated equations in each framework, 
namely the compatibility condition in \eqref{eq:comp}, the two equations in 
either \eqref{eq:EKW12} or \eqref{eq:EKLagW12}, as well as 
\eqref{eq:consEnav}-\eqref{eq:consImpulseav} or \eqref{eq:consenav}-\eqref{eq:consimpulseav}.
In fact, the profile equations in \eqref{eq:EKW0} imply that the averaged energy equation
\eqref{eq:consEnav} is a consequence of \eqref{eq:comp} and \eqref{eq:EKW12}~\eqref{eq:consImpulseav}, and similarly 
 \eqref{eq:EKLagW0}-\eqref{eq:EKLagW12}-\eqref{eq:consimpulseav} imply \eqref{eq:consenav}.

A more indirect way to derive a fourth modulated equation consists in taking the inner product of \eqref{eq:EKabs1d1} and \eqref{eq:EKabsLagb1} with $(\vits_0,\rho_0)$ and $(\vitsL_0,\vol_0)$ respectively. Accordingly with the abstract case considered in Section \ref{s:permod}, those choices are  dictated by Benjamin's impulses $\rho\vits$ and $\vol\vits$, of which the variational derivatives are respectively $(\vits,\rho)^t$ and $(\vits,\vol)^t$, see 
Appendix \ref{concrete} for more details.

To summarize, we have the following.
\begin{proposition}  Whitham's modulated equations associated with \eqref{eq:EKabs1d} and \eqref{eq:EKabsLagb}  read, respectively,
\begin{itemize}
\item in the Eulerian framework
\begin{equation}\label{eq:EKW14}
\left\{\begin{array}{l}
\partial_T K \,+\,\partial_X (\sigma K)\,=\,0\,,\\ [5pt]
\partial_T\langle\rho_0\rangle +\partial_X \langle\rho_0 \vits_0\rangle\,=\,0\,,\\ [5pt]
\partial_T \langle\vits_0\rangle + \partial_X\langle\frac{1}{2}\vits_0^2\rangle \,+\,\partial_X \langle\chem_0\rangle\,=\,0\,, \\ [5pt]
\partial_T\langle \rho_0\,\vits_0\rangle \,+\,\partial_X\left\langle \rho_0\,\vits_0^2\,+\,\rho_0\,\chem_0 \,+\,K\,(\partial_\theta \rho_0)\,\dfrac{\partial \En}{\partial \rho_x}(\rho_0,K\partial_\theta \rho_0)\,-\,\En_0
\right\rangle \,=\,0\,,
\end{array}\right.
\end{equation}
which is endowed with the additional conservation law
\begin{equation}\label{eq:consEnW}
\partial_T\langle \tfrac{1}{2}\,\rho_0\,\vits_0^2+\En_0\rangle \,+\,\partial_X\left\langle \tfrac{1}{2}\,\rho_0\,\vits_0^3\,+\,\rho_0\,\vits_0\,\chem_0\,+\,K\,\partial_\theta (\rho_0 \,\vits_0)\,\dfrac{\partial \En}{\partial \rho_x}(\rho_0,K\partial_\theta \rho_0)\right\rangle \,=\,0\,,
\end{equation}

\item and in the Lagrangian framework
\begin{equation}\label{eq:EKLagW14}\left\{\begin{array}{l}
\partial_S k \,-\,\partial_Y (jk)\,=\,0\,,\\ [5pt]
\partial_S\langle{\vol}_0\rangle \,-\,\partial_Y\langle{\vitsL}_0\rangle\,=\,0\,,\\ [5pt]
\partial_S \langle{\vitsL}_0\rangle \,+\, \partial_Y\langle p_0\rangle\,=\,0\,, \\ [5pt]
\partial_S\langle \vol_0\,\vitsL_0\rangle \,+\,\partial_Y\left\langle -\,\tfrac{1}{2}\,\vitsL_0^2\,+\,\vol_0 \,p_0\,+\,\en_0 \,-\,k\,(\partial_\theta \vol_0)\,\dfrac{\partial \en}{\partial \vol_y}(\vol_0,k\partial_\theta\vol_0)\right\rangle \,=\,0\,,
\end{array}\right.
\end{equation}
which is endowed with the additional conservation law
\begin{equation}\label{eq:consenW}
\partial_S\langle \tfrac{1}{2}\,\vitsL_0^2+\en_0\rangle \,+\,\partial_Y\left\langle \vitsL_0\,p_0\,-\,k\,(\partial_\theta\vitsL_0)\,\frac{\partial \en}{\partial \vol_y}(\vol_0,k\partial_\theta\vol_0)\right\rangle \,=\,0\,.
\end{equation}
\end{itemize}
\end{proposition}

Now, we may go further and make the link between the Eulerian \eqref{eq:EKW14} and the Lagrangian \eqref{eq:EKLagW14}
modulated equations. Interestingly, even though it is all but a surprise, 
\eqref{eq:EKLagW14} can be viewed as a Lagrangian reformulation of \eqref{eq:EKW14}.
More precisely, we are going to show the following.

\begin{theorem}
\label{thm:EL}
Away from vacuum, there is a mass Lagrangian system of coordinates changing 
System \eqref{eq:EKW14} into \eqref{eq:EKLagW14}.  In particular, these systems are simultaneously hyperbolic.
A sufficient condition for the hyperbolicity of \eqref{eq:EKW14} and \eqref{eq:EKLagW14} is that 
$$\meane:=\langle \en_0\rangle\,+\,\tfrac{1}{2}\,\langle \vitsL_0^2\rangle\,-\,\tfrac{1}{2}\,\langle \vitsL_0\rangle^2$$ be a strictly convex function of 
$(\meanvol, k, \Delta/k)$, or equivalently that $$\meanE:=\langle\rho_0\rangle \,\meane\,=\,\langle \En_0\rangle\,+\,\frac{1}{2}\langle\rho_0\vits_0^2\rangle
\,-\,\frac{1}{2}\frac{\langle \rho_0\vits_0\rangle^2}{\langle\rho_0\rangle}$$ 
be a strictly convex function of $(\meanrho,K,D/K)$,
 where 
 $$\meanvol:= \langle \vol_0\rangle\,,\quad 
 \Delta:=\langle \vol_0\vitsL_0\rangle \,-\,\langle \vol_0\rangle\,\langle \vitsL_0\rangle\,,$$
 $$\meanrho:=\langle\rho_0\rangle\,,\quad D:= \langle\rho_0\rangle\langle\vits_0\rangle-\langle\rho_0\vits_0\rangle\,.$$
\end{theorem}

\begin{proof}
Let us recall that $(\Rho,\Vits)$ is a $1/K$-periodic solution of \eqref{eq:profE} if and only if
$(\Vol,\VitsL)$ is a $1/k$-periodic solution of \eqref{eq:profL}, with
$$\Rho(\xi) \Vol(\Y(\xi))\,=\,1\,,\;\Vits(\xi)=\VitsL(\Y(\xi))\,,
\dfrac{\dif \Y}{\dif \xi}\,=\,\Rho\,=\,\dfrac{1}{\Vol(\Y)}\,, \quad \frac{1}{k}=\Y\Big(\frac{1}{K}\Big)\,.$$ 
This implies in particular that 
\begin{equation}\label{eq:linkEL}
\langle \rho_0\rangle = \frac{K}{k}\,,\;\langle \vol_0\rangle = \frac{k}{K}\,,\quad
\langle \vol_0\rangle = \frac{1}{\langle\rho_0\rangle}\,,\quad \langle\vitsL_0\rangle= \frac{\langle \rho_0\vits_0\rangle}{\langle\rho_0\rangle}\,
\,.
\end{equation}
These observations enable us to make the relationship between \eqref{eq:EKW14} and \eqref{eq:EKLagW14} in the same way as it is usually done between the fluid equations in Eulerian coordinates and those in mass Lagrangian coordinates.
As a matter of fact, the second equation in \eqref{eq:EKW14} states that 
$\langle \rho_0\rangle \dif X\,-\,\langle \rho_0\vits_0\rangle \dif T$ is a closed differential form, and thus an exact form in a simply connected domain. We thus introduce the `rescaled mass Lagrangian coordinate' $Y$ defined (up to a constant) by
$$\dif Y\,=\,\langle \rho_0\rangle \dif X\,-\,\langle \rho_0\vits_0\rangle \dif T\,.$$
Setting $S=T$, this equivalently reads thanks to the last two relations in \eqref{eq:linkEL},
$$\dif X\,=\,\langle \vol_0\rangle \dif Y\,+\,\langle \vitsL_0\rangle \dif S\,,$$
hence the second equation in \eqref{eq:EKLagW14}. We can proceed similarly with the other equations.
The first one in \eqref{eq:EKW14} states that 
$$K\,\dif X \,-\,\sigma K\,\dif T\,=\,
\langle\vol_0\rangle K \,\dif Y\,+\,K\,(\langle \vitsL_0\rangle -\sigma )\,\dif S$$
is an exact differential form. Using the second relation in \eqref{eq:linkEL} and that 
$\langle \vitsL_0\rangle -\sigma = j \langle \vol_0\rangle$, we get that 
$k\dif Y + jk\dif S$ is exact, hence the 
 first equation in \eqref{eq:EKLagW14}.
 As regards the other equations, the third one in  \eqref{eq:EKW14} gives the fourth one in  \eqref{eq:EKLagW14},
 and the fourth one in \eqref{eq:EKW14} gives the third one in  \eqref{eq:EKLagW14}
 (this interplay between momentum and velocity equations is well-known for standard fluids motion, in which the conservation law of the momentum $\rho \vits$ in Eulerian coordinates is associated with a conservation law for the velocity in Lagrangian coordinates, the other way round going from a conservation law for $\vol \vits$ in Lagrangian coordinates to a conservation law for $\vits$ in Eulerian coordinates being less classical but still true, as far as smooth solutions are concerned). In order to justify the correspondence, it is convenient to rewrite these equations in a simpler way. This is done by using the profile equations, which give
$$\rho_0(\vits_0-\sigma)=j\,,\;\chem_0\,=\,\mu\,-\,\frac{j^2}{2\rho_0^2}\,,\;K\,(\partial_\theta \rho_0)\,\dfrac{\partial \En}{\partial \rho_x}(\rho_0,K\partial_\theta \rho_0)\,-\,\En_0\,=\,-\,\frac{j^2}{2\rho_0}\,-\,\mu\,\rho_0\,-\,\lambda\,,$$
$$ \vitsL_0-j\vol_0=\sigma\,,\;
\press_0\,=\,-\lambda\,-\,j^2\,\vol_0\,,\;\en_0 \,-\,k\,(\partial_\theta \vol_0)\,\dfrac{\partial \en}{\partial \vol_y}(\vol_0,k\partial_\theta\vol_0)\,=\,\tfrac{1}{2}\,j^2\vol_0^2\,+\,\lambda\,\vol_0\,+\,\mu\,,
$$
so that the last two equations in \eqref{eq:EKW14} and  \eqref{eq:EKLagW14} respectively read
\begin{equation}\label{eq:EKW34}
\left\{\begin{array}{l}
\partial_T \langle\vits_0\rangle + \partial_X
\left(\mu -\,\tfrac{1}{2}\,\sigma^2\,+\sigma \langle \vits_0\rangle \right)\,=\,0\,, \\ [5pt]
\partial_T\langle \rho_0\,\vits_0\rangle \,+\,\partial_X
\left(\left\langle \rho_0\,\vits_0^2\right\rangle\,-\,\lambda\,-\,j^2\,\left\langle \dfrac{1}{\rho_0}\right\rangle\right) \,=\,0\,,
\end{array}\right.
\end{equation}
\begin{equation}\label{eq:EKLagW34}\left\{\begin{array}{l}
\partial_S \langle{\vitsL}_0\rangle \,-\, \partial_Y(\lambda\,+\,j^2\,\langle\vol_0 \rangle)\,=\,0\,, \\ [5pt]
\partial_S\langle \vol_0\,\vitsL_0\rangle \,+\,\partial_Y
\left(\mu\,-\,\tfrac{1}{2}\sigma^2\,-\,j\,\sigma\,\langle\vol_0\rangle\,-\,j^2\,\langle\vol_0^2\rangle\right)
 \,=\,0\,.
\end{array}\right.
\end{equation}
The first equation in \eqref{eq:EKW34} means that 
$$\langle \vits_0\rangle \dif X\,+\,\left(\tfrac{1}{2}\,\sigma^2\,-\,\mu -\,\sigma \langle \vits_0\rangle \right)\,\dif T=
\langle \vol_0\vitsL_0\rangle \dif Y\,+\,\left(\tfrac{1}{2}\,\sigma^2\,-\,\mu\,+\,(\langle \vitsL_0\rangle\, -\,\sigma) \,\frac{\langle \vol_0\vitsL_0\rangle}{\langle\vol_0\rangle} \right)\,\dif S$$
$$=\langle \vol_0\vitsL_0\rangle \dif Y\,+\,\left(\tfrac{1}{2}\,\sigma^2\,-\,\mu\,+\,j\,\sigma\,\langle\vol_0\rangle\,+\,j^2\,\langle\vol_0^2\rangle\right)\,\dif S $$ is exact, which is equivalent to the second conservation law in \eqref{eq:EKLagW34}.
The second equation in  \eqref{eq:EKW34} means that 
$$\langle \rho_0\,\vits_0\rangle\,\dif X \,+\,
\left(-\,\left\langle \rho_0\,\vits_0^2\right\rangle\,+\,\lambda\,+\,j^2\,\left\langle \dfrac{1}{\rho_0}\right\rangle\right)\,\dif T\,=\,$$
$$\langle \vitsL_0\rangle \,\dif Y\,+\,\left(-\,\left\langle \rho_0\,\vits_0^2\right\rangle\,+\,\lambda\,+\,j^2\,\left\langle \dfrac{1}{\rho_0}\right\rangle\right)\,\dif S
=\langle \vitsL_0\rangle \,\dif Y\,+\,(\lambda\,+\,j^2\,\langle \vol_0\rangle^2)\,\dif S
$$
is exact, which is equivalent to the first conservation law in \eqref{eq:EKW34}.
This finishes to prove the equivalence between the two modulated systems \eqref{eq:EKW14} and \eqref{eq:EKLagW14}
(as long as $\langle \rho_0\rangle$ and $\langle \vol_0\rangle$ do not vanish). As a consequence, these two first order systems of conservation laws are simultaneously hyperbolic. A sufficient condition for the hyperbolicity of  \eqref{eq:EKLagW14} has been pointed out by Gavrilyuk and Serre \cite{GavrilyukSerre} in terms of the average energy
$$\meane:=\langle \en_0\rangle\,+\,\tfrac{1}{2}\,\langle \vitsL_0^2\rangle\,-\,\tfrac{1}{2}\,\langle \vitsL_0\rangle^2\,,$$ 
which satisfies the generalized \emph{Gibbs relation}
\begin{equation}\label{eq:gibbs}
\dif \meane \,=\,-\,\meanpress \dif \meanvol\,+\,\Theta\,\dif k\,+\,j\,\dif \Delta\,,
\end{equation}
with $\meanpress:=\langle \press_0\rangle$, $\meanvol:=\langle \vol_0\rangle$, and
\begin{equation}\label{def:ThetaDelta}
\Theta:=\left\langle(\partial_\theta\vol_0)\,\dfrac{\partial \en}{\partial \vol_y}(\vol_0,k\partial_\theta\vol_0)\right\rangle\,,
\quad \Delta:=\langle \vol_0\vitsL_0\rangle \,-\,\langle \vol_0\rangle\,\langle \vitsL_0\rangle\,.
\end{equation}
Namely, it is shown in  \cite{GavrilyukSerre} (alternatively, the reader might take a look at \cite{SBG-DIE})  that, if $\meane$ is a strictly convex function of 
$(\meanvol, k, \Delta/k)$ then \eqref{eq:EKLagW14} is Godunov-symmetrizable and thus hyperbolic.
We may note in addition that, by a standard argument on convex functions (viewed as supremum envelopes of affine functions), the convexity of $\meane$ as a function of 
$(\meanvol, k, \Delta/k)$ is equivalent to the convexity of $\langle\rho_0\rangle \,\meane$ as a function of
$(\langle\rho_0\rangle, k\langle\rho_0\rangle, \langle\rho_0\rangle\Delta/k)$. An easy calculation (using relations in 
\eqref{eq:linkEL} and similar ones) shows that
$$\langle\rho_0\rangle \,\meane\,=\,\langle \En_0\rangle\,+\,\frac{1}{2}\langle\rho_0\vits_0^2\rangle
\,-\,\frac{1}{2}\frac{\langle \rho_0\vits_0\rangle^2}{\langle\rho_0\rangle}\,=:\,\meanE\,,
\;k\langle\rho_0\rangle\,=\,K\,,\;\langle\rho_0\rangle\,\frac{\Delta}{k}\,=\,\frac{1}{K}\,(\langle\rho_0\rangle\langle\vits_0\rangle-\langle\rho_0\vits_0\rangle)\,.$$
This shows that a sufficient condition for the hyperbolicity of \eqref{eq:EKW14} is the strict convexity of 
$\meanE$ as a function of $(\meanrho,K,D/K)$ where $\meanrho:=\langle\rho_0\rangle$, 
$ D:= \langle\rho_0\rangle\langle\vits_0\rangle-\langle\rho_0\vits_0\rangle$.
Of course our generalised Gibbs relation has its counterpart in terms of $\meanE$. It reads
\begin{equation}\label{eq:Gibbs}
\dif \meanE \,=\,\meanchem \dif \meanrho\,+\,\Theta\,\dif K\,+\,\frac{j}{\meanrho}\,\dif D\,,
\end{equation}
where 
\begin{equation}\label{eq:defmeanchem}
\meanchem\,:=\,\langle \chem_0\rangle \,+\,\frac{j^2}{2}\,\left(\left\langle \frac{1}{\rho_0^2}\right\rangle\,-\,\frac{1}{\meanrho^2}\right)\,-\,\frac{j\,D}{\meanrho^2}\,.
\end{equation}
Equation \eqref{eq:Gibbs} can be derived from \eqref{eq:gibbs} as follows. Using that 
$$\meanE= \meanrho\, \meane\,,\;\meanvol=1/\meanrho\,,\;K=k\meanrho\,,\;D\,=\,\meanrho^2\Delta\,,$$
we readily get that 
$$\dif \meanE\,=\,\left(\meane\,+\,\meanpress\,\meanvol\,-\,\frac{K\,\Theta}{\meanrho}\,-\,2\,\frac{j\,D}{\meanrho^2}\right)\,\dif \meanrho\,+\,\Theta\,\dif K\,+\,\frac{j}{\meanrho}\,\dif D\,.$$
It just remains to express the factor of $\dif \meanrho$ in terms of `Eulerian' mean values. Indeed, we can see the `temperature' $\Theta$ either as a `Lagrangian' mean value (by definition, see Eq.~\eqref{def:ThetaDelta}) or as an `Eulerian' one: it turns out that 
$$\Theta\,=\,\left\langle(\partial_\theta \rho_0)\,\dfrac{\partial \En}{\partial \rho_x}(\rho_0,K\partial_\theta \rho_0)\right\rangle\,.$$
As a matter fact, by the Eulerian profile equation we have
$$K\,\left\langle(\partial_\theta \rho_0)\,\dfrac{\partial \En}{\partial \rho_x}(\rho_0,K\partial_\theta \rho_0)\right\rangle\,=\,
\left\langle\En_0\,-\,\frac{j^2}{2\rho_0}\,-\,\mu\,\rho_0\,-\,\lambda\right\rangle\,=\,
\meanrho \,\left\langle\en_0\,-\,\frac{j^2\vol_0^2}{2}\,-\,\mu\,-\,\lambda\,\vol_0\right\rangle$$
by the relationship between Eulerian mean values and Lagrangian mean values we already used several times, hence by the Lagrangian profile equation
$$K\,\left\langle(\partial_\theta \rho_0)\,\dfrac{\partial \En}{\partial \rho_x}(\rho_0,K\partial_\theta \rho_0)\right\rangle\,=\,
\meanrho \,k\,\left\langle(\partial_\theta\vol_0)\,\dfrac{\partial \en}{\partial \vol_y}(\vol_0,k\partial_\theta\vol_0)\right\rangle\,=\,K\,\Theta\,.$$
We can now work out the factor of $\dif \meanrho$ by observing that 
$$j\,\Delta\,=\,\langle  \vitsL_0 j (\vol_0-\meanvol)\rangle\,=\,\langle  \vitsL_0 (\vitsL_0-\langle\vitsL_0\rangle)\rangle\,=\,\langle  \vitsL_0^2\rangle\,-\,\langle  \vitsL_0\rangle^2\,,$$ 
so that by definition
$$\meane\,=\,\langle \en_0\rangle\,+\,\frac{j\Delta}{2}\,.$$
In addition we have that 
$$\meanpress\,\meanvol\,=\,\langle \press_0\,\vol_0\rangle\,+\,j\Delta$$
since
$$\langle \press_0\,\vol_0\rangle\,=\,\langle (\meanpress \,-\,j^2\,(\vol_0\,-\,\meanvol))\,\vol_0\rangle\,=\,
\meanpress\,\meanvol\,-\,\langle j^2 (\vol_0\,-\,\meanvol)^2\rangle\,=\,\meanpress\,\meanvol\,-\, \langle(\vitsL_0\,-\langle\vitsL_0\rangle)^2\rangle\,=\,\meanpress\,\meanvol\,-\,j\Delta\,.
$$
Therefore,
$$\meane\,+\,\meanpress\,\meanvol\,=\,\langle \en_0\,+\,\press_0\,\vol_0\rangle\,+\,\frac{3j\Delta}{2}\,,$$
and by our usual trick,
$$\langle \en_0\,+\,\press_0\,\vol_0\rangle\,=\,\frac{1}{\meanrho}\,\langle \En_0\,+\,\press_0\rangle\,,$$
where $\press_0$ has to be expressed in Eulerian coordinates in the right-hand side. This amounts to writing
$$\langle \En_0\,+\,\press_0\rangle\,=\,\left\langle \En_0\,-\,\lambda\,-\,\frac{j^2}{\rho_0}\right\rangle\,=\,
K\,\Theta\,-\,\frac{j^2}{2}\,\left\langle \frac{1}{\rho_0}\right\rangle\,+\,\mu\,\meanrho=\,K\,\Theta\,-\,\frac{j^2}{2}\,\left\langle \frac{1}{\rho_0}\right\rangle\,+\,\meanrho\,\left\langle\chem_0 \,+\,\frac{j^2}{2\rho_0^2}\right\rangle$$
using once more the profile equations. So we have
$$\meane\,+\,\meanpress\,\meanvol\,-\,\frac{K\,\Theta}{\meanrho}\,=\,\langle \chem_0\rangle\,+\,\frac{j^2}{2}\,\left(\left\langle \frac{1}{\rho_0^2}\right\rangle\,-\,\frac{1}{\meanrho}\left\langle \frac{1}{\rho_0}\right\rangle\right)
 \,+\,\frac{3j\Delta}{2}\,,$$
 and noting that
 $$\frac{j^2}{\meanrho}\left\langle \frac{1}{\rho_0}\right\rangle\,=\,\langle j^2 \vol_0^2\rangle\,=\,
 j^2\,\meanvol^2\,+\,j\,\Delta\,,$$
 we eventually obtain the claimed formula \eqref{eq:Gibbs} with $\meanchem$ defined by \eqref{eq:defmeanchem}.
 \end{proof}

Still, it is not obvious at this stage that \eqref{eq:EKW14} or \eqref{eq:EKLagW14} 
are really evolution systems in closed form. 
We shall observe on numerical experiments that evolutionarity may indeed fail.
For the moment, let us just note that if evolutionarity happens to fail, it does so `\emph{simultaneously}' for  \eqref{eq:EKW14} and \eqref{eq:EKLagW14}.
As a matter of fact, at fixed $(T,X)$, we know that 
$$(\rho_0,\vits_0)(T,X,\theta)=(\Rho,\Vits)(T,X,\theta/K(T,X))\,$$
where the profile
$(\Rho,\Vits)$ is a $1/K$-periodic solution to
\begin{equation}\label{eq:profE}\left\{\begin{array}{l}\Rho\,(\,\Vits\,-\,\sigma)\,\equiv\,j\,,\\ [10pt]
\Rho_\xi\,\dfrac{\partial \En }{\partial \rho_x}(\Rho,\Rho_\xi)\,-\,\En(\Rho,\Rho_\xi)\,+\,\dfrac{j^2}{2\Rho}\,+\,\mu\,\Rho\,\equiv\,-\,\lambda\,.\end{array}\right.
\end{equation}
Here above, the sign $\equiv$ means equalities for functions of $(T,X)$ only.
The quadruplet $(j,\sigma,\lambda,\mu)$ is a natural set of parameters in \eqref{eq:profE}, which we expect to determine properly 
 the wave number $K$ as well as the $1/K$-periodic solution $(\Rho,\Vits)$ to \eqref{eq:profE}, up to translations, hence also all mean values involved in 
\eqref{eq:EKW14}~\eqref{eq:consEnW}.
Similarly,
$$(\vol_0,\vitsL_0)(S,Y,\theta)=(\Vol,\VitsL)(S,Y,\theta/k(S,Y))\,$$
and all mean values involved in \eqref{eq:EKLagW14}~\eqref{eq:consenW}
are expected to be fully determined by $(j,\sigma,\lambda,\mu)$
through the profile equations
\begin{equation}\label{eq:profL}\left\{\begin{array}{l}\VitsL\,-\,j\,\Vol\,\equiv\,\sigma\,,\\ [10pt]
\Vol_\zeta\,\dfrac{\partial \en }{\partial \vol_y}(\Vol,\Vol_\zeta)\,-\,\en(\Vol,\Vol_\zeta)\,+\,\dfrac{j^2\Vol^2}{2}\,+\,\lambda\,\Vol\,\equiv\,-\,\mu\,.\end{array}\right.
\end{equation}
By the one-to-one correspondence we have pointed out between Eulerian and Lagrangian periodic orbits,
the mapping $(j,\sigma,\lambda,\mu)\mapsto (k,\langle \Vol\rangle, \langle \VitsL\rangle, \langle \Vol\VitsL\rangle)$ will be locally invertible if and only if $(j,\sigma,\lambda,\mu)\mapsto (K,\langle \Rho\rangle, \langle \Vits\rangle, \langle \Rho\Vits\rangle)$ is so.
Again, this is not exactly always the case, as we shall see on specific examples in the next section.

\section{Nature of modulated equations in practice}\label{s:num}

Our purpose here is to investigate the hyperbolicity of Whitham's modulated equations for the Euler--Korteweg system.
First of all, let us point out that the parameter $\sigma$ does not play any role in that matter.  This is due to Galilean invariance of the Euler--Korteweg system,  unsurprisingly. To be more precise, let us rewrite the modulated equations in natural coordinates from the `thermodynamical' point of view. In the Lagrangian framework they read
\begin{equation}\label{eq:EKLagW14th}\left\{\begin{array}{l}
\partial_S k \,-\,\partial_Y (jk)\,=\,0\,,\\ [5pt]
\partial_S\meanvol \,-\,\partial_Y\meanvitsL\,=\,0\,,\\ [5pt]
\partial_S \meanvitsL  \,+\, \partial_Y\meanpress\,=\,0\,, \\ [5pt]
\partial_S(\Delta+\meanvol\meanvitsL) \,+\,\partial_Y\left(\meane + \meanvol\meanpress- k\Theta -\frac{1}{2} {\meanvitsL}^2 -2 j\Delta\right) \,=\,0\,,
\end{array}\right.
\end{equation}
together with the generalized Gibbs relation \eqref{eq:gibbs} $\dif \meane \,=\,-\,\meanpress \dif \meanvol\,+\,\Theta\,\dif k\,+\,j\,\dif \Delta$. We easily check that \eqref{eq:EKLagW14th} is invariant by translations of the form $$(S,Y,k,\meanvol,\meanvitsL,\Delta)\mapsto
(S,Y,k,\meanvol,\meanvitsL-\underline{\sigma},\Delta)$$ for any (constant) velocity $\underline{\sigma}$.
Substituting $\sigma+j\meanvol$ for $\meanvitsL$, we thus see that the equations  in \eqref{eq:EKLagW14th} are unchanged if
$\sigma$ is replaced by $\sigma-\underline{\sigma}$. Therefore, the eigenvalues of the linearized equations about a reference `state'
$(\underline{k},\underline{\meanvol},\underline{\sigma},\underline{\Delta})$, and those about the translated one $(\underline{k},\underline{\meanvol},0,\underline{\Delta})$ coincide.
A similar argument works in the Eulerian framework too, see Appendix \ref{s:gal} for more details.

From now on, we concentrate on the Euler--Korteweg system in mass Lagrangian coordinates \eqref{eq:EKabsLagb}, 
with an energy of the form \eqref{eq:enK}. We recall from \eqref{eq:profL}
 that a periodic traveling wave solution $(\vol,\vits)=(\Vol,\VitsL)(y+jt)$ to \eqref{eq:EKabsLagb}-\eqref{eq:enK}
is characterized by a periodic profile $(\Vol,\VitsL)$ that must be solution of
\begin{equation}\label{EKsw_prof}
\displaystyle
\VitsL-j\Vol=\sigma,\quad \tfrac{1}{2}\cap(\Vol)\Vol_\zeta^2+\tfrac{1}{2}{j^2}\Vol^2+\lambda \Vol
-f(\Vol)=-\mu,
\end{equation}
where $\sigma$,
$\lambda$, and $\mu$ are constant of integrations. 

\begin{remark}
The phase portrait of the ODE on $\Vol$ in  \eqref{EKsw_prof} does not depend on the specific capillarity function $\vol\mapsto\cap(\vol)$, provided that it stays positive. In fact, up to a rescaling in $\zeta$,
that ODE reduces to
$$\tfrac{1}{2} \dot\Vol^2\,=\,f(\Vol) \,-\,\tfrac{1}{2}j^2\Vol^2 \,-\,\lambda\,\Vol\,-\,\mu\,.$$
Remarkably enough, this equation is  also the integrated profile equation for traveling wave solutions of speed $-j^2$ to the 
\emph{generalized Korteweg--de Vries} equation (gKdV)
$$\partial_t \vol+ \partial_x \press(\vol)= -\partial_x^3\vol\,.$$
(This relationship between the traveling waves of the Euler--Korteweg equations in Lagrangian coordinates and those of the generalized Korteweg--de Vries equation has been known for a long time, and is investigated in more details for instance in \cite{Hoewing}.)
However, when one turns to specific examples for $\press$ in the generalized Korteweg--de Vries equation, it is most often to consider power laws $\press=\vol^{\gamma}$. By contrast, we consider here laws that go to infinity at zero, or more generally at some \emph{co-volume} $b$, and to zero at infinity.
\end{remark}

Using in particular the first equation in \eqref{EKsw_prof} above to substitute 
$\sigma+j\Vol$ for $\VitsL$ in \eqref{eq:EKLagW14}, and reformulating  the last two equations in  \eqref{eq:EKLagW14} as in 
\eqref{eq:EKLagW34}, we can write the modulated equations as
{\setlength\arraycolsep{1pt}
\begin{eqnarray}
\label{EKsw_WH1}
\displaystyle
\partial_S k-\partial_Y(jk)&=&0,\\
\label{EKsw_WH2}
\displaystyle
\partial_S\langle \Vol \rangle-\partial_Y(j\langle \Vol\rangle+\sigma)&=&0,\\
\label{EKsw_WH3}
\displaystyle
\partial_S(j\langle \Vol\rangle+\sigma)-\partial_Y(\lambda+j^2\langle \Vol\rangle)&=&0,\\
\label{EKsw_WH4}
\displaystyle
\partial_S(\sigma\langle \Vol\rangle+j\langle \Vol^2\rangle)-\partial_Y(\tfrac{1}{2}\sigma^2+j\sigma\langle \Vol\rangle+j^2\langle \Vol^2\rangle-\mu)&=&0.
\end{eqnarray}}
This is just an alternative formulation of \eqref{eq:EKLagW14th} in terms of the `natural' parameters $(\sigma,j,\lambda,\mu)$. 
Our main purpose is to investigate the hyperbolicity of (\ref{EKsw_WH1}-\ref{EKsw_WH4}), which by Theorem \ref{thm:absstab} is a necessary condition for stability of the periodic wave, provided that nearby periodic waves be parametrized by 
$(k,\langle \Vol \rangle,j\langle \Vol\rangle+\sigma,\sigma\langle \Vol\rangle+j\langle \Vol^2\rangle)$. We can hardly check these properties -~evolutionarity and hyperbolicity~- analytically, since neither the wave number $k$ nor the wave profile $\Vol$ is known explicitly in terms of  
$(\sigma,j,\lambda,\mu)$. However, it is not difficult to check them numerically. To make numerical computations more transparent,  we are going to change 
$(\lambda,\mu)$ for more convenient parameters, under suitable assumptions on the pressure law $\vol\mapsto \press(\vol)=-f'(\vol)$. In this respect, we shall start with shallow water type pressure laws, and turn to the more complicated, Van der Waals type pressure laws afterward. 
We shall see that the form of the capillarity $\cap(\vol)$ also plays a role -~if a rescaling along orbits does not change the phase portrait, it does change mean values. Our minimal assumption will be that $\cap:(b,+\infty)  \to  (0,+\infty)$ is a smooth function for some nonnegative `co-volume' $b$.

Before going into specific examples, let us say a few words about two asymptotic limits, namely the small amplitude regime, and the solitary wave limit.
The former has been analyzed in some detail in \ref{ss:small}. In particular, it has been pointed out that a necessary condition for the hyperbolicity of modulated equations about a small amplitude periodic wave for the Euler--Korteweg system is that the Euler equations be hyperbolic at the mean value of this wave.
This is why we shall not try to get too close to center points of the wave profile equations in the numerical computations that follow: we readily know that the small amplitude wave trains about those points where the Euler equations are not hyperbolic are unstable.
(As to the other necessary condition, namely the positivity of  $\omega_2\,\omega_{0,kk}$, we leave it aside for the moment.)
The solitary wave limit corresponds to when the wavenumber $k$ tends to zero. As noticed in earlier work (see for instance \cite{El1,El2}) for which explicit computations can be made~--~involving elliptic integrals~--~, we expect modulational instability for waves of small wavenumber when the endpoint of the limiting solitary wave is an unstable state of the Euler equations. Thus we shall not try, in numerical computations, to get too close to solitary wave orbits either.

\subsection{Shallow water type pressure laws}

\subsubsection{Parametrization of periodic waves}

Here we assume that 
$\press:(b,+\infty)  \to  (0,+\infty)$ is smooth and \emph{strictly convex}, with
$$\lim_{\vol \searrow b} \press(\vol)=+\infty\,,\quad \lim_{\vol \to +\infty} \press(\vol)=0\,,$$
hence in particular $\press$ is monotonically decaying to zero at infinity.
The shallow water case $\press(\vol)=1/\vol^2$ is the main application we have in mind.
If $\vol_\infty$ is to denote the endpoint of a solitary wave with $j\neq 0$ and $\lambda$ as constants of integration, it must satisfy
$$
j^2< -p'(\vol_\infty)\,,\quad \lambda=-j^2 \vol_\infty-\press(\vol_\infty)\,,
$$
in which case there is exactly one other solution $\vol_0\in (\vol_\infty,+\infty)$ to $\lambda=-j^2 \vol-\press(\vol)\,.$
Now, inside the homoclinic loop connecting $\vol_\infty$ to itself, there is a collection of periodic orbits surrounding $\vol_0$, which are determined for instance by their trough $\vol_*$  (see Figure \ref{fig:convexportrait}). More precisely, if $\vol_*\in (\vol_\infty,\vol_0)$ there is a unique periodic orbit passing through $\vol_*$ and
solving the ODE 
$$\tfrac{1}{2}\cap(\Vol)\Vol_\zeta^2+\tfrac{1}{2}{j^2}\Vol^2+\lambda \Vol
-f(\Vol)=-\mu,\quad \mu:=f(\vol_*)\,-\,\tfrac{1}{2}{j^2}\vol_*^2-\lambda \vol_*\,.$$

\begin{figure}[ht]\label{fig:convexportrait}
\begin{center}
\includegraphics[height=9cm]{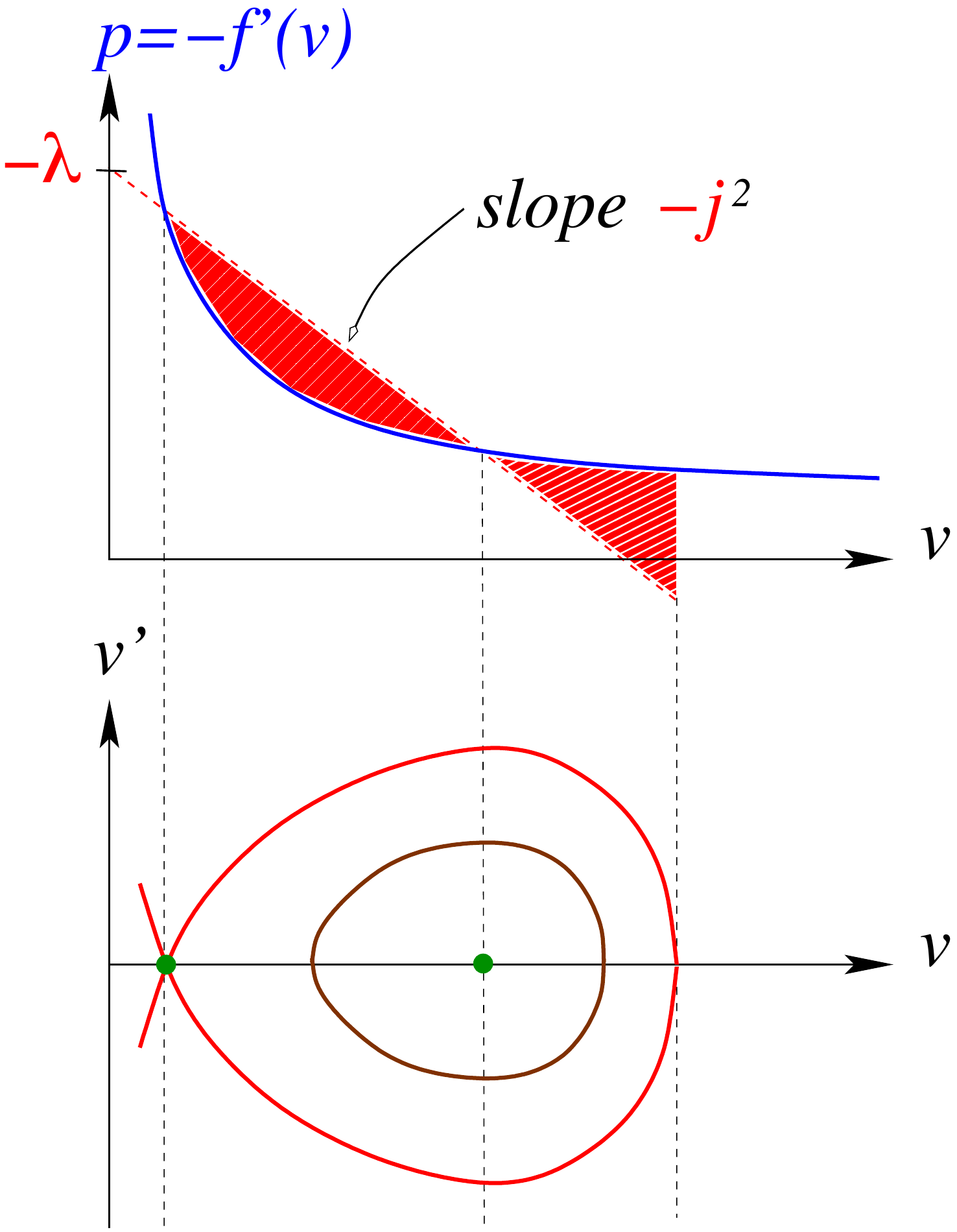}
\end{center}
\caption{Phase portrait for convex pressure law}
\end{figure}
\noindent
Let us consider the mapping
{\setlength\arraycolsep{1pt}
\begin{eqnarray*}
\displaystyle
\Phi : \Omega\subset\mathbb{R}^3& \to &\mathbb{R}^3\\
\displaystyle
(j,\vol_\infty, \vol_*)&\mapsto & (j,\:\lambda=-j^2 \vol_\infty-\press(\vol_\infty),\:\mu=f(\vol_*)\,-\,\tfrac{1}{2}{j^2}\vol_*^2-\lambda \vol_*),
\end{eqnarray*}}
\noindent
with 
$$
\displaystyle
\Omega=\left\{(j,v_\infty,\vol_*)\:|\: 
0<j^2< -p'(\vol_\infty),\quad 
b<\vol_\infty<\vol_*<\vol_0\,;\;j^2 \vol_\infty+\press(\vol_\infty)=j^2 \vol_0+\press(\vol_0)\;\right\}\,.
$$
It is a diffeomorphism onto the open subset $\Lambda$ of $\mathbb{R}^3$ made of parameters $(j,\lambda,\mu)$  for which we do have a periodic wave. From now on, we parametrize  periodic waves by $(j,\vol_\infty, \vol_*)\in \Omega$ instead of 
$(j,\lambda,\mu)\in \Lambda$. Given $(j,\vol_\infty, \vol_*)\in \Omega$ and $(j,\lambda,\mu)=\Phi(j,\vol_\infty, \vol_*)$ there is a unique 
$\vol^*
$ (the peak of the periodic wave) greater than $\vol_*$ such that
$$f(\vol^*)\,-\,\tfrac{1}{2}{j^2}(\vol^*)^2-\lambda \vol^*=\mu=f(\vol_*)\,-\,\tfrac{1}{2}{j^2}\vol_*^2-\lambda \vol_*\,.$$
The wave number $k$ of the corresponding periodic wave is given by
$$\frac{1}{2k}\,=\,\int_{\vol_*}^{\vol^*} \sqrt{\frac{\cap(\vol)}{2(f(\vol)\,-\,\tfrac{1}{2}{j^2}\vol^2-\lambda \vol-\mu)}}\,\dif\vol\,,$$
and mean values can be computed by quadrature in a similar way. In particular, we have
$$\langle\Vol\rangle\,=\,2k\,\int_{\vol_*}^{\vol^*} \vol\,\sqrt{\frac{\cap(\vol)}{2(f(\vol)\,-\,\tfrac{1}{2}{j^2}\vol^2-\lambda \vol-\mu)}}\,\dif\vol\,,$$
$$\langle\Vol^2\rangle\,=\,2k\,\int_{\vol_*}^{\vol^*} \vol^2\,\sqrt{\frac{\cap(\vol)}{2(f(\vol)\,-\,\tfrac{1}{2}{j^2}\vol^2-\lambda \vol-\mu)}}\,\dif\vol\,.$$
An elegant way to remove (integrable) singularities at endpoints in the integrals above has been pointed out in \cite{BronskiJohnson}. 
Indeed, factorizing the denominator as 
$$f(\vol)\,-\,\tfrac{1}{2}{j^2}\vol^2-\lambda \vol-\mu\,=\,(\vol^*-\vol)\,(\vol-\vol_*)\,\varphi(\vol;j,\vol_\infty, \vol_*)\,,$$
and introducing the change of variables
$${\mathscr V}: (\theta;j,\vol_\infty, \vol_*)\mapsto v:=\frac{\vol^*+\vol_*}{2}\,+\,\frac{\vol^*-\vol_*}{2}\,\sin \theta\,,$$
we find that
$$\frac{1}{2k}\,=\,\int_{-\pi/2}^{\pi/2} \sqrt{\frac{\cap({\mathscr V}(\theta;j,\vol_\infty, \vol_*))}{2\,\varphi({\mathscr V}(\theta;j,\vol_\infty, \vol_*);j,\vol_\infty, \vol_*)}}\,\dif\theta\,,$$
$$\langle\Vol\rangle\,=\,2k\,\int_{-\pi/2}^{\pi/2} {\mathscr V}(\theta;j,\vol_\infty, \vol_*)\,\sqrt{\frac{\cap({\mathscr V}(\theta;j,\vol_\infty, \vol_*))}{2\,\varphi({\mathscr V}(\theta;j,\vol_\infty, \vol_*);j,\vol_\infty, \vol_*)}}\,\dif\theta\,,$$
$$\langle\Vol^2\rangle\,=\,2k\,\int_{-\pi/2}^{\pi/2} {\mathscr V}(\theta;j,\vol_\infty, \vol_*)^2\,\sqrt{\frac{\cap({\mathscr V}(\theta;j,\vol_\infty, \vol_*))}{2\,\varphi({\mathscr V}(\theta;j,\vol_\infty, \vol_*);j,\vol_\infty, \vol_*)}}\,\dif\theta\,.$$
Note that in the shallow water case $f(\vol)=1/\vol$,
$$\vol\,(-f(\vol)\,+\,\tfrac{1}{2}{j^2}\vol^2+\lambda \vol+\mu)\,=\,(\vol-\vol_*)\,\left(\tfrac{1}{2}{j^2}\vol^2\,+\,(\tfrac{1}{2}{j^2}\vol_*+\lambda)\,\vol\,+\,\frac{1}{\vol_*}\right)\,,$$
hence $\varphi$ is explicitly given by
$$\varphi(\vol;j,\vol_\infty, \vol_*)\,=\,\frac{j^2}{2\vol}\,(\vol+\vol_{m})\,,\quad \vol_{m}:=\frac{1}{j^2}\,\left(\tfrac{1}{2}{j^2}\vol_*+\lambda+
\sqrt{(\tfrac{1}{2}{j^2}\vol_*+\lambda)^2-2\,{j^2}/{\vol_*}}\right)\,.$$

First of all, we want to check whether \eqref{EKsw_WH1}-\eqref{EKsw_WH4} is an evolutionary system, or in other words if 
the Jacobian matrix $\bM_0$ of 
the mapping $$\bP:=(j,\sigma,\vol_\infty, \vol_*)^T\mapsto \bW(\bP):=(k,\langle\Vol\rangle,\sigma+j\langle V\rangle,j\langle V^2\rangle+\sigma\langle\Vol\rangle)^T$$ is nonsingular. This can be done numerically.
Next, we can rewrite \eqref{EKsw_WH1}-\eqref{EKsw_WH4} in the more compact form
$$\partial_S \bW(\bP) - \partial_Y (j \bW(\bP) + \bF(\bP))=0\,,$$
$$\bF(\bP):= (0,\sigma,\lambda-j\sigma,\tfrac{1}{2}\sigma^2-\mu)^T\,,\quad (j,\lambda,\mu)=\Phi(j,\vol_\infty, \vol_*)\,,$$
from which we easily infer the quasilinear form
\begin{equation}\label{EKsw_WHql}\bM_0(\bP) (\partial_S  - j \partial_Y) \bP + \bM_1(\bP) \partial_Y\bP =0\,,
\end{equation}
with $\bM_1(\bP)\,:=-\,\mbox{Jac } \bF(\bP)\,-\,j\,\bW(\bP)\,(1,0,0,0)$, \emph{i.e.}
$$\bM_1(j,\sigma,v_\infty, v_*)=\left(\begin{array}{cccc} \displaystyle -k & 0 & 0 & 0\\
                                                     \displaystyle -\langle V\rangle & -1 & 0 & 0\\
                                                     \displaystyle -j\langle V\rangle-\frac{\partial \lambda}{\partial j} & j & \displaystyle -\frac{\partial \lambda}{\partial \vol_\infty} & 0\\
                                                     \displaystyle -(\sigma\langle V\rangle+j\langle V^2\rangle)+\frac{\partial \mu}{\partial j} & -\sigma & \displaystyle \frac{\partial \mu}{\partial \vol_\infty} &  \displaystyle \frac{\partial \mu}{\partial \vol_*}
                                                     \end{array}\right)=
                                                     $$
                                                     $$\left(\begin{array}{cccc} \displaystyle -k & 0 & 0 & 0\\
                                                     \displaystyle -\langle V\rangle & -1 & 0 & 0\\
                                                     \displaystyle -j\langle V\rangle+2j v_\infty & j &\displaystyle  j^2+\press'(v_\infty) & 0\\
                                                     \displaystyle -(\sigma\langle V\rangle+j\langle V^2\rangle)+2jv_\infty v_*-j v_*^2 & -\sigma &\displaystyle  (j^2+\press'(v_\infty))v_* & \press(\vol_\infty)-\press(\vol_*)+j^2(\vol_\infty-\vol_*)
                                                     \end{array}\right).
                                                     $$

\noindent
By `change of frame' $Y\mapsto Y+\underline{j}S$, the
hyperbolicity of \eqref{EKsw_WHql} at $\bP=\underline{\bP}$ is equivalent to the hyperbolicity of 
$$\bM_0(\bP) \partial_S \bP + \bM_1(\bP) \partial_Y\bP =0\,$$
at $\bP=\underline{\bP}$.
Once $\bM_0(\bP)$ is known to be nonsingular, a sufficient condition for hyperbolicity is that the eigenvalues of 
 $\bM_0(\bP)^{-1}\bM_1(\bP)$ have four distinct real parts, because then they must be real and distinct. We comment below on a series of numerical results obtained with the shallow water pressure law $\press(\vol)=1/\vol^2$, in which we have computed mean values by using
the trapezoidal rule  
with $10\,000$ points of discretization, and the Jacobian matrices $\bM_0(\bP)$ by means of a finite difference method and a discretization step $h=10^{-6}$. In each picture, we have plotted the real part of the eigenvalues of $\bM_0(\bP)^{-1}\bM_1(\bP)$ as a function of the period $\Xi$: 
the modulated equations are hyperbolic for a given $k$ if we find four distinct real parts. When two curves collide, there is a set of complex conjugate eigenvalue and the system is not hyperbolic any more.

\mathversion{bold}
\subsubsection{Numerical results}
\mathversion{normal}

\noindent
We have checked the hyperbolicity of the modulated equations  in the cases $j=1$ and various values of $v_\infty\in(0.55 ,1.2)$ and in the case $j=4$, $v_\infty\in(0.15, 0.49)$.  Since the hyperbolicity does not depend on  $\sigma$, we have set $\sigma=0$ in $\bM_0$ and $\bM_1$.  In both cases we found a similar scenario. Let us describe the case $j=1$. If $v_\infty\geq v_\infty^1\approx 0.86$, the Whitham equations are hyperbolic for all periodic waves (see figure \ref{ekswj1v09}).

 \begin{figure}[h!]
 \begin{minipage}[c]{.48\linewidth}
   \includegraphics[scale=0.52]{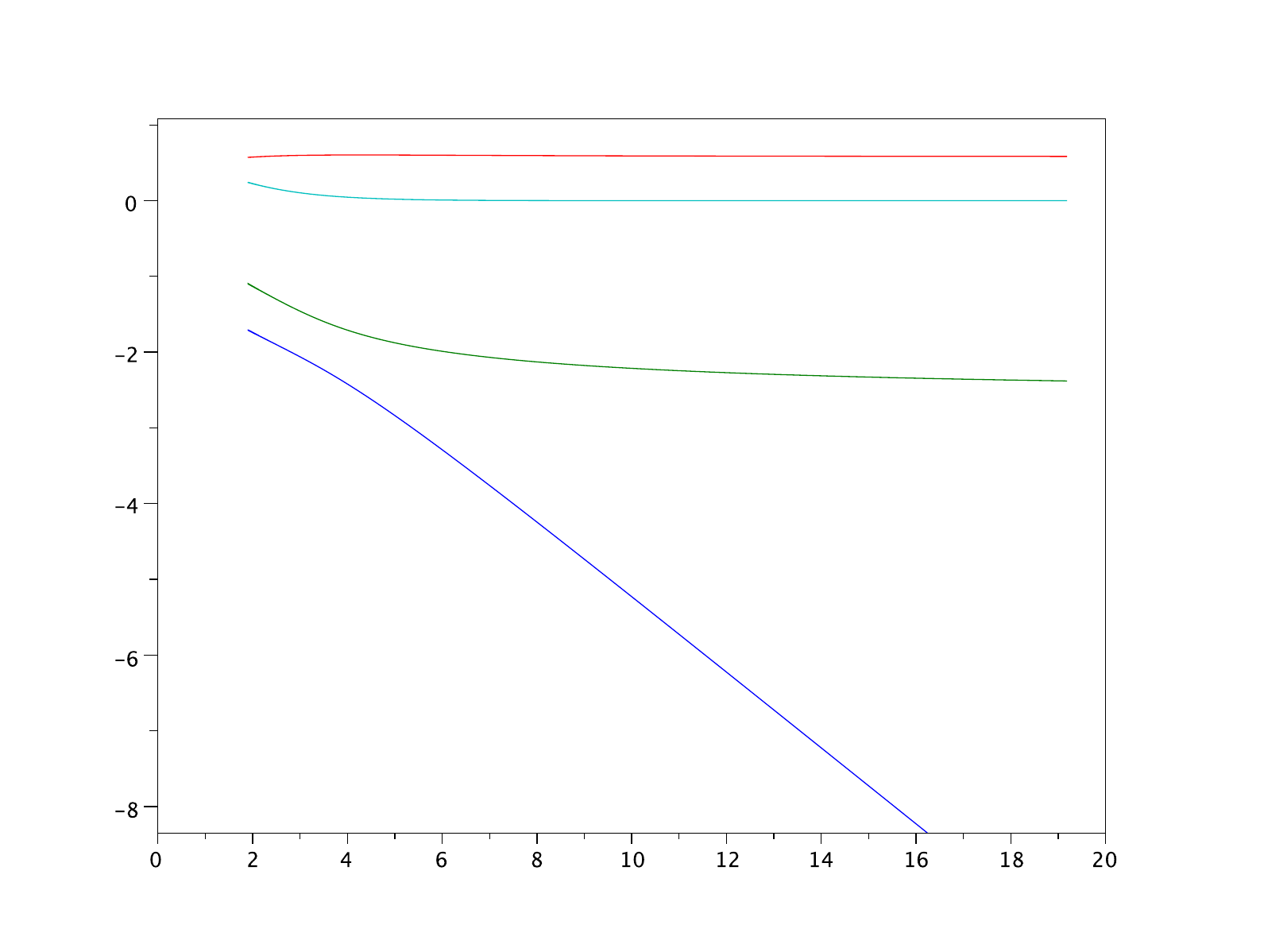}
 \end{minipage} \hfill
   \begin{minipage}[c]{.48\linewidth}
  \includegraphics[scale=0.52]{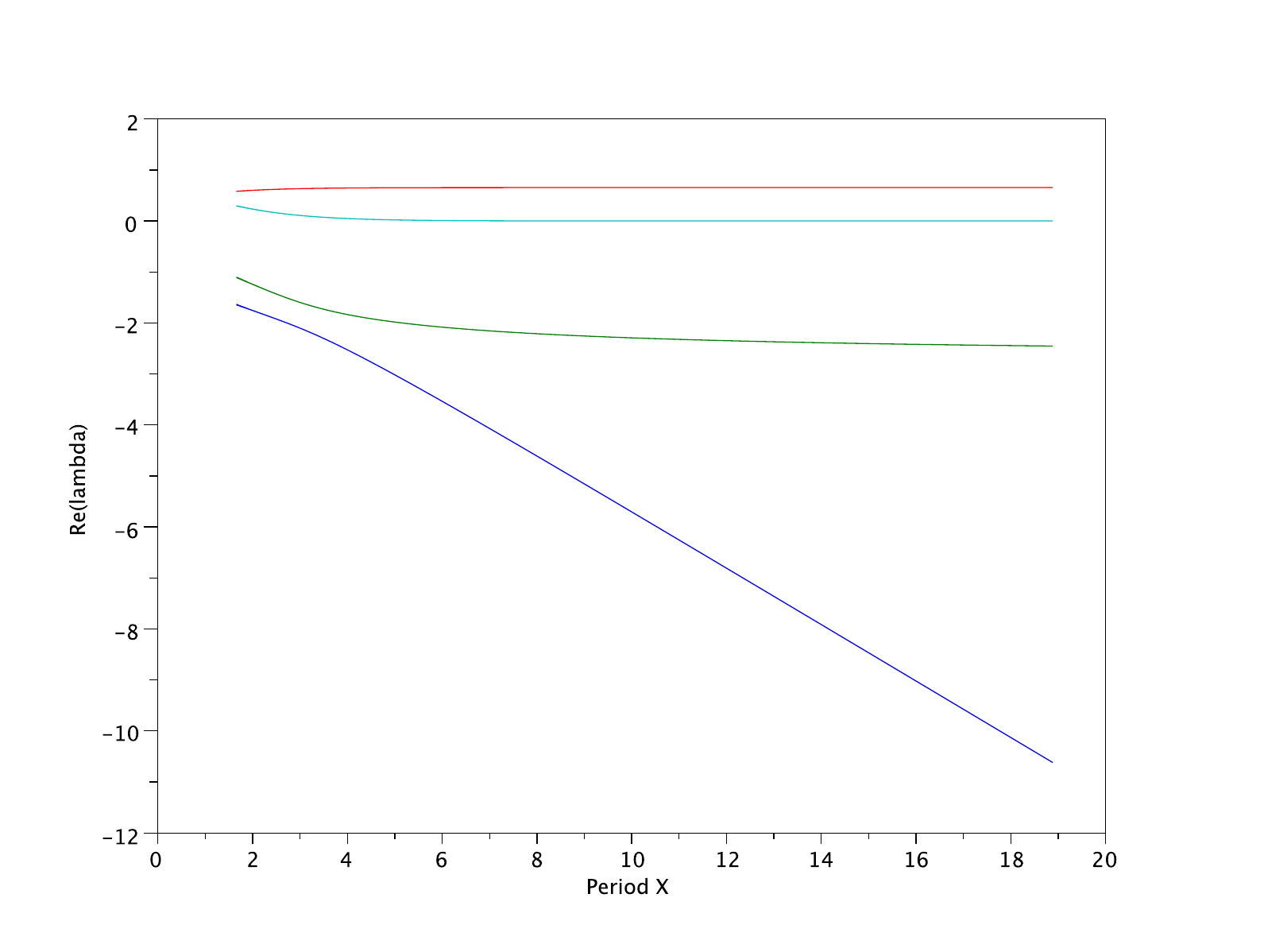}
\end{minipage}
  \caption{\label{ekswj1v09} $\Real(\lambda_i),\:i=1,2,3,4$ as functions of the period $\Xi$. On the left: $j=1$ and $v_\infty=0.93$. On the right: $j=1$ and $v_\infty=0.9$.}
 \end{figure}
 
If $v_\infty\leq v_\infty^1$, one finds a range of periodic  wave periods $\Xi\in (\Xi_m; \Xi_M)$ 
for which the Whitham equations are not hyperbolic, and thus periodic waves are spectrally unstable; see figure \ref{ekswj1v08} (left). Moreover, if $v_\infty\lessapprox 0.58$, there is one eigenvalue that diverges to $\pm\infty$ which means that there  exists $\Xi_c$ such that ${\rm det}(\bM_0)=0$ and the modulation equations are not of evolution type. If $\Xi\neq \Xi_c$, the scenario is similar to the previous case: there is a range of unstable periodic waves; see figure \ref{ekswj1v08} (right).
  \begin{figure}[h!]
 \begin{minipage}[c]{.48\linewidth}
   \includegraphics[scale=0.52]{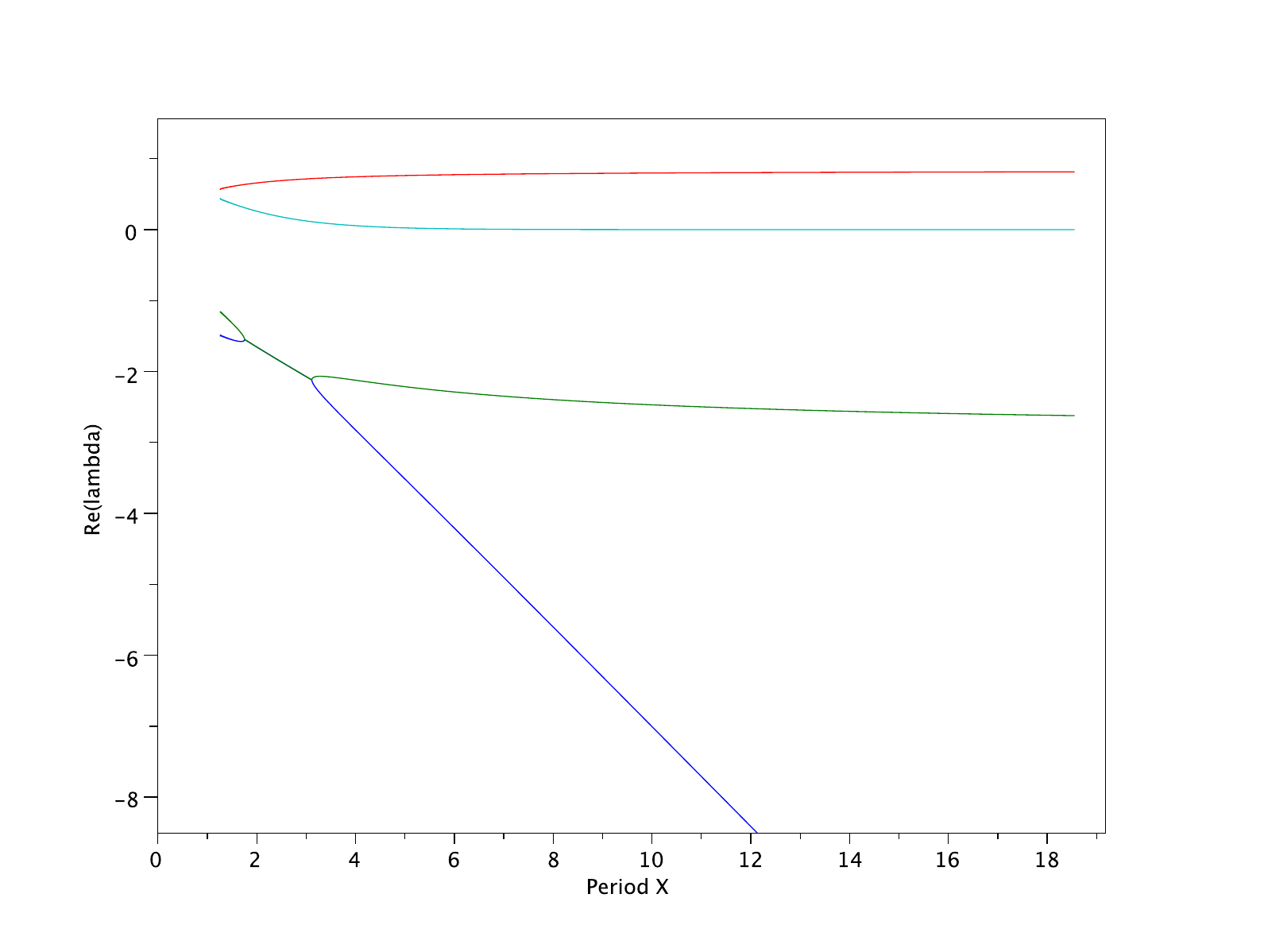}
 \end{minipage} \hfill
  \begin{minipage}[c]{.48\linewidth}
  \includegraphics[scale=0.52]{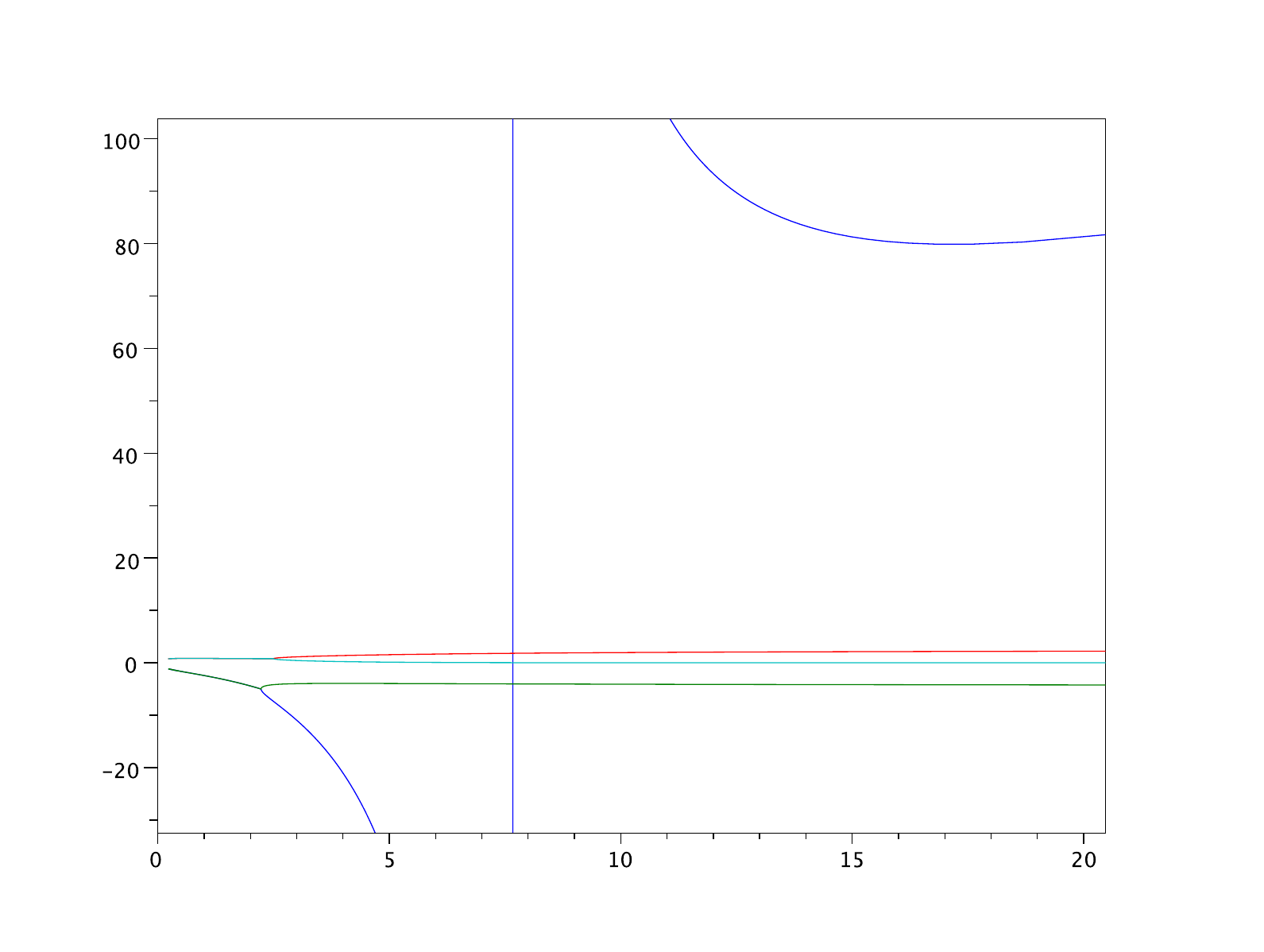}
\end{minipage}
  \caption{\label{ekswj1v08} $\Real(\lambda_i),\:i=1,2,3,4$ as functions of the period $\Xi$. On the left: $j=1$ and $v_\infty=0.84$. On the right: $j=1$ and $v_\infty=0.55$.}
 \end{figure}

 \subsection{Van der Waals type pressure laws}
 
 \subsubsection{Parametrization of periodic waves}
 
 We consider in this section pressure laws that are possibly nonconvex, and even nonmonotone. A typical example is 
the van der Waals pressure law
 $$
 \displaystyle
 \press(\vol;T)=\frac{RT}{\vol-b}-\frac{a}{\vol^2},
 $$
 which exhibits various types of behaviors depending on the temperature $T$ (compared to $a/(bR)$, where 
 $R$ is the perfect gas constant, and $a$, $b$, are parameters of the specific fluid). More precisely, defining
 $$T_0:=\frac{81a}{256bR}\,,\;T_c:=\frac{8a}{27bR}\,,$$
we easily see that 
\begin{enumerate}
\item for $T>T_0$, $\vol\mapsto p(\vol;T)$ is monotonically decaying and convex,
\item for $T_c<T<T_0$, $\vol\mapsto p(\vol;T)$ is monotonically decaying and admits two inflection points,
\item for $T<T_c$, $\vol\mapsto p(\vol;T)$ admits one local minimum, one local maximum, and two inflection points.
\end{enumerate}
After nondimensionalization this pressure law reduces to 
 $$ \press(\vol)=\frac{\gamma}{v-1}-\frac{1}{v^2}$$
 with $\gamma:=\frac{RTb}{a}$. As seen above, transition values of $\gamma$ are
 $\displaystyle\gamma_0=\frac{81}{256}$ and $\displaystyle \gamma_c=\frac{8}{27}$.\\

\noindent
In what follows, we have chosen to deal with the case of a non monotone pressure and $T=600 K$ with $a$ and $b$ roughly corresponding to water,
so that $\gamma\lesssim 0.275<\gamma_c$. \\

If we choose $\vol_\infty$ such that 
$j^2<-p'(\vol_\infty)$, and take as in the previous subsection
$\lambda=-j^2\vol_\infty-\press(\vol_\infty)$,
the phase portrait of the travelling wave ODE
$$
\displaystyle
\cap(\Vol)\Vol_{\zeta\zeta}+\tfrac{1}{2}\cap'(\Vol)\Vol_\zeta^2+\press(\Vol)+j^2\Vol+\lambda=0,
$$
is certainly independent of the function $\cap$, but heavily depends on the values of $j$ and $v_\infty$ when the pressure $\press$ is non monotone. 

\begin{figure}[ht]\label{fig:nonconvexportrait}
\begin{center}
\includegraphics[width=7cm]{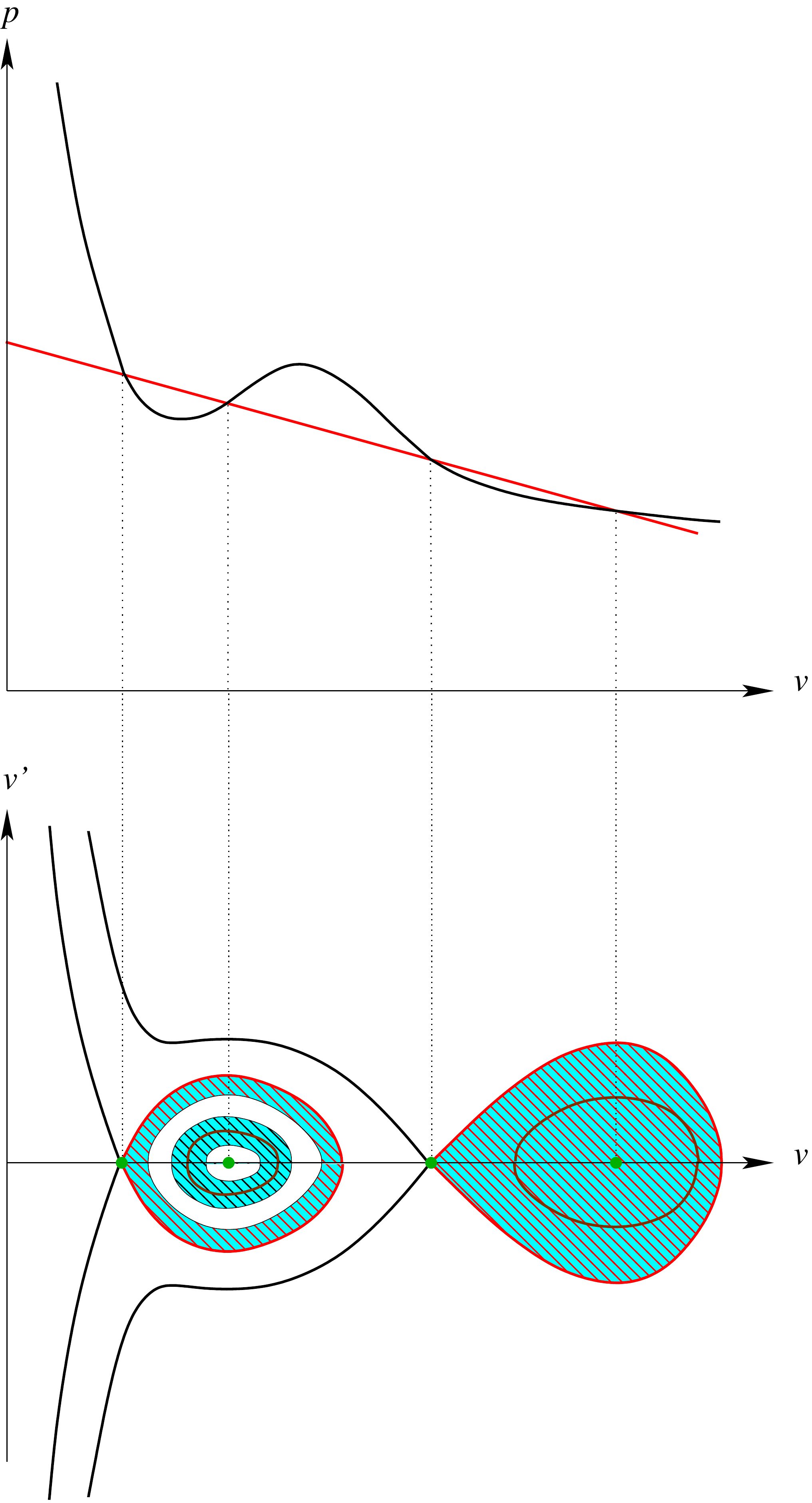} \includegraphics[width=7cm]{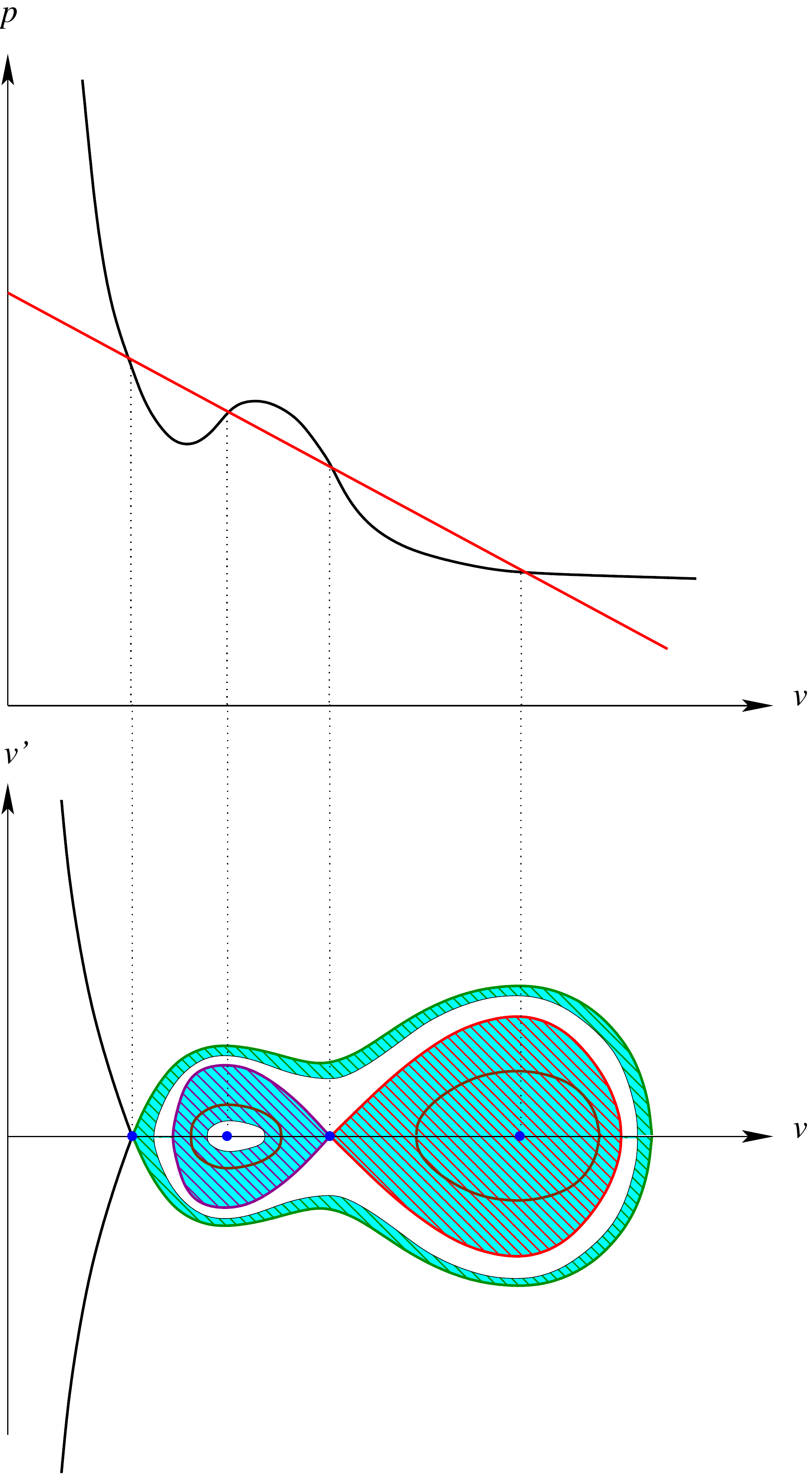}
\end{center}
\caption{\label{portraitVdW} Left: two fish phase portrait. Right: eyes-and-guitar phase portrait. The shaded regions correspond to the domain of hyperbolicity of  Whitham's modulation equations for the cases we have tested}
\end{figure} 

Depending on areas delimited by the pressure curve and by the so-called Rayleigh line, of equation $y=-j^2(\vol-\vol_\infty)$ in the $(\vol,\press)$ plane, two typical phase portraits arise, which may be described as follows (see figure \ref{portraitVdW}).
\begin{description}
\item[Two fish] There are two disconnected homoclinic loops (the `fish'), one ending at $\vol_\infty$ and one with another endpoint, and all the trajectories outside these loops are unbounded; thus there are two types of periodic orbits,  surrounded by either one of the homoclinic ones;
\item[Eyes-and-guitar] There are two homoclinic loops (the `eyes') ending at the same point (maybe $\vol_\infty$), and a third homoclinic loop (the `guitar') surrounds them; there are three types of periodic orbits, those inside the eyes, and those in between the eyes and the guitar.
\end{description}

A series of numerical investigations in various cases is reported below.

\subsubsection{Numerical results in the eyes-and-guitar case
}

\noindent
We have first considered the case $j=
0.023732$, $v_\infty=6.598196$
and checked the hyperbolicity of Whitham's equations with a simple form of the capillary coefficient: $\kappa(v)=1$. Here, we have an ``eyes and guitar'' type phase portrait. There are three families of periodic waves. The first one that is `above' the doubly homoclinic orbit. In this case, there exists $\Xi_c\approx 559$ such that the modulation system is not of evolution type. Furthermore, there exists $\Xi_M\approx 555.4$, such that if $\Xi>\Xi_M$ and $\Xi\neq \Xi_c$, the modulated equations are hyperbolic and thus periodic waves are stable under large wavelength perturbations (see figure \ref{ekvdwj011v201DB}). On the other hand, we found that the periodic waves we could compute in the loops of the doubly homoclinic orbit are stable under large scale perturbations (here Whitham's equations are hyperbolic); see figure  \ref{ekvdwj011v201SB}. 
In the smallest loop, we stopped the computations at $\Xi\approx105$: below this value, the amplitude of periodic waves is too small in comparison to the precision we fixed. Anyway, as explained at the beginning, 
we expect that these periodic waves are unstable for sufficiently small amplitudes, since the Euler equations are not hyperbolic at the center point (pressure is nondecreasing at that point).  We have tested $j=0.032$ and the situation is the same.

\begin{figure}[h!]
\begin{minipage}[c]{.48\linewidth}
\includegraphics[scale=0.52]{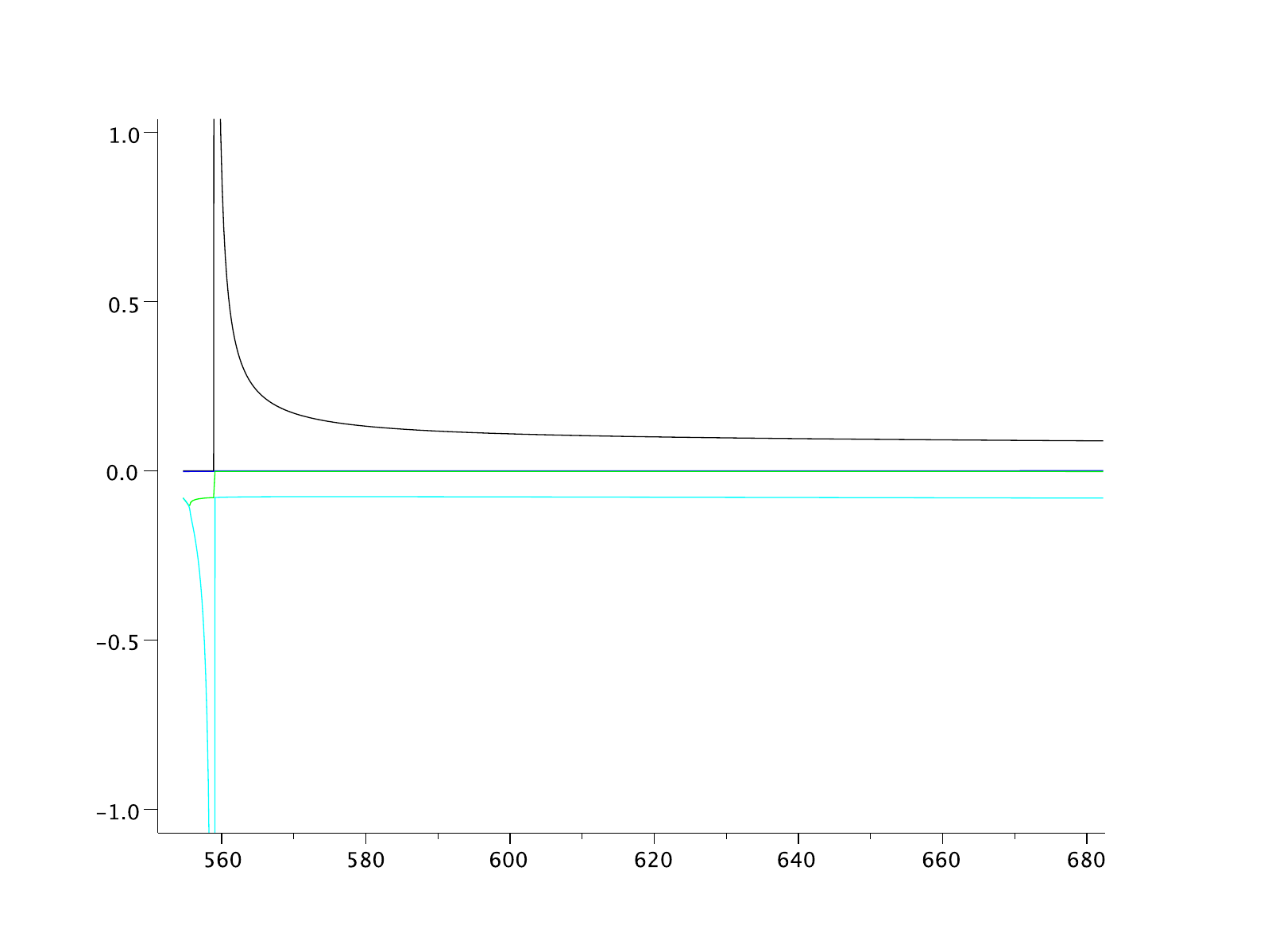}
\end{minipage} \hfill
\begin{minipage}[c]{.48\linewidth}
\includegraphics[scale=0.52]{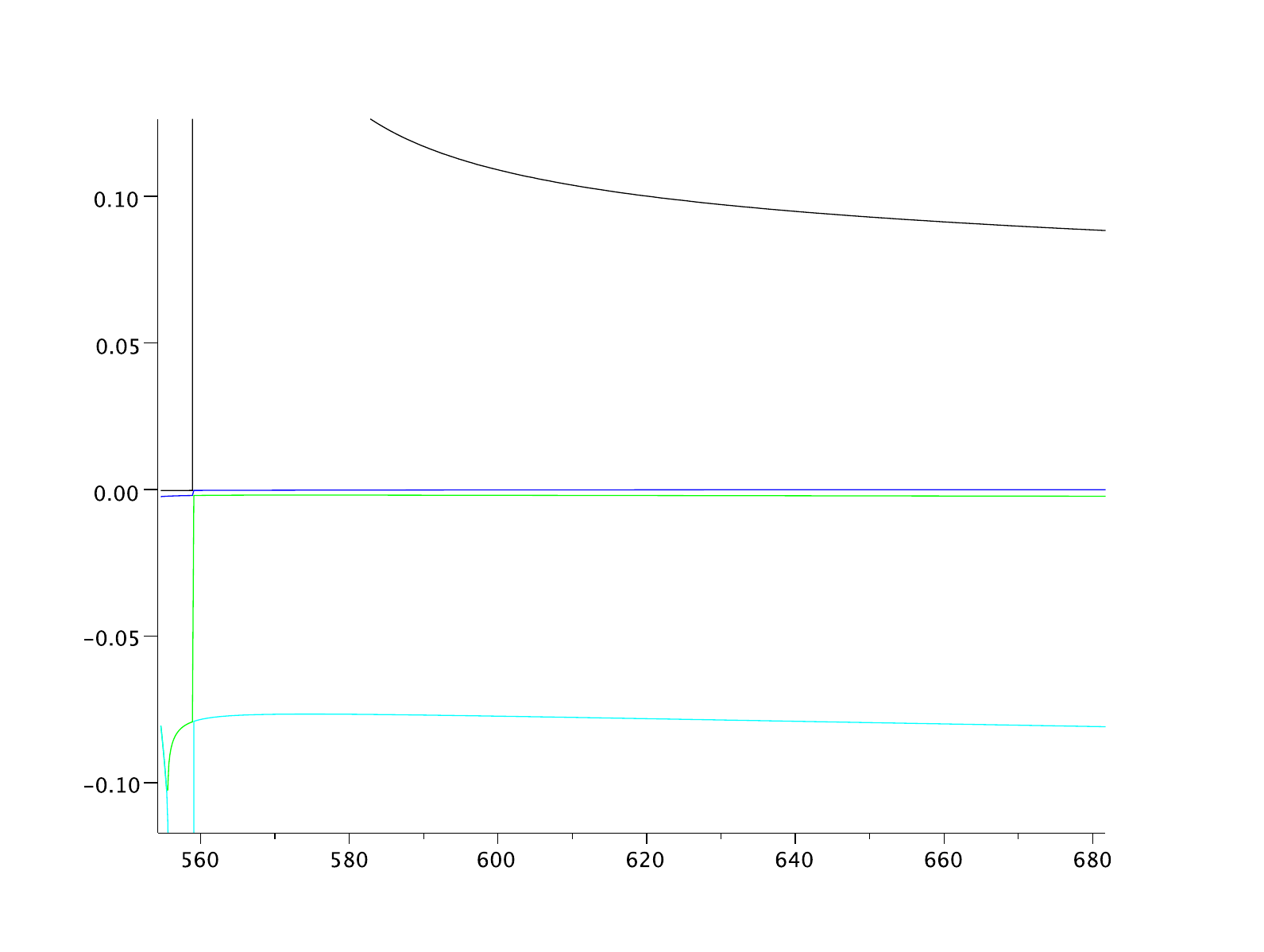}
\end{minipage}
 \caption{\label{ekvdwj011v201DB} $\Real(\lambda_i),\:i=1,2,3,4$ as functions of the period $\Xi$. On the left: $j=0.023732$ and $v_\infty=6.598196$. On the right: zoom of the previous picture}
 \end{figure}

\begin{figure}[h!]
\begin{minipage}[c]{.48\linewidth}
\includegraphics[scale=0.52]{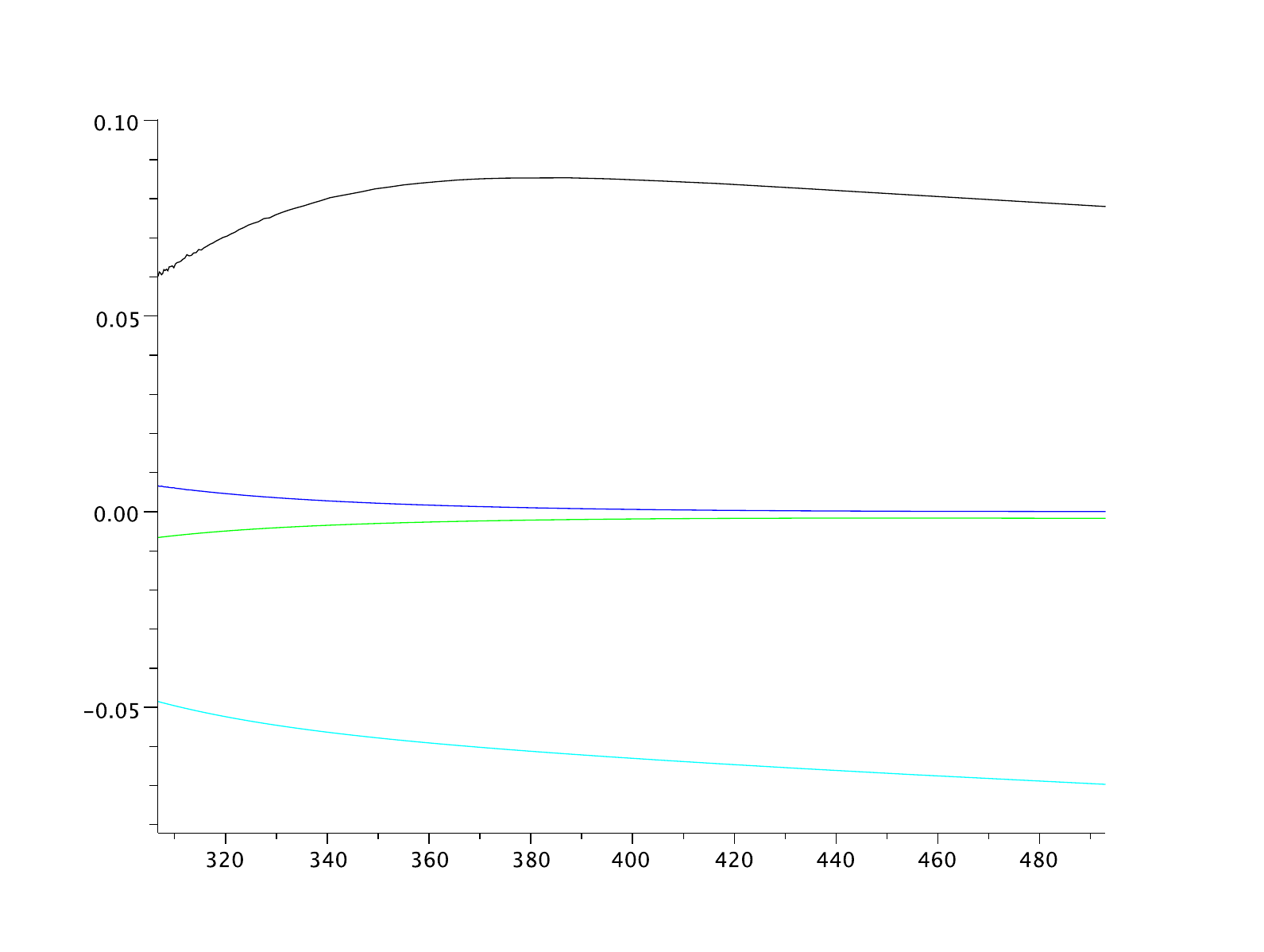}
\end{minipage} \hfill
\begin{minipage}[c]{.48\linewidth}
\includegraphics[scale=0.52]{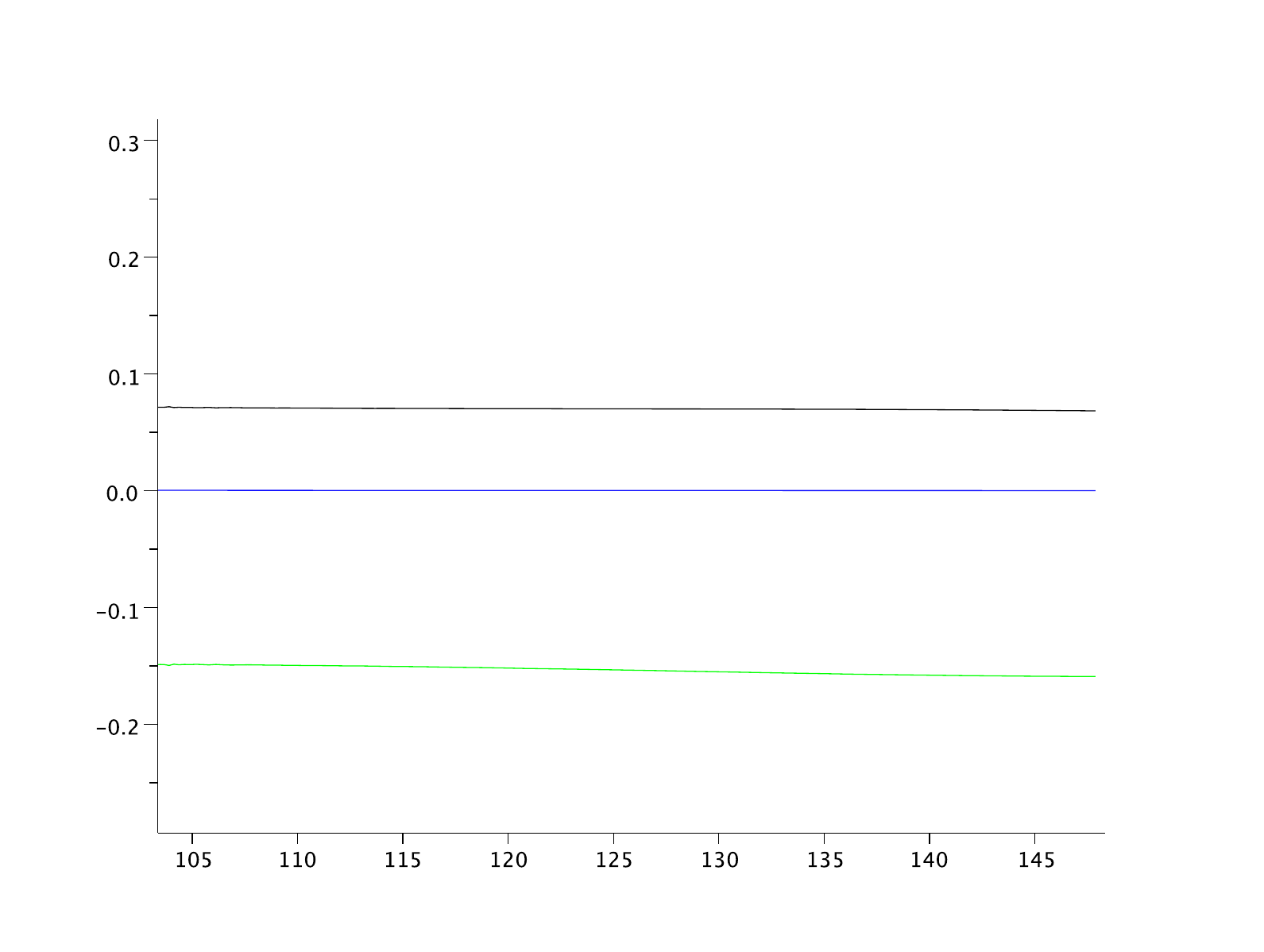}
\end{minipage}
 \caption{\label{ekvdwj011v201SB} $\Real(\lambda_i),\:i=1,2,3,4$ as functions of the period $\Xi$ for periodic waves for $j=0.023732$ and $v_\infty=6.598196$. On the left: periodic waves in the largest loop     On the right: periodic waves in the smallest loop up to $\Xi\approx 105$ (the fourth eigenvalue is real and smaller than $-10$).}
 \end{figure}

\subsubsection{Numerical results in the two-fish case}

Here, we have chosen $j=0.0258$ and $v_\infty=1.90285$: in this case there are two separated homoclinic orbits. The first one is homoclinic to $v_\infty$ and the second one is homoclinic to $w_\infty=7.57197$. In this latter homoclinic, there is a center at $w_0=32.49447$ from which bifurcates a family of periodic orbits. We have checked numerically the hyperbolicity of Whitham's equations associated with this family: the modulation system is always hyperbolic (see figure \ref{ekvdwj013v58LB}).\\

 \begin{figure}[h!]
\begin{center}
\includegraphics[scale=0.6]{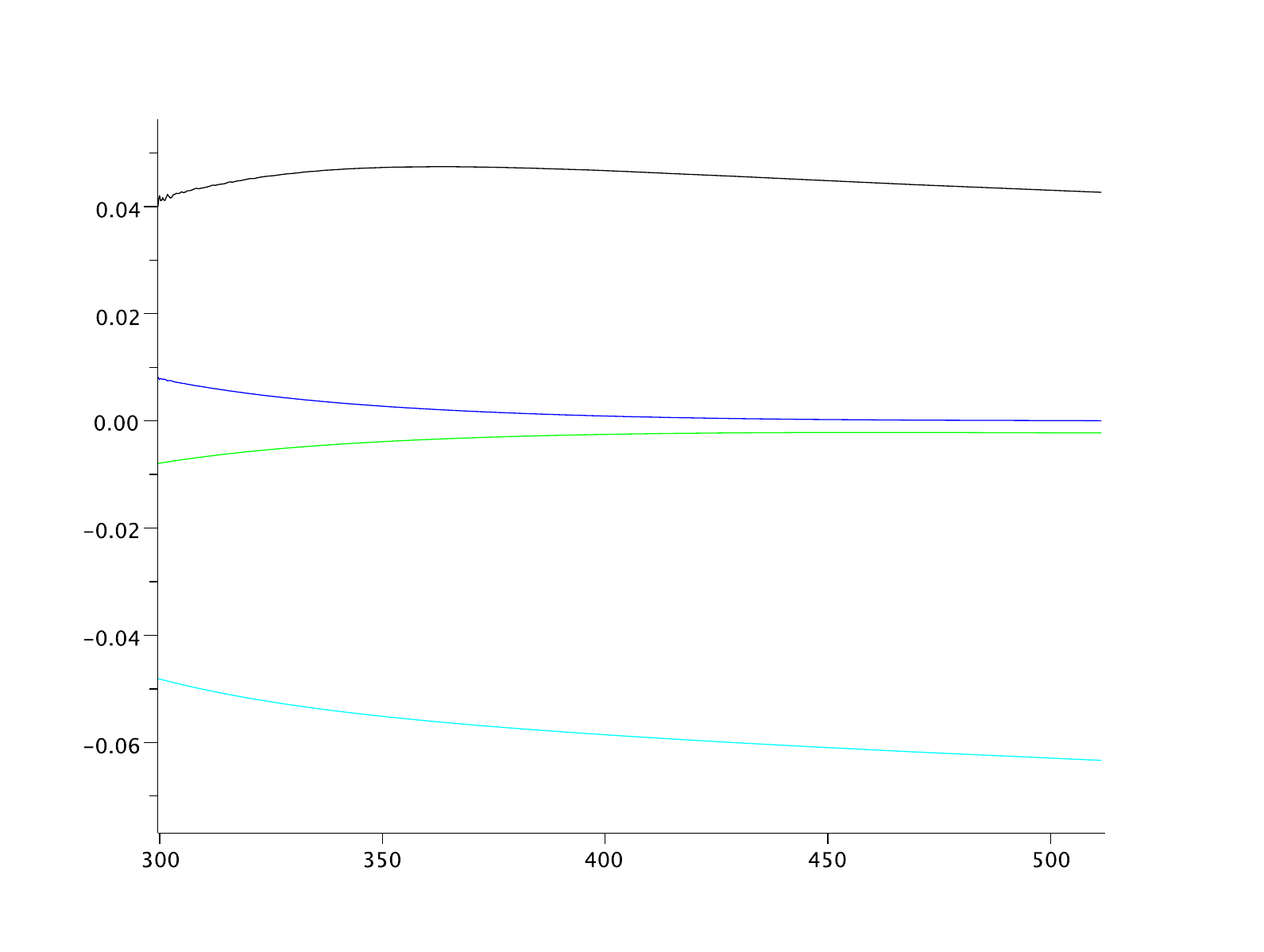}
\end{center} 
\caption{\label{ekvdwj013v58LB} $\Real(\lambda_i),\:i=1,2,3,4$ as functions of the period $\Xi$ for periodic waves in the loop of the homoclinic orbits to $w_\infty=7.57197$. Here $j=0.0258$ and $v_\infty=1.90285$}
\end{figure}
 
 \noindent
 In the homoclinic orbit to $v_\infty=1.90285$, there is a center at $v_0=3.2089$ from which bifurcates the second family of periodic orbits. The picture is a bit more complicated: if $90\leq \Xi\leq \Xi_m\approx 95.15$, the modulation system is hyperbolic. 
 Again, we stopped computations at $\Xi=90$ in the smallest loop since amplitude of the waves are too small (for the precision of computations fixed here): it is expected 
 that for sufficiently small amplitudes, periodic waves are unstable since the Euler equations are not hyperbolic there.
 Then if $\Xi\in(\Xi_m;\; \Xi_M\approx125.5)$, the system is not hyperbolic and the associated periodic waves are unstable. Finally, if $\Xi>\Xi_M$, the Whitham equations are hyperbolic again (see figure \ref{ekvdwj013v58SB}).
 
 \begin{figure}[h!]
\begin{minipage}[c]{.48\linewidth}
\includegraphics[scale=0.52]{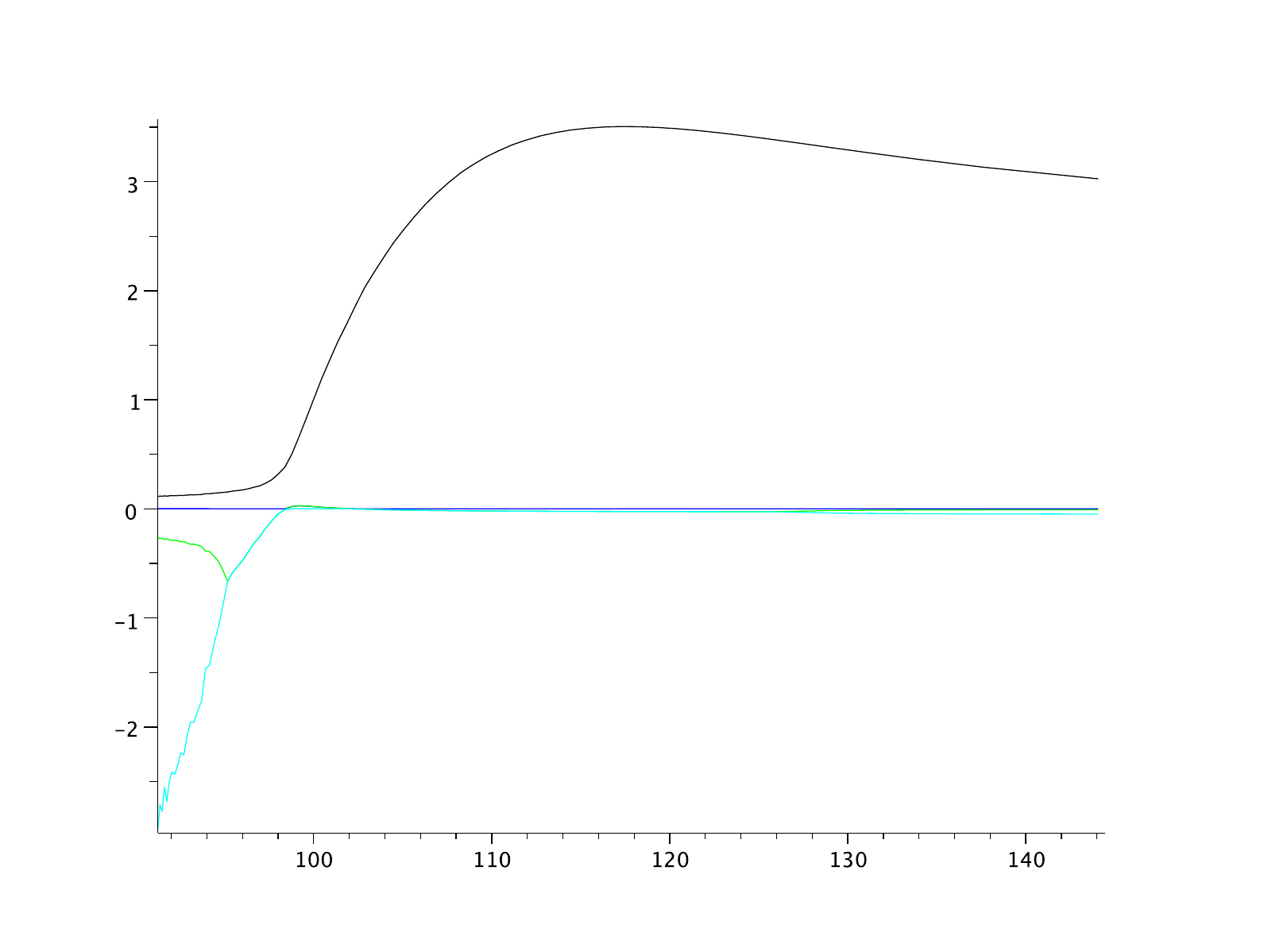}
\end{minipage} \hfill
\begin{minipage}[c]{.48\linewidth}
\includegraphics[scale=0.52]{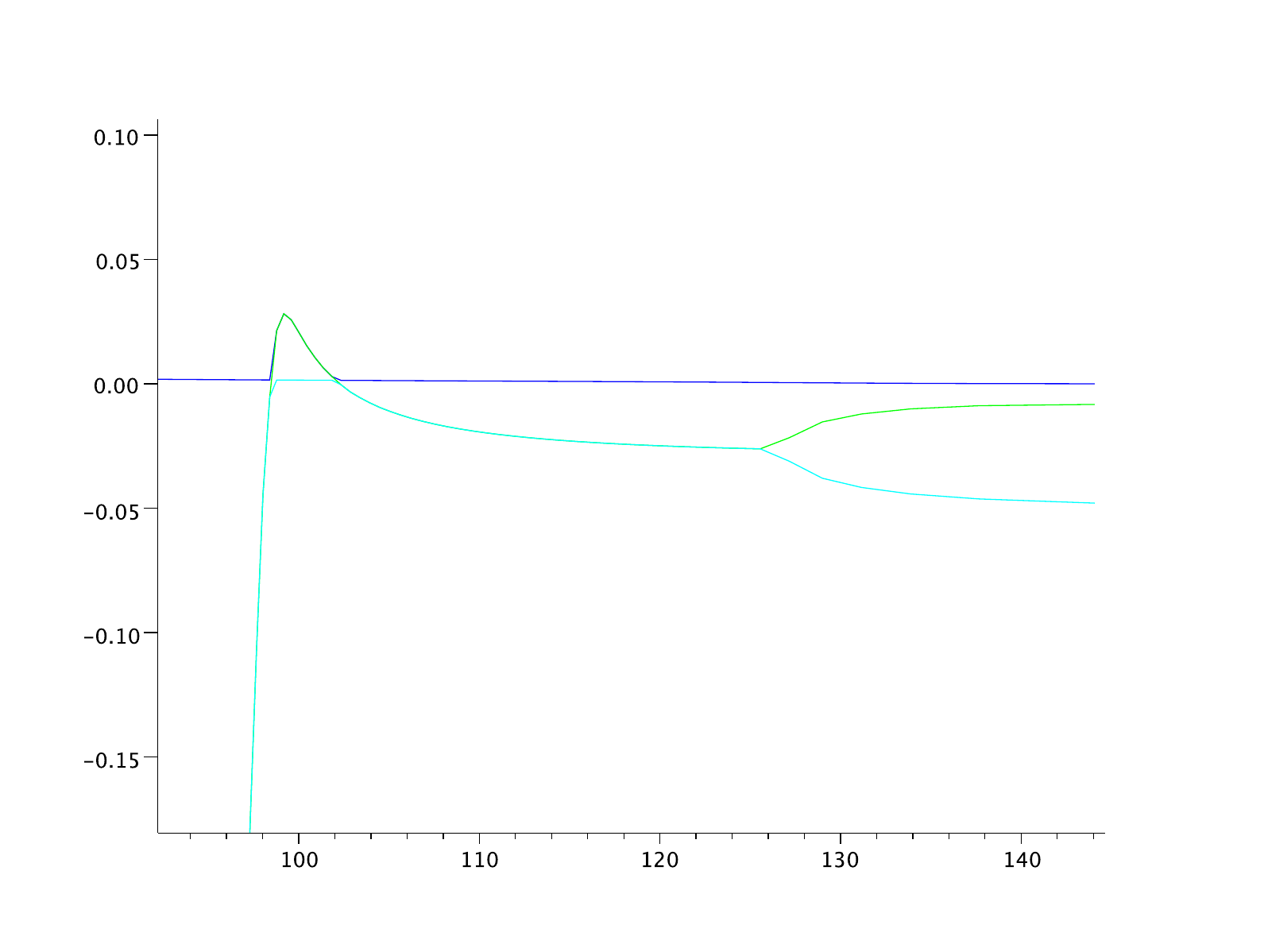}
\end{minipage}
 \caption{\label{ekvdwj013v58SB} On the left: $\Real(\lambda_i),\:i=1,2,3,4$ as functions of the period $\Xi$ for periodic waves in the loop of the homoclinic to $v_\infty$. Here $j=0.0258$ and $v_\infty=1.90285$. On the right: zoom of the previous picture on the three eigenvalues with the smallest real part. The fourth one is always real}
 \end{figure}

\subsubsection{Conclusion from numerical investigations}
 
We have found, in the cases where a doubly homoclinic orbit occurs, that periodic waves that are inside a homoclinic loop are always spectrally stable under long wave length perturbations or, more precisely, that the
 modulation equations are hyperbolic. There is an additional family of periodic waves for which we observed that there is a critical period $\Xi_c$ where the modulated equations are not of evolution type. Moreover, we found that there is $\Xi_M$ such that Whitham's equations are hyperbolic only if $\Xi>\Xi_M$.\\
 
 If there are two separated homoclinic orbits, we found that periodic waves which correspond to the larger solutions pass our stability test 
whereas there are range of periods where Whitham's equations are not hyperbolic for the smaller solutions.

\appendix

\section{A concrete computation}\label{concrete}
We derive here the averaged equations associated with Benjamin's impulses $\rho\vits$ and $\vol\vits$ for \eqref{eq:EKabs1d} and \eqref{eq:EKabsLagb} respectively. 
This computation is given for concreteness, even though it is contained in the abstract computation made in Section \ref{s:permod}.
Taking the inner product of \eqref{eq:EKabs1d1} and \eqref{eq:EKabsLagb1} with $(\vits_0,\rho_0)$ and $(\vitsL_0,\vol_0)$ respectively, then averaging and integrating by parts in $\theta$, we receive
$$\begin{array}{l}-\,K\,\langle \partial_\theta(\rho_0 (\vits_0-\sigma)) \,\vits_1\,+\,
\partial_\theta(\tfrac{1}{2}(\vits_0-\sigma)^2)\,\rho_1\,+\,A_0(\partial_\theta\rho_0)\,\rho_1\rangle\\ [5pt]
\,+\, \partial_T\langle\rho_0\,\vits_0\rangle +\partial_X \langle \rho_0 \vits_0^2\rangle\,+\,\langle \rho_0\,(\partial_X\chem_0\,+\,K\,\partial_\theta B_0)\rangle\,=\,0\,,\end{array}
$$
$$\begin{array}{l}
k\,\langle \partial_\theta(\vitsL_0-j \vol_0)\,\vitsL_1\,-\,(j\,\partial_\theta\vitsL_0)\,\vol_1\,-\,a_0(\partial_\theta\vol_0)\,\vol_1\rangle\, \\ [5pt]
+\,\partial_S\langle{\vol}_0\,\vitsL_0\rangle \,-\,\partial_Y\langle\tfrac{1}{2} {\vitsL}_0^2\rangle\,+\,\langle \vol_0 (\partial_Yp_0\,+\,k\partial_\theta\,b_0)\rangle\,=\,0\,,
\end{array}
$$
All terms with index one here above cancel out because of the profile equations \eqref{eq:EKW0} and \eqref{eq:EKLagW0}, which imply indeed that
$$\partial_\theta(\rho_0 (\vits_0-\sigma))\,=\,0\,,\quad  \partial_\theta(\tfrac{1}{2}(\vits_0-\sigma)^2)\,+\,A_0(\partial_\theta\rho_0)\,=\,0 \,,$$
$$\partial_\theta(\vitsL_0-j \vol_0)\,=\,0\,,\quad 
j\,\partial_\theta\vitsL_0\,+\,a_0(\partial_\theta\vol_0)\,=\,0\,.$$
This is straightforward for the equations on the left, and for the  ones on the right we observe that the second order differential operators $A_0$ and $a_0$ have been defined in such a way that
$$A_0(\partial_\theta\rho_0)\,=\,\partial_\theta \chem_0\,,\quad a_0(\partial_\theta\vol_0)\,=\,\partial_\theta p_0\,.$$
Finally, we recover the equations in \eqref{eq:consImpulseav} and \eqref{eq:consimpulseav} obtained by averaging the impulses' conservation laws by checking that 
$$\langle \rho_0\,(\partial_X\chem_0\,+\,K\,\partial_\theta B_0)\rangle\,=\,\partial_X\left\langle \rho_0\,\chem_0 \,+\,K\,(\partial_\theta \rho_0)\,\frac{\partial \En}{\partial \rho_x}(\rho_0,K\partial_\theta \rho_0)\,-\,\En(\rho_0,K\partial_\theta \rho_0)
\right\rangle$$
$$\langle \vol_0 (\partial_Yp_0\,+\,k\,\partial_\theta b_0)\rangle\,=\,\partial_Y\left\langle \vol_0 \,p_0\,+\,\en(\vol_0,k\partial_\theta\vol_0) \,-\,k\,(\partial_\theta \vol_0)\,\frac{\partial \en}{\partial \vol_y}(\vol_0,k\partial_\theta\vol_0)\right\rangle\,.$$
Let us check the first equality, the second one being identical through the symmetry
$(X,K,\rho_0,\En, \chem_0,B_0) \leftrightarrow (Y,k,\vol_0,\en,-\press_0,-b_0)$. 
We have
$$\partial_X\left\langle \rho_0\,\chem_0 \,+\,K\,(\partial_\theta \rho_0)\,\frac{\partial \En}{\partial \rho_x}(\rho_0,K\partial_\theta \rho_0)\,-\,\En(\rho_0,K\partial_\theta \rho_0)
\right\rangle\,-\,\langle \rho_0\,(\partial_X\chem_0\,+\,K\,\partial_\theta B_0)\rangle\,=\,$$
$$\left\langle (\partial_X\rho_0)\,\chem_0 \,+\,\partial_X\left(K\,(\partial_\theta \rho_0)\,\frac{\partial \En}{\partial \rho_x}(\rho_0,K\partial_\theta \rho_0)\,-\,\En(\rho_0,K\partial_\theta \rho_0)\right)
\right\rangle\,-\,K\,\langle \rho_0\,\partial_\theta B_0\rangle\,.$$
Recalling the definition of $\chem_0$ and integrating once by parts (in $\theta$ of course), we get
$$\langle(\partial_X\rho_0) \chem_0\rangle\,=\,\left\langle (\partial_X\rho_0)\,\frac{\partial \En}{\partial \rho}(\rho_0,K\partial_\theta \rho_0)\,+\,
(\partial_\theta\partial_X\rho_0)\,\frac{\partial \En}{\partial \rho_x}(\rho_0,K\partial_\theta \rho_0)\right\rangle\,=\,
$$
$$\partial_X\langle \En(\rho_0,K\partial_\theta \rho_0)\rangle\,-\,
\left\langle (\partial_XK)(\partial_\theta\rho_0)\,\frac{\partial \En}{\partial \rho_x}(\rho_0,K\partial_\theta \rho_0)\right\rangle\,=\,$$
$$\partial_X\left\langle\En(\rho_0,K\partial_\theta \rho_0) \,-\,K\,(\partial_\theta \rho_0)\,\frac{\partial \En}{\partial \rho_x}(\rho_0,K\partial_\theta \rho_0)\right\rangle\,+\,K\,\partial_X\left\langle (\partial_\theta \rho_0)\,\frac{\partial \En}{\partial \rho_x}(\rho_0,K\partial_\theta \rho_0) \right\rangle\,.$$
So it just remains to show that 
$$\langle \rho_0\,\partial_\theta B_0\rangle\,=\,\partial_X\left\langle (\partial_\theta \rho_0)\,\frac{\partial \En}{\partial \rho_x}(\rho_0,K\partial_\theta \rho_0) \right\rangle\,.$$
Once again, this follows from integrations by parts. Indeed, 
$$-\,\langle (\partial_\theta\rho_0)\, B_0\rangle\begin{array}[t]{l}\,=\,\left\langle - 
(\partial_\theta\rho_0)\,
\dfrac{\partial^2\En}{\partial\rho\partial\rho_x}(\rho_0,K \partial_\theta\rho_0)\,(\partial_X\rho_0)\,+\,(\partial_\theta\rho_0)\,\partial_X\left(\dfrac{\partial\En}{\partial\rho_x}(\rho_0,K \partial_\theta\rho_0)\right)\right.\\ [10pt]
\left.\;\;\;-\,K\,(\partial^2_\theta\rho_0)\,\left(\dfrac{\partial^2\En}{\partial\rho_x^2}(\rho_0,K \partial_\theta\rho_0)\right)\,(\partial_X\rho_0)\right\rangle\\ [15pt]
\,=\, \left\langle -\,\partial_\theta\left(\dfrac{\partial\En}{\partial\rho_x}(\rho_0,K \partial_\theta\rho_0)\right)\,(\partial_X\rho_0)\,
\,+\,(\partial_\theta\rho_0)\,\partial_X\left(\dfrac{\partial\En}{\partial\rho_x}(\rho_0,K \partial_\theta\rho_0)\right)\right\rangle
\\ [15pt]
\,=\, \displaystyle\partial_X\left\langle (\partial_\theta \rho_0)\,\frac{\partial \En}{\partial \rho_x}(\rho_0,K\partial_\theta \rho_0) \right\rangle\,.
\end{array}$$

\section{A convenient structural assumption}
Our purpose here is to check that, under a reasonable structure assumption on the Hamiltonian $\Ham$,
the operator $\Lin^\nu$ behaves properly, and if the kernel of $\Linz$ has the expected size, 
our parametrization hypothesis for periodic profiles is met.

\paragraph{Structure of the Hamiltonian}
To go further into the analysis of the abstract equation \eqref{eq:absHam}, we need to be more specific about the form of the Hamiltonian $\Ham$. Inspired by our examples \eqref{eq:EKabs1d}\eqref{eq:EnK}, \eqref{eq:EKabsLagb}\eqref{eq:enK}, and \eqref{eq:KdVgen}, we write the $\bU$-space as $\R^N=\R^n\times\R^{N-n}$ for some integer $n$, $0\leq n\leq N$,  require that
$$
\bU=\left(\begin{array}{c} \bv \\ \bu\end{array}\right)\,,\quad
\Ham(\bU)=\Ec(\bv,\bu)+\En(\bv,\bv_x)\,,
$$
and assume that $\Ham+c\Impulse$ is uniformly strongly convex in both $\bv_x$ and $\bu$  on the range of $(\bU,\bv_x)$-values and speeds $c$ under consideration. Note that a simple way to make this assumption independent of $c$ is to assume that $\bJ^{-1}$ has a block structure of the form
$$
\bJ^{-1}=\left(\begin{array}{c|c}*&*\\ \hline *&0_{(N-n)\times(N-n)}\end{array}\right)\,,
$$
as is the case 
for the Euler-Korteweg system.

\subsection{Compactness of resolvents}\label{compact_resolvent}

We briefly sketch here a proof of the fact that our structural assumption on $\Ham$ ensures that $\Linz$ with domain $H^3(\R/\Z;\R^n)\times H^1(\R/\Z;\R^{(N-n)})$ have a nonempty resolvent set and compact resolvents. In turn, this implies that, for all $\nu$, $\Lin^\nu$ is a relatively compact perturbation of $\Linz$.

Our structural assumptions readily yield
$$
\Linz=\bJ k\d_\theta \Linvarz,\qquad \Linvarz=\St+\Sd k\d_\theta - k\d_\theta\Sd^*-k\d_\theta\Su k\d_\theta
$$
with $\Su=\Su^*$, $\St=\St^*$,
$$
\Su=\left(\begin{array}{c|c}\su&0_{n\times(N-n)}\\\hline 0_{(N-n)\times n}&0_{(N-n)\times(N-n)}\end{array}\right),\ 
\Sd=\left(\begin{array}{c|c}*&0_{n\times(N-n)}\\\hline 0_{(N-n)\times n}&0_{(N-n)\times(N-n)}\end{array}\right),\ 
\St=\left(\begin{array}{c|c}*&*\\\hline *&\st\end{array}\right),
$$
with $\su$ and $\st$ being uniformly positive definite. Then, standard energy estimates enable us to show that
$$
\begin{array}{rcl}
\left|\langle \bU, 
(z-\Linz)\bU
\rangle\right|&\geq&|\Real(z)|\,\|\bU\|^2-C\,\|\bv\|_{H^2}^2-C\,\|\bu\|_{H^1}^2,\\
\left|\left\langle \left(\begin{array}{c}(k\d_\theta)^3\bv\\C_0k\d_\theta\bu\end{array}\right), 
(z-\Linz)\bU
 \right\rangle\right|&\geq&
\frac12 \langle (k\d_\theta)^3\bv,\su(k\d_\theta)^3\bv\rangle\ +\ \frac{C}{2}\langle k\d_\theta\bu,\st k\d_\theta\bu\rangle\\
&&\ -C\,|\Imag(z)|\,[\|\bv\|_{H^{5/2}}^2+\|\bu\|_{H^{1/2}}^2]-C\,\|\bU\|^2
\end{array}
$$
where $C$ is a positive constant independent of $z$. 
From this we obtain that there exist $\eta>0$ and $C'>0$ such that if $|\Real(z)|\geq \eta\,[1+|\Imag(z)|^6]$ then
\begin{equation}\label{eq:estimres}
\|\bU\|_{H^3\times H^1}\ \leq\ C'\|(z-\Linz)\bU\|.
\end{equation}
This already shows that, for such a $z$, $(z-\Linz)$ has a closed range and is one-to-one with a continuous inverse. To check that the previous range is dense, we only need to examine whether the formal adjoint is indeed one-to-one (on smooth functions). This amounts to show that
\begin{equation}\label{eq:ev}
(-\bar z-\Linvarz\bJ k\d_\theta)\bV\ =\ 0 
\end{equation}
has no nontrivial smooth solution $\bV$. 
Applying the operator $\bJ k\d_\theta$ to 
\eqref{eq:ev},
we deduce from 
\eqref{eq:estimres}
applied to $\bar z$ 
and $\bU=\bJ k \partial_\theta \bV$
that $k\d_\theta\bV=0$, 
which in turn implies  $\bV=0$, because of \eqref{eq:ev} and the fact that $z$ is nonzero. 
This proves that for the above $\eta>0$, if $|\Real(z)|\geq \eta\,[1+|\Imag(z)|^6]$ then $z$ lies in the resolvent set of $\Linz$.

Since we have not used any Poincar\'e inequality in our previous arguments, they apply \emph{mutatis mutandis} to 
$\Lin$ with domain $H^3(\R;\R^n)\times H^1(\R;\R^{(N-n)})$,
and show that it has a nonempty resolvent set.

\subsection{Parametrization of periodic orbits}\label{Whitham_param}

Let us prove that, under our structural assumption on $\Ham$, there is no restriction in assuming, as done in Theorem~\ref{thm:absstab}, that Whitham's parametrization by $(k,\bM,P)$ is admissible. More precisely, we are going to show that, if nearby periodic traveling wave profiles form a $N+2$ dimensional manifold, and if the generalized kernel of $\Linz$ is of dimension $N+2$, then those nearby profiles are parametrized by $(k,\bM,P)$. This extends to our Hamiltonian framework a result previously shown for parabolic conservation laws by Serre \cite{Serre}.

To set things on a more formal ground, we define on some open neighborhood $\mathcal{U}$ of wave values $(\uXi,\uc,\ubv(0),\ubv_{x}(0),\ublambda)$ the map
$$
\mathcal{R}:\qquad\begin{array}{ccl}\mathcal{U}&\longrightarrow&\R^{2n},\\(\Xi,c,\bUuz,\bUuzx,\blambda)&\longmapsto&([\bv]_0^{\Xi},[\bv_x]_0^{\Xi})\end{array}
$$
where $[\,\cdot\,]_0^{\Xi}$ denotes the jump $[f]_0^{\Xi}=f(\Xi)-f(0)$, and $\bU$ is the solution of
$$
\Euler ( \Ham +c \Impulse)[\bU] = \blambda,\quad \bv(0)=\bUuz,\quad \bv_{x}(0)=\bUuzx.
$$
We identify in the usual way nearby periodic traveling wave profiles with elements of the zero set of $\mathcal{R}$.

\begin{proposition}
Assume that $\ubU$ is non trivial and that $\mathcal{R}$ has constant rank $2n-1$. Then the generalized kernel of $\Linz$ is of dimension $N+2$ if and only if, up to translation, nearby periodic traveling wave profiles may be regularly parametrized by $(k,\bM,P)$.
\end{proposition}

\begin{proof}
Our proof is based upon the fact that the dimension of the generalized kernel of $\Linz$ is the algebraic multiplicity of zero as a root of some Evans function $\Evans(\,\cdot\,)$ (see \cite{Gardner93}). Indeed, viewing spectral problem
\begin{equation}\label{Evans_eq}
z\,\bV\ =\ \Lin\bV
\end{equation}
for $(z,\bV)=(z,(\bv,\bu)^T)$ as a system of coupled differential equations of third-order in $\bv$ and first-order in $\bu$, we may introduce its fundamental solution $R(z;\cdot)$ normalized by $R(z;0)=\id_{\R^{3n}\times\R^{(N-n)}}$ and define
$$
\Evans(z)\ =\ \det ([R(z;\,\cdot\,)]_0^{\uXi}).
$$
Then the condition on the dimension  of the generalized kernel of $\Linz$ reads
$$
\Evans(z)\ =\ a\,z^{N+2}\ +\ \cO(z^{N+3})
$$
for some nonzero $a$ \cite{Gardner93}.
We want to convert this into some information about profiles parametrization.

Let us denote by $\bV^j(z;\cdot)$ the solution to \eqref{Evans_eq} corresponding to the $j$-th column of the matrix $R(z;\cdot)$, that is $\bV^j(z;\cdot)$ solves \eqref{Evans_eq} and $(\bv^j(z;0),\bv^j_x(z;0),\bv^j_{xx}(z;0),\bu^j(z;0))^T$ is the $j$-th vector of the canonical basis of $\R^{3n}\times\R^{(N-n)}$. The Evans function is then written
$$
\Evans(z)\ =\ \left|\begin{array}{rrlcrl}
&[\!\!\!\!&\bv^1]& \cdots &[\!\!\!\!&\bv^{N+2n}]\\ 
&[\!\!\!\!&\bv^1_x]& \cdots &[\!\!\!\!&\bv^{N+2n}_x]\\ 
&[\!\!\!\!&\bv^1_{xx}]& \cdots &[\!\!\!\!&\bv^{N+2n}_{xx}]\\ 
&[\!\!\!\!&\bu^1]&\cdots &[\!\!\!\!&\bu^{N+2n}]\\ 
\end{array}\right|
$$
where we have dropped the marks $0$ and $\uXi$ on jumps. To go further, using our structural assumptions, we write\footnote{We warn the reader that, 
%
because of $\uk$ factors,
these notations are not compatible with the ones of the previous section of the present Appendix.}
$$
\Lin=\bJ\d_x\Linvar,\qquad \Linvar=\St+\Sd\d_x-\d_x\Sd^*-\d_x\Su\d_x
$$
with $\Su=\Su^*$, $\St=\St^*$,

$$
\Su=\left(\begin{array}{c|c}\su&0_{n\times(N-n)}\\\hline 0_{(N-n)\times n}&0_{(N-n)\times(N-n)}\end{array}\right),\ 
\Sd=\left(\begin{array}{c|c}*&0_{n\times(N-n)}\\\hline 0_{(N-n)\times n}&0_{(N-n)\times(N-n)}\end{array}\right),\ 
\St=\left(\begin{array}{c|c}*&*\\\hline *&\st\end{array}\right),
$$
$\su$ and $\st$ being uniformly positive definite. Integrating \eqref{Evans_eq} from $0$ to $\uXi$ yields
$$
z\,\int_0^{\uXi} \bV^j\ =\ \bJ\left(\begin{array}{c}\su(0)[\bv^j_{xx}]+\ *[\bv^j_x]+*[\bv^j]+*[\bu^j]\\
\st(0)[\bu^j]+\ *[\bv^j]
\end{array}\right).
$$
Therefore, up to a nonzero multiplicative constant, $\Evans(z)$ is also written
$$
\left|\begin{array}{rrlcrl}
&[\!\!\!\!&\bv^1]& \cdots &[\!\!\!\!&\bv^{N+2n}]\\ 
&[\!\!\!\!&\bv^1_x]& \cdots &[\!\!\!\!&\bv^{N+2n}_x]\\ 
&z\!\!\!\!&\int_0^{\uXi} \bV^1 &\cdots&z\!\!\!\!&\int_0^{\uXi} \bV^{N+2n}
\end{array}\right|.
$$
Now, corresponding to the impulse equation, we also have
$$
\begin{array}{rcl}
z\,\dfrac{\d\Impulse}{\d U_{\alpha}}(\bV^j_\alpha)&=&
\d_x\left(\ubU\cdot \Hess(\Ham+\uc\Impulse)(\bV^j)+\bV^j\cdot\Euler(\Ham+\uc\Impulse)(\ubU)
-\dfrac{\d(\Ham+\uc\Impulse)}{\d U_{\alpha}}(\bV^j_\alpha)\right)\\
&&+\d_x\left(
\dfrac{\d^2(\Ham+\uc\Impulse)}{\d U_{\alpha,x}\d U_{\beta}}(\ubU^j_{\alpha,x})(\bV^j_{\beta})
+\dfrac{\d^2(\Ham+\uc\Impulse)}{\d U_{\alpha,x}\d U_{\beta,x}}(\ubU^j_{\alpha,x})(\bV^j_{\beta,x})\right)\\
&=&\d_x\left(\ubU\cdot \Linvar \bV^j+\Su\ubU_x\cdot\bV^j_x+((\Sd-\Sd^*)\ubU_x-\Su\ubU_{xx})\cdot\bV^j\right)
\end{array}
$$
where the convention is as before that linearization and derivatives are taken at $\ubU$. By integrating the relation here above, we obtain
$$
z\,\int_0^{\uXi}\dfrac{\d\Impulse}{\d U_{\alpha}}(\bV^j_\alpha)
\ =\ (\su\ubv_x)(0)\cdot[\bv^j_x]+\ *\ [\bv^j]+\ *\ z\,\int_0^{\uXi} \bV^j.
$$
We still need to check that it is not a trivial relation. But, since $\ubU$ is non trivial, there is a point where $\ubv_x$ is nonzero, otherwise $\ubv$ would be constant thus $\ubu$ and $\ubU$ would also be constant, a contradiction. 
Then, since assumptions of the proposition and terms of the equivalence we are currently proving are invariant by translation, we may assume that $\ubv_x(0)$ is nonzero. Now let us pick 
$\ell$
such that the 
$\ell$-th
 component of $\su(0)\ubv_x(0)$ is nonzero and, for any $\bV=(\bv,\bu)^T\in\R^N=\R^n\times\R^{(N-n)}$, denote by 
$\bv_{*}$
 the vector of $\R^{(n-1)}$ obtained from $\bv$ by deleting the 
 $\ell$-th
 component. Then, up to a nonzero multiplicative constant, $\Evans(z)$ is
$$
z^{N+1}\ \left|\begin{array}{rrlcrl}
&[\!\!\!\!&\bv^1]& \cdots &[\!\!\!\!&\bv^{N+2n}]\\ 
&[\!\!\!\!&(\bv^1_{*})_x]& \cdots &[\!\!\!\!&(\bv^{N+2n}_{*})_x]\\
&&\int_0^{\uXi} \frac{\d\Impulse}{\d U_{\alpha}}(\bV^1_\alpha)&\cdots&&\int_0^{\uXi} \frac{\d\Impulse}{\d U_{\alpha}}(\bV^{N+2n}_\alpha)\\
&&\int_0^{\uXi} \bV^1&\cdots&&\int_0^{\uXi} \bV^{N+2n}
\end{array}\right|.
$$

Up to a change of basis we may assume that $\bV^1(0\,;\cdot)=\ubU_x$, $\bV^j(0\,;\cdot)=\bU_{\lambda_{j-2n}}$ for $2n+1\leq j\leq N+2n$ and $(\bV^1(0\,;\cdot),\cdots,\bV^{2n}(0\,;\cdot))$ is a basis of the linear span of $\{\bU_{(\bUuz)_1},\cdots,\bU_{(\bUuz)_{n}}\bU_{(\bUuzx)_1},\cdots,\bU_{(\bUuzx)_{n}}\}$. With this choice, after setting $\widetilde{\bV}^1=\bV_z(0;\cdot)$,
$\Evans (z)$ is written up to a nonzero multiplicative constant
$$
z^{N+2}\ \left|\begin{array}{rrlrlcrl}
&[\!\!\!\!&\widetilde{\bv}^1]&[\!\!\!\!&\bv^2]& \cdots &[\!\!\!\!&\bv^{N+2n}]\\ [1ex]
&[\!\!\!\!&(\widetilde{\bv}_{*}^1)_x]&[\!\!\!\!&(\bv^2_{*})_x]& \cdots &[\!\!\!\!&(\bv^{N+2n}_{*})_x]\\[1ex]
&&\int_0^{\uXi} \frac{\d\Impulse}{\d U_{\alpha}}(\widetilde{\bV}^1_\alpha)&&\int_0^{\uXi} \frac{\d\Impulse}{\d U_{\alpha}}(\bV^2_\alpha)&\cdots&&\int_0^{\uXi} \frac{\d\Impulse}{\d U_{\alpha}}(\bV^{N+2n}_\alpha)\\[1ex]
&&\int_0^{\uXi} \widetilde{\bV}^1&&\int_0^{\uXi} \bV^2&\cdots&&\int_0^{\uXi} \bV^{N+2n}
\end{array}\right|\ +\ \cO(z^{N+3}).
$$
Since $\widetilde{\bV}^1$ satisfies $\ubU_x\ =\ \Lin\widetilde{\bV}^1$, it differs from $\bU_c$ by an element of the kernel of $\Lin$ which is spanned by $(\bV^1(0\,;\cdot),\cdots,\bV^{N+2n}(0\,;\cdot))$. This yields that, up to a multiplicative nonzero constant and an additive remainder $\cO(z^{N+3})$, $\Evans(z)$ reads
$$
z^{N+2}\ \left|\begin{array}{rrlrlcrl}
&[\!\!\!\!&\bv_c]&[\!\!\!\!&\bv^2]& \cdots &[\!\!\!\!&\bv^{N+2n}]\\ 
&[\!\!\!\!&((\bv_{*})_{c})_x]&[\!\!\!\!&(\bv^2_{*})_x]& \cdots &[\!\!\!\!&(\bv^{N+2n}_{*})_x]\\
&&\int_0^{\uXi} \frac{\d\Impulse}{\d U_{\alpha}}((\bU_c)_\alpha)&&\int_0^{\uXi} \frac{\d\Impulse}{\d U_{\alpha}}(\bV^2_\alpha)&\cdots&&\int_0^{\uXi} \frac{\d\Impulse}{\d U_{\alpha}}(\bV^{N+2n}_\alpha)\\
&&\int_0^{\uXi} \bU_c&&\int_0^{\uXi} \bV^2&\cdots&&\int_0^{\uXi} \bV^{N+2n}
\end{array}\right|
$$
or
$$
z^{N+2}\ \left|\begin{array}{crlrlcrl}
\ubv_x(0)&[\!\!\!\!&\bv_c]&[\!\!\!\!&\bV^2_1(0\,;\cdot)]& \cdots &[\!\!\!\!&\bv^{N+2n}(0\,;\cdot)]\\ 
(\ubv_{*})_{xx}(0)&[\!\!\!\!&((\bv_{*})_{c})_x]&[\!\!\!\!&(\bv^2_{*})_x(0\,\cdot)]& \cdots &[\!\!\!\!&(\bv^{N+2n}_{*})_x(0\,;\cdot)]\\
1&&\quad0&&\quad0&\cdots&&\quad0\\
\Impulse(\ubU)(0)&&\int_0^{\uXi} \frac{\d\Impulse}{\d U_{\alpha}}((\bU_c)_\alpha)&&\int_0^{\uXi} \frac{\d\Impulse}{\d U_{\alpha}}(\bV^2_\alpha(0\,;\cdot))&\cdots&&\int_0^{\uXi} \frac{\d\Impulse}{\d U_{\alpha}}(\bV^{N+2n}_\alpha(0\,;\cdot))\\
\ubU(0)&&\int_0^{\uXi} \bU_c&&\int_0^{\uXi} \bV^2(0\,;\cdot)&\cdots&&\int_0^{\uXi} \bV^{N+2n}(0\,;\cdot)
\end{array}\right|.
$$

%
We are ready to complete the proof by observing the latter determinant.
Indeed our assumption on $\mathcal{R}$ implies that the $(2n-1)$-st rows of the above matrix are linearly independent. 
Furthermore,
 the kernel of the corresponding linear map is the tangent space at $\ubU$ of the profiles manifold (profiles being identified when equal up to translation).
Thus the differential map of $\bU\mapsto (\Xi,\int_0^{\Xi}\Impulse(\bU),\int_0^{\Xi}\bU)$ is invertible on this tangent space if and only if the above determinant is non zero. 
Consequently, 
this map is full-rank if and only if the generalized kernel of $\Linz$ is of dimension $N+2$.
\end{proof}

Note that, by introducing Floquet exponents in previous arguments as in analogous computations in \cite{Serre}, one may obtain an alternative proof of Theorem~\ref{thm:absstab}.

\section{Galilean invariance}\label{s:gal}
In order to check that the hyperbolicity of the  Euler--Korteweg modulation equations \eqref{eq:EKW14} does not depend on $\sigma$, 
it is convenient to rewrite those equations in a form that is similar to \eqref{eq:EKLagW14th} for \eqref{eq:EKLagW14}. Substituting $\meanrho \sigma +j$ for  the mean momentum $\langle\rho_0\vits_0\rangle$ in  \eqref{eq:EKW14}, and manipulating the remaining mean values as in the proof of Theorem \ref{thm:EL},   we receive the system
\begin{equation}\label{eq:EKW14th}
\left\{\begin{array}{l}
\partial_T K \,+\,\partial_X (\sigma K)\,=\,0\,,\\ [5pt]
\partial_T\meanrho +\partial_X (\meanrho \sigma +j)\,=\,0\,,\\ [5pt]
\partial_T \left(\dfrac{\meanrho \sigma +j+D}{\meanrho}\right) + \partial_X\left(\dfrac{1}{2}\dfrac{(\meanrho \sigma +j)(\meanrho \sigma +j+2D)}{\meanrho^2} +\,\meanchem\right)\,=\,0\,, \\ [10pt]
\partial_T(\meanrho \sigma +j) \,+\,\partial_X\left(\dfrac{(\meanrho \sigma +j)^2+2jD}{\meanrho}\,+\,\meanrho\,\meanchem \,+\,K\,\Theta\,-\,\meanE
\right) \,=\,0\,,
\end{array}\right.
\end{equation}
together with the generalized Gibbs relation  \eqref{eq:Gibbs} $\dif \meanE \,=\,\meanchem \dif \meanrho\,+\,\Theta\,\dif K\,+\,\dfrac{j}{\meanrho}\,\dif D$.
Then, we easily check that \eqref{eq:EKW14th} is invariant by  the Galilean transformation 
$$(T,X,K,\meanrho,\sigma,D)\mapsto
(T,X-\underline{\sigma} T,K,\meanrho,\sigma-\underline{\sigma},D)$$ for any $\underline{\sigma}$.
Since this transformation leaves invariant all the `thermodynamic' variables 
$(K,\meanrho,D,\meanchem,\Theta,j,\meanE)$, it leaves \eqref{eq:EKW14th}  invariant just because it does so for the reduced system
$$\left\{\begin{array}{l}
\partial_T K \,+\,\partial_X (\sigma K)\,=\,0\,,\\ [5pt]
\partial_T\meanrho +\partial_X (\meanrho \sigma)\,=\,0\,,\\ [5pt]
\partial_T  \sigma + \partial_X\left(\dfrac{1}{2} \sigma^2\,+\,\dfrac{j+D}{\meanrho}\,\sigma\right)\,=\,0\,, \\ [10pt]
\partial_T(\meanrho \sigma) \,+\,\partial_X\left(\meanrho \sigma^2 + 2 j \sigma\right) \,=\,0\,.\end{array}\right.$$


\paragraph{Acknowledgement.} This work has been partly supported by
the European Research Council \href{http://math.univ-lyon1.fr/~filbet/nusikimo/nusikimo.htm}{ERC Starting Grant 2009, project
239983- NuSiKiMo}.

\newpage 

\bibliographystyle{plain}
\addcontentsline{toc}{section}{References}
\bibliography{BNR}

\end{document}